%% file: bkhthesis.tex
\documentclass[11pt]{ucthesis}
\def\dsp{\def\baselinestretch{2.0}\large\normalsize}
\dsp

\usepackage{amssymb,amsmath,amsthm, amscd, verbatim}\usepackage[bookmarks, colorlinks=true]{hyperref}
    \newtheorem{thm}{Theorem}[section]
    \newtheorem{cor}[thm]{Corollary}
    \newtheorem{lem}[thm]{Lemma}
    \newtheorem{pro}[thm]{Proposition}
    \newtheorem{nota}[thm]{Notation}

    \theoremstyle{definition}
    \newtheorem{df}[thm]{Definition}
    \theoremstyle{remark}
    \newtheorem{rem}[thm]{Remark}
    \numberwithin{equation}{chapter}
    \newcommand{\la}{\langle}\newcommand{\ra}{\rangle}
    \newcommand{\Iff}{\Leftrightarrow}\renewcommand{\iff}{\leftrightarrow}
    \newcommand{\Implies}{\Rightarrow}
    
    \newcommand{\mc}{\mathcal}\newcommand{\mb}{\mathbf}
    \newcommand{\mf}{\mathfrak}

    \newcommand{\union}{\cup}
    \newcommand{\intersect}{\cap}
    \newcommand{\inter}{\cap}
    \newcommand{\Union}{\bigcup}

    \renewcommand{\to}{\rightarrow}
    \newcommand{\from}{\leftarrow}
    \newcommand{\restrict}{\upharpoonright}
    \newcommand{\bu}{\begin{bundle}}\newcommand{\ub}{\end{bundle}}

    \renewcommand{\and}{\,\&\,}

    \renewcommand{\|}{\,|\,}

\newsavebox{\savepar}

\begin{document}

    \title{Lattice Initial Segments of the Turing Degrees}
    \author{Bjorn Kjos-Hanssen}
    \degreeyear{Fall 2002} \degree{Doctor of Philosophy} \chair{Theodore A. Slaman} \othermembers{Leo A. Harrington\\
    Robin C. Hartshorne}
    \numberofmembers{3} \prevdegrees{Candidatus scientiarum (University
    of Oslo) 1997} \field{Logic and the Methodology of Science} \campus{Berkeley}
    
    
    
    \copyrightpage
    \begin{abstract}
    We characterize the isomorphism types of principal ideals of the Turing degrees below $\mb 0'$ that
    are lattices as the lattices with a $\Sigma^0_3$ presentation, by showing that each
    $\Sigma^0_3$-presentable bounded upper semilattice is isomorphic to such a principal ideal.
    We get a similar result for the Turing degrees below any degree above $\mb 0''$.
    \abstractsignature
    \end{abstract}

\begin{frontmatter}

\tableofcontents 
\begin{acknowledgements}
First of all I would like to thank my adviser Ted Slaman, who was supportive and enthusiastic
throughout my period of dissertation research, provided advice and ideas and presented me with good
problems to work on.

On the mathematical side, Pavel Pudl\'ak and Ralph McKenzie were of great help in establishing the
lattice-theoretic meaning of the notion of homogeneity used in initial segment constructions.
Manuel Lerman explained his construction to me in person, helping me understand the
recursion-theoretical part of this thesis. Leo Harrington, Richard Shore and Barry Cooper aided me
in mathematical and career-related ways. Shamil Ishmukhametov raised an interesting question
concerning initial segment constructions.

Among the staff of the Department of Mathematics I would like especially to thank Catalina Cordoba
and Barbara Peavy for their help and support. Professors Charles Silver, Robin Hartshorne, John
Steel, Paolo Mancosu and Charles Chihara dutifully served on my qualifying exam and dissertation
committees. In Norway, 1994-1997, Andrew J. I. Jones, St\aa l Aanderaa and Dag Normann prepared me
for graduate school.

Last but not least, my friends Bonnie Sakura Huggins, Assaf Peretz, Katia Hayati, Gizem Karaali,
Michael Anshelevich and many others helped make my experience at Berkeley a pleasant one, and my
parents Nina and Odd Kjos-Hanssen provided abundant moral and financial support.

\end{acknowledgements}
\end{frontmatter}

\chapter{Introduction}

\section{Intuitive presentation of results}
The Turing degree of unsolvability of a set of integers $A$, as introduced by Kleene and Post
\cite{Kleene.Post:54}, consists of those sets of integers $B$ that are, intuitively, just as
noncomputable as $A$, i.e. if we had an oracle that could answer questions about which integers are
in $A$ then we would be able to find out which integers are in $B$ as well, and vice versa. These
degrees are ordered by saying that the degree of $A_1$ is less than the degree of $A_2$ if $A_2$ is
more noncomputable than $A_1$. It turns out that the set $A\oplus B=\{2x\|x\in
A\}\union\{2x+1\|x\in B\}$ represents the least upper bound of $A$ and $B$ in this ordering, but it
is not so clear how to get greatest lower bounds, and indeed Spector \cite{Spector:56} showed that
greatest lower bounds in the Turing degrees do not always exist. A result of Jockusch and Posner
\cite{Jockusch.Posner:78} gives a restriction on which initial segments of the degrees can satisfy
the condition that greatest lower bounds always exists. In this thesis we prove that in a sense
there are no other restrictions than the one found by Jockusch and Posner.

\section{Technical presentation of results}
This section may be skipped by readers unfamiliar with initial segment constructions as in
\cite{Lerman:83}.

By Jockusch and Posner \cite{Jockusch.Posner:78}, if a principal ideal in the Turing degrees $[\mb
0,\mb a]$ is a lattice then $\mb a''=(\mb a\vee\mb 0')'$ and hence if $\mb a\le\mb 0'$, we have
$\mb a''=\mb 0''$. Consequently $[\mb 0,\mb a]$ is $\Sigma^0_3$-presentable. On the other hand, we
shall see in Theorem \ref{410} that every $\Sigma^0_3$-presentable bounded upper semilattice is
isomorphic to an ideal $[\mb 0,\mb a]$ with $\mb a\le\mb 0'$. Hence we have the result promised in
the abstract: a bounded lattice $L$ is isomorphic to $[\mb 0,\mb a]$ for some $\mb a\le\mb 0'$ if
and only if $L$ is $\Sigma^0_3$-presentable. The forcing construction of Lachlan and Lebeuf
\cite{Lachlan.Lebeuf:76} is known to produce, given a $\Delta^0_3$-presentable upper semilattice
$L$, a degree $\mb a$ such that $\mb a''=\mb 0''$ and $[\mb 0,\mb a]$ is isomorphic to $L$ (see for
example \cite{Lerman:83}). Their construction combined with our method of Pudl\'ak tables gives the
corresponding result for $\Sigma^0_3$ instead of $\Delta^0_3$. In fact the Lachlan-Lebeuf
construction produces hyperimmune-free degrees, hence degrees $\mb a$ satisfying $\mb a''=(\mb
a\vee\mb 0')'$ (i.e. $\mb a\in\text{GL}_2$). Now using a finite injury argument as in
\cite{Lerman:83} Exercise VIII.1.16 it follows that given any degree $\mb a\ge\mb 0''$, a lattice
$L$ is $\Sigma^0_2(\mb a)$-presentable if and only if $L$ is isomorphic to $[\mb 0,\mb g]$ for some
$\mb g\in\text{GL}_2$ with $\mb g\le\mb a$. Moreover using $e$-total trees one obtains the fact
that a lattice $L$ is $\Sigma^0_1(\mb a)$-presentable if and only if $L$ is isomorphic to $[\mb
0,\mb g]$ for some $\mb g$ satisfying $\mb g''\le\mb a$. From this it follows that any
jump-preserving automorphism of the Turing degrees is equal to the identity on the degrees above
$\mb 0''$ (a result which is well-known) since the isomorphism type of $[\mb 0,\mb a]$ determines
$\mb a$ if $\mb a\ge\mb 0''$.

We now begin the mathematical development that will culminate in the proof of Theorem \ref{410}.

\section{Conventions} We assume that the reader is familiar with the basic results and notation
of mathematical logic in general and classical recursion theory in particular, and with standard
set theoretic notation for sequences, functions and natural numbers. If $A\subseteq\omega$ and
$g:\omega\to\omega$ then their Turing degrees are denoted by lower-case boldface letters $\mb a$,
$\mb g$, respectively. The length of a finite string $\sigma\in\omega^{<\omega}$ is denoted by
$|\sigma|$, and the concatenation of $\sigma$ and $\tau$ by $\sigma*\tau$.

We will sometimes speak of the ``elements'' of a structure; depending on context, these may be the
elements of the underlying set of the structure, or the functions or relations in the structure.
Given a set or structure $X$, $|X|$ may denote the cardinality or underlying set of $X$. All
countable lattices and semilattices considered are assumed to be bounded, i.e. have a least element
0 and a greatest element 1. Composition of maps may be denoted by juxtaposition or, for emphasis,
by the symbol $\circ$. Given a predicate $F$, we use $\mu x(F(x))$ to denote the least $x<\omega$
such that $F(x)$.

As usual, $\{e\}^A$ or $\{e\}(A)$ refers to the function computed by the $e^{\text{th}}$ oracle
Turing machine, in some standard enumeration of oracle Turing machines, when run with oracle $A$.
The value of of this function at $x<\omega$ is denoted $\{e\}^A(x)$ or $\{e\}(A;x)$, and if such a
value (does not) exist, we write ($\{e\}^A(x)\uparrow$) $\{e\}^A(x)\downarrow$. For
$\sigma\in\omega^{<\omega}$, $\{e\}^\sigma(x)\downarrow$ means that the $e^\text{th}$ oracle Turing
machine, when run with any oracle extending $\sigma$, halts after making no oracle queries beyond
$|\sigma|$. Then $\{e\}_s^\sigma(x)\downarrow$ ($\{e\}^A_s(x)\downarrow$) means that the
computation $\{e\}^\sigma(x)$ ($\{e\}^A(x)$) halts by stage $s$.

Sets will be determined by listing their elements as $\{a_0,a_1,\ldots\}$ or by specification as
the set of all $x$ satisfying property $P$, denoted by $\{x\|P(x)\}$. If $A$ and $B$ are sets, then
we write $x\in A$ for \emph{$x$ is an element of $A$} and $A\subseteq B$ for \emph{$A$ is a subset
of $B$}. We use $A\subset B$ to denote $A\subseteq B$ but $A\ne B$ (placing $\not\text{ }$ through
a relation symbol denotes that the relation fails to hold for the specified elements). $A\union B$
($A\inter B$) ($A-B$) ($A\times B$) is the union (intersection) (difference) (cartesian product) of
$A$ and $B$, $\emptyset$ is the empty set, etc. A partial function $\varphi$ from $A$ to $B$
(written $\varphi:A\to B$) is considered as a subset of the set of ordered pairs $\{\la
x,y\ra\|x\in A\and y\in B\}$, and we write $\varphi(x)\downarrow$ ($\varphi(x)$ \emph{converges})
for $(\exists y)(\la x,y\ra\in\varphi)$, and $\varphi(x)\uparrow$ ($\varphi(x)$ \emph{diverges})
otherwise. We will sometimes denote the function $\varphi$ with the notation $x\mapsto\varphi(x)$.
The \emph{domain} of $\varphi$ is denoted by dom$(\varphi)=\{x\|\varphi(x)\downarrow\}$, and the
\emph{range} of $\varphi$ is denoted by ran$(\varphi)=\{y\|(\exists x)(\varphi(x)=y)\}$. If
$C\subseteq A$ then $\varphi\restrict C$ denotes the \emph{restriction} of $\varphi$ to $C$, i.e.
$\varphi\inter (C\times B)$. We sometimes identify sets with their characteristic functions. The
set $\{0,1,2,\ldots\}$ is denoted by $\omega$, and if $f:\omega\to\omega$ is a total function then
we write $\lim_s f(s)=y$ if $\{s\|f(s)\ne y\}$ is finite, and $\lim_s f(s)=\infty$ if for each
$y<\omega$, $\{s\|f(s)=y\}$ is finite. ``Iff'' is an abbreviation of ``if, and only if''.

\chapter{Lattice representations} \label{c:tables}

\input{PRESEN-1}

\chapter{Lerman constructions}
    \input{PRESEN-21}

    \section{Properties of Lerman constructions}
    \input{PRESEN-22}

\chapter{$\Sigma^0_3$-presentable initial segments of $\mc D(\le\mb 0')$}
    \input{PRESEN-3}

\bibliographystyle{amsplain}
\bibliography{computability}
\end{document}

%% file: PRESEN-1.tex
\subsection{Representations as congruence lattices of algebras}

\begin{df}
	If $\la L,\le\ra$ is a \emph{partial order} (a transitive, reflexive, antisymmetric relation) such that greatest lower bounds $a\wedge b$ and least upper bounds $a\vee b$ of all pairs $a,b\in L$ exist, then $\la L,\le,\vee,\wedge\ra$ is called a \emph{lattice}.

	A \emph{unary algebra} is a set $X$ together with a collection of functions $f:X\rightarrow X$. If $A$ is a unary algebra then $\Con A$ denotes the \emph{congruence lattice} of $A$, i.e., the lattice of all equivalence relations $E$ on $X$ preserved by all $f\in A$, i.e., satisfying $(\forall x,y\in X)(xEy\rightarrow f(x)E f(y))$, ordered by inclusion.

	The \emph{partition lattice} $\Part(X)$ on a set $X$ consists of all equivalence relations (considered as sets of ordered pairs) on $X$. A \emph{lattice table} $\Theta$ is a $\la 0,1,\vee,\wedge\ra$-sublattice of a partition lattice, where $A\le B\iff A\supseteq B$ and $A\vee B=A\inter B$. (Here 0 (1) is the least (greatest) element of the lattice.) (I.e. $\Theta$ and $\Part(X)$ have the same interpretations of the symbols 0,1, $\vee$, $\wedge$.) If we assume $X\subseteq\omega$ and are given a fixed isomorphism between $\Theta\subseteq$ $\Part(X)$ and a lattice $L$, taking $k\in L$ to $\sim_k\in\Part(X)$, then for $a\in X$ and $k\in L$, we define $a^{[k]}=\mu b(a\sim_k b)$. The map $\la a,k\ra\mapsto a^{[k]}$ gives a visualization of the lattice table $\Theta$ as a matrix, or table; and conversely this matrix determines $\Theta$, since $a\sim_k b\iff a^{[k]}=b^{[k]}$. For example, a lattice table for the 2-element lattice $L=\{0,1\}$ with $0<1$ can be obtained using $X=\{0,1\}$ and letting $\Theta=\Part(X)$. (In general, though, $\Theta$ and $L$ do not have the same number of elements and $\Theta$ is a proper sublattice of $\Part(X)$.) In this case there is only one possible isomorphism. In fact our definitions imply that it is always the case that $\sim_1$ is the identity relation on $X$ and $\sim_0$ is the relation $X\times X$. The reader can check that the table in our example is given by 
\[
	0^{[0]}=0,\quad 0^{[1]}=0,\quad 1^{[1]}=1,\quad\text{and}\quad 1^{[0]}=0.
\]

	If $\Theta$ is a lattice table then $\End \Theta$ denotes the set of all \emph{endomorphisms} of $\Theta$, i.e., maps from $X$ to $X$ preserving all equivalence relations in $\Theta$ (that is, $(\forall a,b\in X)(\forall k\in L)(a\sim_k b\to f(a)\sim_k f(b))$), considered as a unary algebra.

	$C_\Theta(a,b)$ denotes the \emph{principal equivalence relation in $\Theta$ generated by} $\la a,b\ra$, i.e., the intersection of all equivalence relations in $\Theta$ containing $\la a,b\ra $. In particular $C(a,b)$ denotes $C_\Theta(a,b)$ for $\Theta=\Part(X)$. So we always have $C(a,b)\subseteq C_\Theta(a,b)$. We define $\End_\Theta(a,b)$ to be the the \emph{principal congruence relation on $\Theta$ generated by} $\la a,b\ra$, i.e., the equivalence relation generated by all pairs $\la f(a),f(b)\ra$ for $f\in\End\Theta$.

	Note that $\End_\Theta(a,b)\subseteq C_\Theta(a,b)$. Indeed if $\la c,d\ra\in\End_\Theta(a,b)$ then $\la c,d\ra$ is in the transitive closure of $\{ \la f(a),f(b)\ra\mid f\in\End\Theta\}$, so it suffices to show each such $\la f(a),f(b)\ra\in C_{\Theta}(a,b)$. For this it suffices to show $\la f(a),f(b)\ra$ is in $\sim_k$ provided that $a\sim_k b$; this holds since $f\in\End\Theta$.
\end{df}

\begin{df}
	Let $\Theta$ be a finite lattice table. We say that $\Theta$ is \emph{Malcev homogeneous} if for all $a,b\in \Theta$, $C_\Theta(a,b)\subseteq \End_\Theta(a,b)$.
\end{df}

In other words, suppose $a,b,c,d\in\Theta$ satisfy $(\forall k)(a\sim_k b\rightarrow c\sim_k d)$. Then Malcev homogeneity says that 
\[
(\exists n<\omega)(\exists z_1,\ldots,z_n\in\Theta)(\exists
f_0,\ldots,f_n\in\End\Theta)
\]
\[
(\forall i\le n)(\{ f_i(a), f_i(b) \} = \{ z_i, z_{i+1}\}))
\]
where $z_0=c$ and $z_{n+1}=d$. We refer to the $z_i$ as \emph{homogeneity interpolants}. This notion of homogeneity is more general (weaker) than those considered by Lerman \cite{Lerman:83}. The following observation can be traced back to Malcev (see \cite{Malcev:54}, \cite{Malcev:63}, \cite{Gratzer:68}).

\begin{pro}\label{p:hoho}
	For any unary algebra $A$, $\Con A$ is a Malcev homogeneous lattice table.
\end{pro}
\begin{proof}
	$\Theta=\Con A$ is a sublattice of $\Part(A)$, so it is a lattice table.

	$A\subseteq\End\Theta$ since if $f$ is an operation in $A$ and $\sim\in\Theta$ then $\sim$ is a congruence relation on $A$ and hence $\forall x,y(x\sim y\to f(x)\sim f(y))$, which means that
$f\in\End\Theta$. Clearly for any unary algebras $A$, $B$ on the same underlying set, we have $A\subseteq B\Implies \Con A\supseteq\Con B$. Hence $\Con\End\Theta\subseteq\Con A=\Theta$.

	If $f\in\End\Theta$ and $\la x,y\ra\in\End_\Theta(a,b)$ then there exist $z_1,\ldots,z_k$ such that $\la z_i,z_{i+1}\ra=\la g_i(a),g_i(b)\ra$ for $g_i\in\End\Theta$ with $z_1=x$, $z_k=y$, hence letting $w_i=f(z_i)$ and $h_i=f\circ g_i$ we have $\la w_i,w_{i+1}\ra=\la h_i(a),h_i(b)\ra\in\End_\Theta(a,b)$, $w_1=f(x)$, $w_k=f(y)$, and so $\la f(x),f(y)\ra\in\End_\Theta(a,b)$. Hence we have shown $\End_\Theta(a,b)\in \Con\End\Theta\subseteq\Theta$. Since End $\Theta$ contains the identity map, $\End_\Theta(a,b)$ contains $\la a,b\ra$. Hence $\End_\Theta(a,b)$ is in $\Theta$ and contains $\la a,b\ra$, so it contains $C_\Theta(a,b)$.
\end{proof}

\subsection{Pudl\'ak graphs}
We present a modification of the construction of \cite{Pudlak:76}. Pudl\'ak observed that a slight
modification of a construction found by J\'onsson \cite{Jonsson:53} would yield stronger
properties than those noted by J\'onsson.

\begin{df}
	Let $L$ be a finite lattice. $\Theta=(A,r,h)$ is called an \emph{$L-\{1\}$-valued graph} (or, \emph{$L-\{1\}$-colored}) if $A$ is a set, $r$ is a set of size 2 subsets of $A$, i.e., $(A,r)$ is an undirected graph without loops, and $h:r\rightarrow L-\{1\}$ is a mapping of the set $r$ of the edges of the graph into $L-\{1\}$.

	$e:L\rightarrow \Part(A)$ is defined by: for $\alpha\in L$, $e(\alpha)$ is the equivalence relation on $\Theta$ (i.e., on $A$) generated by identifying points $a,b$ if there is a path from $a$ to $b$ in the graph, consisting of edges all of color (value) $\ge
\alpha$. We say that $a,b$ are connected with color $\ge\alpha$.
\end{df}

\begin{df}[$\alpha$-cells]
	An $\alpha$-\emph{cell} $\mc B_\alpha=(B_\alpha,s_\alpha,k_\alpha)$ consists of a \emph{base} edge $\{a,b\}$ colored $\alpha$, and for each inequality $\alpha_1\wedge\alpha_2\le \alpha$ that holds in $L$, where $\alpha_1,\alpha_2\in L$, a chain of edges $\{a,u_1\},\{u_1,u_2\},\{u_2,u_3\},\{u_3,b\}$, colored $\alpha_1,\alpha_2,\alpha_1,\alpha_2$, respectively, so that $a,b,u_1,u_2,u_3$ are distinct. The base edge and chain of edges corresponding to a particular inequality $\alpha_1\wedge\alpha_2\le\alpha$ is referred to as a \emph{pentagon}. So a cell consists of several pentagons, intersecting only in a common base edge.
\end{df}

\begin{df}[Pudl\'ak graphs]
	Pudl\'ak's original graph is defined as follows.
	\begin{enumerate}
		\item $\Theta^P_0$ consists of a single edge valued by $0\in L$. (In fact, how we choose to color this one edge has no impact on later proofs.)
		\item $\Theta^P_{n+1}$ is obtained by attaching to each edge of $\Theta^P_n$ of any color $\alpha$ an $\alpha$-cell.
		\item $\Theta^P=\union_{n\in\omega} \Theta^P_n$.
	\end{enumerate}
	We will use the following modification.
	\begin{enumerate}
		\item $\Theta^{(i)}_0=\Theta^P_0$, for each $i\in\omega$.
		\item $\Theta^{(i)}_j$ is obtained by attaching to each edge of $\Theta^{(i)}_{j-1}$ of any color $\alpha$, $i$ $\alpha$-cells.
		\item $\Theta_j=\Theta^{(j)}_j$.
		\item $\Theta=\union_{n\in\omega} \Theta_n$ we call the Pudl\'ak graph for $L$.
	\end{enumerate}

	Since $e(\alpha)$ for $\alpha\in L$ is determined by the colored graph structure $\Theta^P(L)$, we may consider $\Theta^P$ as the structure $\la\Theta^P,\{e(\alpha)\mid \alpha\in L\}\ra$ and
	$\Theta^P_n=\la\Theta^P_n,\{e(\alpha)\restrict\Theta^P_n\mid \alpha\in L\}\ra$.

	Similarly $\Theta=\la\Theta,\{e(\alpha)\mid \alpha\in L\}\ra=\union_{n\in\omega}\Theta_n$ where $\Theta_n=\la\Theta_n,\{e(\alpha)\restrict\Theta_n\mid \alpha\in L\}\ra$. For emphasis we may write
	$\Theta(L)$ for $\Theta$. By \emph{the Pudl\'ak algebra of $L$} we mean the algebra of unary functions $\End \Theta(L)$.
\end{df}

\begin{thm}\label{t:pud}
	Let $L$ be a finite lattice and let $L^*$ denote the dual of $L$ (that is, $L^*$ has the same underlying set as $L$ and if $a\le b$ in $L$ then $b\le a$ in $L^*$). $L$ is dual isomorphic to the congruence lattice of its Pudl\'ak algebra. In fact, $e:L^*\rightarrow\Theta(L)$ is an isomorphism, and $\Theta(L)=\Con\End\Theta(L)$.
\end{thm}
\begin{proof}
	Pudl\'ak (\cite{Pudlak:76}; see subsection \ref{s:pudlaksproof}) assumes that $L$ is an algebraic lattice \cite{Gratzer:98}, defines a certain algebra $S\subseteq\End\Theta^P(L^*)$, and
shows that $e:L\rightarrow\Theta^P(L^*)$ is an isomorphism, and $\Theta^P(L^*)=\Con S$. Now trivially $\Theta^P(L^*)\subseteq\Con\End\Theta^P(L^*)$ holds, and $S\subseteq\End\Theta^P(L^*)$ implies $\Con\End\Theta^P(L^*)\subseteq\Con S$. So we have
$\Theta^P(L^*)=\Con\End\Theta^P(L^*)$.

	In fact Pudl\'ak's proof shows that $e:L\rightarrow\Theta(L^*)$ is an isomorphism, and $\Theta(L^*)=\Con\End\Theta(L^*)$. (This is an important step; for details see Section
\ref{s:pudlaksproof}.)

	Now let $L$ be a finite lattice. Then $L^*$ is also a finite lattice, and since every finite lattice is algebraic, $L^*$ is an algebraic lattice. Hence
$e:L^*\rightarrow\Theta(L^{**})=\Theta(L)$ is an isomorphism and $\Theta(L)=\Con\End\Theta(L)$.
\end{proof}

We now prove some basic properties of Pudl\'ak graphs and the associated tables.

\begin{df}
	$\Theta=\la\Theta_n\mid n<\omega\ra$ is a \emph{sequential Malcev homogeneous lattice table} if (1) and (2) hold:
	\begin{enumerate}
		\item[(1)] 
		$\Theta$ is a \emph{sequential lattice table}, i.e., each $\Theta_n$ is a $\la 0,1,\vee\ra$-substructure of $\Part(|\Theta_n|)$ ($\Theta_n$ is an \emph{usl table}), $\Theta$ is a lattice table, and meet interpolants for elements of $\Theta_n$ exist in $\Theta_{n+1}$. (If $i\wedge j=k$ in $L$ then $\sim_i$ and $\sim_j$ generate $\sim_k$
		in $\Theta(L)$. So if $a\sim_k b$ then there exist \emph{meet interpolants} $z_1,\ldots,z_n$ such that $a\sim_i z_1\sim_j z_2\sim_i\cdots\sim_j z_n\sim_i b$.)
		\item[(2)] 
		$\Theta$ is Malcev homogeneous, with homogeneity interpolants for elements of $\Theta_n$ appearing in $\Theta_{n+1}$ (compare \cite{Lerman:83}*{VII.1.1, 1.3}).
	\end{enumerate}
\end{df}

\begin{lem}\label{l:basic}
	$\Theta=\la\Theta_n(L)\mid n<\omega\ra$ has a subsequence which is a recursive Malcev homogeneous sequential lattice table.
\end{lem}

\begin{proof}
	Since $\Theta$ is a congruence lattice, $\Theta$ is a Malcev homogeneous lattice table. Hence $\Theta$ can be given a ramification as a sequential Malcev homogeneous lattice table by taking a
subsequence of $\la\Theta_n\mid n<\omega\ra$.

	The sequence is recursive since to compute an equivalence relation on elements of $\Theta_n$, it is sufficient to consider paths in $\Theta_n$, since any path wandering outside of $\Theta_n$ can be
contracted to one that does not which has equal or better color. (Here by definition a path is contracted by performing a series of operations of two kinds: either replacing subpaths
$a,c_1,\ldots,c_s,a$ by just $a$, or replacing a path $a,b,c,d,e$ around a pentagon by an edge $a,e$ cutting across.)
\end{proof}

	In the following we use $\la\Theta_n(L)\mid n<\omega\ra$ to denote the subsequence taken in Lemma \ref{l:basic}. (It can actually be shown that it is not necessary to take a subsequence, and even
better, each $\Theta_n$ is already closed under homogeneity interpolants. However we do not need this fact.)

\subsection{Pudl\'ak's proof}\label{s:pudlaksproof}

	This subsection contains the proof due to Pudl\'ak referred to in Theorem \ref{t:pud}. Let $K=L-\{1\}$. A triple $\Theta=\la \Theta,r,h\ra$ is called a $K$-\emph{valued graph} if $\la\Theta,r\ra$
is an undirected graph without loops and $h:r\to K$ is a mapping of the set of its edges into $K$.

	If $\Theta_i=\la\Theta_i,r_i,h_i\ra$, $i=0,1$, are two $K$-valued graphs, a mapping $f:\Theta_0\to\Theta_1$ is called \emph{stable} if for every $\{a,b\}\in r_0$, either $f(a)=f(b)$,
or else $\{f(a),f(b)\}\in r_1$ and $h_1(\{f(a),f(b)\})=h_0(\{a,b\})$.

	Let Stab$(\Theta)$ denote the monoid of all stable mappings of $\Theta$ into itself --- the \emph{stabilizer} of $\Theta$.

\begin{df}
	Let $F$ be a set of mappings of a set $X$ into itself, let $\{a,b\}$, $\{c,d\}$ be subsets of $X$. Then \emph{$\{a,b\}$ dominates $\{c,d\}$ through $F$} if there are $c=u_0,u_1,\ldots,u_n=d$ in $X$
and $f_1,\ldots,f_n\in F$ such that $\{f_i(a),f_i(b)\}=\{u_{i-1},u_i\}$ for $i=1,\ldots,n$; in notation, $\{a,b\}D_F\{c,d\}$.
\end{df}

If $F$ is a monoid, then $D_F$ is a transitive and reflexive relation.

\begin{nota}
	For a given set $X$ and a function $f$ defined over the set $X$, denote by $2^X$ the set of all subsets of $X$, by $\Part(X)$ the set of all equivalences on $X$, by $f^+$ the function defined over
the set $2^X$ by the formula $f^+(Y)=\{f(x)\mid x\in Y\}$ for $Y\subseteq X$.
\end{nota}

\begin{df}
	Let $\Theta=\la\Theta,r,h\ra$ be a $K$-valued graph. From now on we consider $L$ to be the set of all filters of $K$ ordered by reverse inclusion (including $\emptyset$, which corresponds to 1).
This does not change $L$ up to isomorphism since $\{a\mid a\ge b\}\supseteq\{a\mid a\ge c\}\Iff b\le c$. Let $\mc E$ denote the operator of the equivalence closure on $\Theta$. Define two mappings
$\varphi:L\to\Part(\Theta)$, $\psi:\Con(\Theta,\Stab(\Theta))\to 2^K$ by the formulas $\varphi(I)=\mc Eh^{-1}(I)$ for $I\in L$, $\psi(E)=h^+(r\inter E)$, for
$E\in\Con(\Theta,\Stab(\Theta))$.
\end{df}

In fact, $\varphi:L\to\Con(\Theta,\Stab(\Theta))$, since every operation of the latter
algebra is a stable mapping.

\begin{thm}
	Let $\Theta=\la\Theta,r,h\ra$ be a $K$-valued graph, $F\subseteq\Stab(\Theta)$ and let us assume that
	\begin{enumerate}
		\item \label{Pudlak(1)} $h$ maps $r$ onto $K$,
		\item \label{Pudlak(2)} if $x\in r$ and $\alpha_1\wedge\alpha_2\le h(x)$, $\alpha_1,\alpha_2\in K$,
		then there is a cycle $x,x_1,\ldots,x_n$ of $r$ such that $h(x_i)\in\{\alpha_1,\alpha_2\}$, for
		$i=1,\ldots,n$,
		\item \label{Pudlak(3)} if $x,x_1,\ldots,x_n$ is a cycle of $r$, then $\bigwedge_{i=1}^n h(x_i)\le
		h(x)$,
		\item \label{Pudlak(4)} if $x,y\in r$ are such that $h(x)=h(y)$, then there is an $f\in F$ such that
		$f^+(x)=y$,
		\item \label{Pudlak(5)} if $c\ne d$ are in $\Theta$, then there is a path $x_1,\ldots,x_n\in r$
		connecting $c$ to $d$ such that $\{c,d\}D_F x_i$ for all $i=1,\ldots,n$.
	\end{enumerate}

	Then $\psi(E)\in L$ for every congruence $E$ of $\la\Theta,F\ra$ and $\psi=\varphi^{-1}$.
\end{thm}
\begin{proof}
Let $\alpha_1,\alpha_2\in\psi(E)$, $\alpha_1\wedge\alpha_2\le\alpha$. Since $h$ is onto,
$\alpha=h(x)$ for some $x\in r$. By \ref{Pudlak(2)} there is a cycle $x,x_1\ldots,x_n$ of $r$ such
that $h(x_i)\in\{\alpha_1,\alpha_2\}$, $i=1,\ldots,n$, $\alpha_1=h(y_1)$, $\alpha_2=h(y_2)$ for
some $y_1,y_2\in r\inter E$, so that, by \ref{Pudlak(4)}, every $x_i\in E$. Since
$x,x_1,\ldots,x_n$ is a cycle, $x\in E$. Therefore $h(x)=\alpha\in\psi(E)$. Thus $\psi(E)\in L$.

$\psi\varphi=1$: Let $I$ be a filter (dual ideal) of $K$. Since $h$ is onto,
$\psi\varphi(I)=h^+(r\inter\mc Eh^{-1}(I))\supseteq h^+h^{-1}(I)=I$. Conversely, if $x\in
r\inter\mc Eh^{-1}(I)$, then by \ref{Pudlak(3)}, $\bigwedge_{i=1}^n h(x_i)\le h(x)=\alpha$, where
$x_1,\ldots,x_n\in h^{-1}(I)$ connect the endpoints of $x$. Since $I$ is a filter, $\alpha\in I$,
which proves the opposite inclusion.

$\varphi\psi=1$: Let $E\in\Con(\Theta,F)$, then $\varphi\psi(E)=\mc Eh^{-1}(r\inter
E)\supseteq\mc E(r\inter E)$. Let $\{a,b\}\in\mc Eh^{-1}h^+(r\inter E)$, then there is a path
$x_1,\ldots,x_k\in r$ connecting $a$ to $b$ such that $h(x_i)=h(y_i)$ for some $y_i\in r\inter E$.
\ref{Pudlak(4)} yields $x_i\in E$ so that $\{a,b\}\in\mc E(r\inter E)$. Since $E$ is an
equivalence, $E\supseteq\mc E(r\inter E)$. Now, let $\{c,d\}\in E$. \ref{Pudlak(5)} implies the
existence of a path $z_1,\ldots,z_n$ of elements of $r$ connecting $c$ to $d$ such that $\{c,d\}D_F
z_i$, $i=1,\ldots,n$. Since $E$ is a congruence of $\la\Theta,F\ra$, all $z_i$ belong to $E$, so
that $\{c,d\}\in\mc E(r\inter E)$.
\end{proof}

Since $\varphi$ and $\psi$ are order-preserving mappings, $L$ is isomorphic to the congruence
lattice of $\la\Theta,F\ra$. The remainder of the proof deals with the construction of a $K$-valued
graph $\la\Theta,r,h\ra$ which, together with its stabilizer $F=$Stab$(\Theta)$, satisfies
\ref{Pudlak(1)}-\ref{Pudlak(5)} above.

\begin{df}
Let us construct a $K$-valued graph $B_\alpha=\la B_\alpha,s_\alpha,k_\alpha\ra$ for every
$\alpha\in K$, ($B_\alpha$ will be called the $\alpha$-\emph{cell}) as follows: set $\alpha
C=\{\la\alpha_1,\alpha_2\ra\mid \alpha_1\wedge\alpha_2\le\alpha\}$. For every $\mb
a=\la\alpha_1,\alpha_2\ra\in\alpha C$ let 
\[
	\tag{+} x^{\mb a}_i=\{u_{i-1}^{\mb a},u_i^{\mb a}\},\quad i=1,2,3,4, 
\]
where $u_0^{\mb a}=a$, $u_4^{\mb a}=b$, for all $\mb a\in\alpha C$, and all the other
points are distinct.

Set 
\[
	B_\alpha=\{a,b\}\union\{u^{\mb a}_i\mid i=1,2,3, \,\mb a\in\alpha C\},
\]
\[
	s_\alpha=\{\{a,b\}\}\union\{x_i^{\mb a}\mid i=1,2,3,4,\,\mb a\in\alpha C\},
\]
\[
	k_\alpha(\{a,b\})=\alpha, \quad k_\alpha(x_1^{\mb a})=k_\alpha(x_3^{\mb a})=\alpha_1, \quad
	k_\alpha(x_2^{\mb a})=k_\alpha(x_4^{\mb a})=\alpha_2.
\]
$\{a,b\}$ will be called the \emph{base} of the $\alpha$-cell $B_\alpha$. (+) is called a
\emph{chain} of $B_\alpha$. In the sequel the superscripts are omitted.

The required $K$-valued graph $\Theta=\la\Theta,r,h\ra$ is constructed by induction as follows.
Choose an element $\alpha_0\in K$. $\Theta_0$ will be a single edge valued by $\alpha_0$. If
$\Theta_n=\la\Theta_n,r_n,h_n\ra$ has already been constructed, produce $\Theta_{n+1}$ as follows:
Let $x\in r_n-r_{n-1}$, $x=\{c,d\}$, $h_n(x)=\alpha$. Attach a copy of $B_\alpha$ to $\Theta_n$,
identifying $\{c,d\}$ with the base $\{a,b\}$ of $B_\alpha$, arbitrarily by the definition $a=c$,
$b=d$ or $a=d$, $b=c$. In other words, the extended graph has the set of vertices
$\Theta_n\union(B_\alpha-\{a,b\})$, and $r_n\union(s_\alpha-\{\{a,b\}\})$ are all its edges. Since
$h_n(x)=\alpha=k_\alpha(\{a,b\})$, the common extension of $h_n$ and $k_\alpha$ is well-defined.
Applying this one step extension to every $x\in r_n-r_{n-1}$ we arrive at
$\Theta_{n+1}=\la\Theta_{n+1},r_{n+1},h_{n+1}\ra$.

Note that $\Theta_n\subseteq\Theta_{n+1}$, $r_n\subseteq r_{n+1}$, $h_n\subseteq h_{n+1}$, so that
we may define $\Theta=\la\Theta,r,h\ra$ by $\Theta=\Union_{n=0}^\infty$, $r=\Union_{n=0}^\infty
r_n$, $h=\Union_{n=0}^\infty h_n$, and that $\Theta$ is indeed a $K$-valued graph.
\end{df}

\begin{lem}
$\Theta$ satisfies \ref{Pudlak(1)}, \ref{Pudlak(2)}, \ref{Pudlak(3)}.
\end{lem}
\begin{proof}
\ref{Pudlak(1)}: Since $\alpha_0\wedge\alpha\le\alpha_0$ for all $\alpha\in K$, $h_1$ is already an
onto mapping.

\ref{Pudlak(2)}: Every edge $x\in r$ is the base of a copy of an $h(x)$-cell. Thus if
$\alpha_1\wedge\alpha_2\le h(x)$, then the cycle $x,x_1,\ldots,x_4$ valued by
$\alpha,\alpha_1,\alpha_2,\alpha_1,\alpha_2$ has been included in $B_{h(x)}$.

\ref{Pudlak(3)} is true for $\Theta_0$, since it contains no cycles. If proved for all cycles in
$\Theta_m$, assume that $x,x_1,\ldots,x_n$ is a cycle in $\Theta_{m+1}$. If $x\in r_{m+1}-r_m$,
then $x$ lies on one of the chains of some $B_\alpha$ added at the $(m+1)$-st step and every cycle
containing $x$ has to contain this whole chain --- which contains another edge $y$ with $h(x)=h(y)$.
Thus $h(x)=h(x_k)$ for some $k$ and $\bigwedge_{i=1}^n h(x_i)\le h(x)$. If $x\in r_m$ then by
replacing every portion of the cycle $x,x_1,\ldots,x_n$ not contained in $r_m$ by the base of its
cell we obtain a cycle $x,y_1,\ldots,y_k$ in $r_m$ satisfying $\bigwedge h(x_i)\le\bigwedge
h(y_i)\le h(x)$, by the induction hypothesis and by an easy observation that the value of the base
of an arbitrary $\alpha$-cell is greater than or equal to the meet of values of the edges of any of its chains.
\end{proof}

The conditions \ref{Pudlak(4)} and \ref{Pudlak(5)} involve the stabilizer of $\Theta$ --- they say
that ``there are enough stable mappings''. To prove these two conditions, we will define certain
stable mappings explicitly.

First of all, note that there is a stable involution of every cell $B_\alpha$ onto itself which
interchanges the endpoints of the base of $B_\alpha$ ($\la\alpha_1,\alpha_2\ra\in\alpha C$ iff
$\la\alpha_2,\alpha_1\ra\in\alpha C$). This gives rise to the following:

\begin{lem}\label{PudlakClaim2}
Every stable mapping $f:\Theta_n\to\Theta_m$ can be extended to a stable mapping
$f':\Theta_{n+1}\to\Theta_{m+1}$ and, consequently, to a stable mapping $f'':\Theta\to\Theta$.
\end{lem}
\begin{proof}
If $t\in\Theta_{n+1}-\Theta_n$, denote by $\{a,b\}$ the base of a cell $B$ containing $t$. If
$f(a)=f(b)$, set $f'(t)=f(a)$. If $c=f(a)\ne f(b)=d$, then $h_n(\{a,b\})=h_m(\{c,d\})$ since $f$ is
stable. Therefore $\{c,d\}$ is the base of another copy of the cell $B$. $f'$ is the extension of
$f$ by the bijection of $B$ which is determined by $f\restrict\{a,b\}$. In both cases it is clear
that $f'$ is again stable.
\end{proof}

\begin{df}
Let $x\in r_m-r_{m-1}$, $m>0$. Define $f_x:\Theta_m\to\Theta_m$ as follows: $x=x_k$ in some chain
$x_1,x_2,x_3,x_4$ of a cell based on $\{a,b\}\in r_{m-1}$, $a=u_0$, $b=u_4$, $x_i=\{u_{i-1},u_i\}$,
$h_m(x_j)=h_m(x_{j+2})=\alpha_j$ for $j=1,2$. If $k=1,2$, set $f_{x_k}(u_k)=f_{x_k}(u_{k+1})=u_k$,
$f_{x_k}(y)=u_{k-1}$ for all other $y\in\Theta_m$, if $k=3,4$ set
$f_{x_k}(u_k)=f_{x_k}(u_{k-1})=u_{k-1}$, $f_{x_k}(y)=u_k$ for all other $y\in\Theta_m$.

For $x\in r_0$, let $f_x$ be the identity mapping on $\Theta_0$. In particular,
$f^+_{x_k}(\Theta_m)=x_k$ and all $f_{x_k}$ are stable.
\end{df}

\begin{lem}
$\Theta$ satisfies \ref{Pudlak(4)}.
\end{lem}
\begin{proof}
Let $x\in r_m-r_{m-1}$, $y\in r_n$, $h(x)=h(y)$. Let $f_x:\Theta_m\to\Theta_m$ be the stable
mapping defined above. Since $f^+_x(\Theta_m)=x$, the composite mapping $g\circ f_x$ --- where
$g:x\to y$ is a bijection --- is also a stable mapping of $\Theta_m$ to $\Theta_n$. By Lemma
\ref{PudlakClaim2}, $g\circ f_x$ can be extended to a stable $f:\Theta\to\Theta$.
\end{proof}

\begin{lem}
$\Theta$ satisfies \ref{Pudlak(5)}.
\end{lem}
\begin{proof}
By induction: Since the identity mapping of $\Theta$ is in $F$, all the edges of $r$ satisfy
\ref{Pudlak(5)}, in particular $\Theta_0$ satisfies \ref{Pudlak(5)}.

The induction step from $\Theta_m$ to $\Theta_{m+1}$: Let $c,c'\in\Theta_{m+1}$, $\{c,c'\}\not\in
r$, $c\ne c'$. We shall construct a path $x_1,\ldots,x_s$ connecting some point $a$ in $\Theta_m$
to $c$. When $c\in\Theta_m$, the path is empty. If $c\in\Theta_{m+1}-\Theta_m$ then $c$ lies on
some chain $a=u_0,\ldots,u_4=b$ of a cell with base $\{a,b\}$, joined in the $m+1$-th step, say
$c=u_s$, $0<s\le 2$. Then the path $x_1,\ldots,x_s$ is $\{u_0,u_1\},\ldots,\{u_{s-1},u_s\}$. Since
$\{c,c'\}\not\in r$, we have $c'\ne u_{s+1}$, $c'\ne u_{s-1}$, therefore $f_{x_i}$ maps $\{c,c'\}$
onto $x_i$ for $i=1,\ldots,s$. Thus $\{c,c'\}$ dominates $\{c,a\}$. Similarly, there is a path
$x_1',\ldots,x'_t$ connecting some point $a'$ in $\Theta_m$ to $c'$ such that $\{c,c'\}D_F
x_1',\ldots,x_t'$, consequently $\{c,c'\}D_F\{c',a'\}$. Since, clearly, $\{c,c'\}D_F\{c,c'\}$
through the identity mapping, we have $\{c,c'\}D_F\{a,a'\}$. Now, by induction hypothesis, there is
a path in $r$ connecting $a$ to $a'$, say $y_1,\ldots,y_n$, such that $\{a,a'\}D_F\{y_i\}$ for all
$i=1,\ldots,n$. Since $F$ is a monoid, $D_F$ is transitive so that $\{c,c'\}D_F\{y_1,\ldots,y_n\}$.
Hence $x_s,\ldots,x_1,y_1,\ldots,y_n,x_1',\ldots,x_t'$ is the desired path in $r$ connecting $c$ to
$c'$.
\end{proof}

\subsection{Representations of homomorphic images}

Throughout this subsection fix finite lattices $L^0$, $L^1$ and a $\la 0,1,\vee\ra$-homomorphism
$\varphi:L^0\rightarrow \varphi(L^0)\subseteq L^1$. Let the $\wedge$-homomorphism
$\varphi^*:L^1\rightarrow L^0$ be defined by $\varphi^* x=\bigvee\{a\in L^0\mid \varphi(a)\le x\}$ (in
for example \cite{Gierz.ea:80}, $\varphi^*$ would be called the \emph{Galois adjoint} of $\varphi$).

The map $\varphi^*$ has many nice properties; we list the ones we need in the following lemma.

\begin{lem}\label{l:ka}
\begin{enumerate}
\item $\varphi^*$ is a $\la\wedge,1\ra$-homomorphism.
\item If $x<1$ then $\varphi^*x<1$.
\item $\varphi^*$ is injective on $\varphi L^0$.
\item $a\le\varphi^*x\iff\varphi^*\varphi a\le\varphi^* x$.
\end{enumerate}
\end{lem}
\begin{proof}
These all follow easily from the definition of $\varphi^*$ and the fact that $\{a\in
L^0\mid \varphi(a)\le x\}$ is the principal ideal generated by $\varphi^*(x)$, i.e., $\{a\in
L^0\mid a\le\varphi^*(x)\}$.
\end{proof}

\begin{df}\label{d:key}
We say that \emph{$\Theta(L^1)$ embeds in $\Theta(L^0)$ with respect to $\varphi$} if there is a function
$\Theta(\varphi):\Theta L^1\rightarrow\Theta L^0$ such that $x\sim_{\varphi\alpha}
y\Iff\Theta(\varphi)(x)\sim_\alpha \Theta(\varphi)(y)$.
\end{df}

\begin{lem}\label{l:kiy}
Let $\mf C(\varphi)\Theta L^1$ be the graph obtained from $\Theta L^1$ by replacing each color
$\alpha$ by $\varphi^*\alpha$. $\mf C(\varphi)\Theta L^1$ is isomorphic to a subgraph of $\Theta
L^0$.
\end{lem}
\begin{proof}
Each pentagon of $\mf C(\varphi)\Theta L^1$ represents an inequality of the form
``$\varphi^*\alpha_1\wedge\varphi^*\alpha_2\le\varphi^*\alpha$'', for $\alpha_1,\alpha_2,\alpha\in
L^1$ satisfying $\alpha_1\wedge\alpha_2\le\alpha$. Then
$\varphi^*\alpha_1\wedge\varphi^*\alpha_2=\varphi^*(\alpha_1\wedge\alpha_2)\le\varphi^*\alpha$, so
the represented inequality holds in $L^0$.

Hence we can obtain an isomorphic copy of $\mf C(\varphi)\Theta L^1$ within $\Theta L^0$ by running
through the construction of $\Theta L^0$, omitting every pentagon that either represents an
inequality involving members of $L^0-\varphi^* L^1$ or (recursively) is based on such a pentagon,
and omitting inequalities that are true in $L^0$ but not in $L^1$. Since $\varphi^*(x)=1\to x=1$,
recoloring of points is never identification of points.
\end{proof}

\begin{lem}\label{l:tableproperty}
$\Theta(L^1)$ embeds in $\Theta(L^0)$ with respect to $\varphi$.
\end{lem}
\begin{proof}
Let $\Theta(\varphi):\mf C(\varphi)\Theta L^1\rightarrow \Theta L^0$ be the embedding of Lemma
\ref{l:kiy}.

Using edges in $\Theta L^0 - \Theta(\varphi)\mf C(\varphi)\Theta L^1$ does not induce any new
equivalences in $\Theta(\varphi)\mf C(\varphi)\Theta L^1$, so we can confuse $\Theta(\varphi)\mf
C(\varphi)\Theta L^1$ and $\Theta L^0$ when discussing elements of the former. Elements of the
underlying set of $\Theta L^1$ will be understood to be colored as in $\Theta L^1$ and not as in
$\mf C(\varphi)\Theta L^1$.

Let $x,y\in\Theta L^1$ and $\alpha\in L^0$. By Lemma \ref{l:kiy} and by definition of $\mf
C(\varphi)\Theta L^1$, $x\sim_{\varphi\alpha}y\Iff\Theta(\varphi) x
\sim_{\varphi^*\varphi\alpha}\Theta(\varphi) y$. (Here we use the fact that $\varphi^*$ is
injective on $\varphi L^0$.)

Colors of edges in $\Theta(\varphi)\mf C(\varphi)\Theta L^1$ are all of the form $\varphi^*\alpha$
for some $\alpha$, and hence of the form $\varphi^* \varphi\gamma$ for some $\gamma$ (namely
$\gamma=\varphi\varphi^*\alpha$).

Now $\Theta(\varphi)x\sim_{\varphi^*\varphi\alpha}\Theta(\varphi)y$ iff there is a path from
$\Theta(\varphi)x$ to $\Theta(\varphi)y$, all edges of which are colored
$\ge\varphi^*\varphi\alpha$, or equivalently by Lemma \ref{l:ka}, colored $\ge\alpha$. Hence
equivalently $\Theta(\varphi)x\sim_\alpha\Theta(\varphi)y$.
\end{proof}

\subsection{$\Sigma^0_2$-presentable semilattices as direct limits}

\begin{df}
Suppose $\lesssim$ is a preorder (transitive and reflexive binary relation) on a countable set $L$
and $\vee^*$ is a binary operation on $L$. Define an equivalence relation $\approx$ by 
\[
	a\approx b\quad\Longleftrightarrow\quad a\lesssim b\quad\text{and}\quad b\lesssim a.
\] 
Let $L/\negthickspace\approx$ be the set of
$\approx$-equivalence classes.

Let $(\lesssim\negthickspace/\negthickspace\approx)$ and $(\vee^*/\negthickspace\approx)$ be the
relation on $L/\negthickspace\approx$ induced by $\lesssim$ and the operation on
$L/\negthickspace\approx$ induced by $\vee^*$, respectively.

Assume $L/\negthickspace\approx=\la
L/\negthickspace\approx,\lesssim\negthickspace/\negthickspace\approx,\vee^*/\negthickspace\approx\ra$
is an upper semilattice. Then we call $\vee^*$ a \emph{pre-join function}. Assume that $\lesssim$ is
$\Sigma^0_1(\mb a)$ (so $\approx$ is $\Sigma^0_1(\mb a)$ too) and $\vee^*$ is $\Delta^0_1(\mb a)$,
where $\mb a$ is a Turing degree.

If such $\lesssim$, $\vee^*$ exist then (the upper semilattice isomorphism type of)
$L/\negthickspace\approx$ is called $\Sigma^0_1(\mb a)$-\emph{presentable}.
\end{df}

\begin{lem}
$[\mb a,\mb b]$ is $\Sigma^0_3(\mb b)$-presentable, for any Turing degrees $\mb a\le\mb b$.
\end{lem}
\begin{proof}
Let $B\in\mb b$, $A\in\mb a$, choose $e$ such that $A=\{e\}^B$, and let 
\[
C=\{i\mid \{i\}^B\text{ is total and } \{e\}^B\le_T \{i\}^B\}.
\]
The set $C$ is $\Sigma^0_3(B)$ by a standard argument, so
$C=\{h(n)\mid n<\omega\}$ for some injective $h\le_T B''$.

Let $\lesssim$ be the binary relation on $\omega$ given by $i\lesssim j\iff \{h(i)\}^B\le_T
\{h(j)\}^B$. Recall the function $\oplus:2^\omega\times 2^\omega\rightarrow 2^\omega$ defined by
$A\oplus B=\{2x\mid x\in A\}\cup\{2x+1\mid x\in B\}$. Let $\oplus:\omega\times\omega\rightarrow\omega$ be
a total recursive function such that for each $X\subseteq\omega$ and $a,b<\omega$, if $\{a\}^X$,
$\{b\}^X$ are both total and $\{a\}^X, \{b\}^X\subseteq\omega$ then $\{a\oplus
b\}^X=\{a\}^X\oplus\{b\}^X$. Let $i\vee^* j=h^{-1}(h(i)\oplus h(j))$.

It is easily verified that $\lesssim$ is $\Sigma^0_3(B)$ and
$\vee^*:\omega\times\omega\rightarrow\omega$ is $B''$-recursive, and $\lesssim$ and $\vee^*$ have
the required properties.
\end{proof}

\begin{df}[Direct limit]\label{d:bo}
Let a sequence $i\mapsto\la L^i,\varphi_i\ra$ be given, where each $L^i$ is a finite lattice,
$\varphi_i:L^i\rightarrow L^{i+1}$ is a $\la 0,1,\vee\ra$-homomorphism, and $L^i\intersect
L^j=\emptyset$ for $i\ne j$.

Let $L=\bigcup_{i<\omega} L_i$ as a set. Let $\approx$ be the equivalence relation on $L$ generated by $a\approx\varphi_i(a)$ for $a\in
L^i$.

In each $L^i$, let the order be denoted by $\le$ and the join function by $\vee$. Define:
\[
a\lesssim b\quad\Longleftrightarrow\quad(\exists i)(\exists a_0,b_0\in L^i)( a_0\approx a\and b_0\approx b\and
L^i\models a_0\le b_0),
\]
and
\noindent $a\vee^* b=c$ if, for the $i$ with $c\in L^i$, we have
\[
(\exists a_0,b_0\in L^i)(a_0\approx a\and b_0\approx b\and L^i\models a_0\vee b_0=c),
\]
\[
\text{but }\neg(\exists j<i)(\exists a_0,b_0\in L^j)(a_0\approx a\and b_0\approx b).
\]
Let $L/\negthickspace\approx=\la L/\negthickspace\approx,\lesssim\negthickspace/\negthickspace\approx,\vee^*/\negthickspace\approx\ra$
and $L=\la L,\lesssim,\vee^*\ra$.

The \emph{direct limit} of the sequence $i\mapsto\la L^i,\varphi_i\ra$ is the usl $L/\negthickspace\approx$.

By abuse of language we may also speak of $L$ as an usl and as the direct limit of $i\mapsto\la
L^i,\varphi_i\ra$.
\end{df}

\begin{df}\label{permissible}
Let $L$ be an usl and $\mb a$ a Turing degree. We say that \emph{$L$ is $\mb a$-permissible} if the
following holds:

There exists an $\mb a$-recursive relation $R$, a recursive function $f$, and a sequence $L^i$ of
finite lattices with $\la 0,1,\vee\ra$-homomorphisms $\varphi_i:L^i\rightarrow L^{i+1}$, such that
$L$ is the direct limit of $i\mapsto \la L^i,\varphi_i\ra$, and such that for each $i$,
\[
	\forall x R(x,i)\iff f(L^i,i)\ne \la L^i,\text{id}\ra \iff f(L^i,i)=\la L^{i+1},\varphi_i\ra.
\]
That is, a $\Pi_1(\mb a)$ outcome corresponds to an increment in $i$, whereas a $\Sigma_1(\mb a)$
outcome does not.
\end{df}

\begin{pro}
Let $\mb a$ be a Turing degree, and let $L$ be a $\Sigma^0_2(\mb a)$-presentable usl. Then $L$ is
$\mb a$-permissible. Conversely, if $L$ is $\mb a$-permissible then $L$ is $\Sigma^0_2(\mb a)$-presentable.
\end{pro}
\begin{proof}
Since $\vee^*$ is $\Delta^0_2(\mb a)\subseteq\Sigma^0_2(\mb a)$, there exists an $\mb a$-recursive
relation $U$ such that $x\vee^*y=z\iff(\exists n)(\forall w)U(x,y,z,w,n)$. (We can write $x\vee^*
y[n]\downarrow=z$ if $(\forall w)U(x,y,z,w,n)$.)

Since $\lesssim$ is $\Sigma^0_2(\mb a)$, there exists an $\mb a$-recursive relation $V$ such that
$x\lesssim y\iff(\exists n)(\forall z)V(x,y,z,n)$. (We can write $x\lesssim y[n]$ if $(\forall
z)S(x,y,z,n)$.)

Fix recursive functions $f_1:\omega^4\to\omega$, $f_2:\omega^3\to\omega$, written $f_1(a,b,c,d)=\la
a,b,c,d\ra$, $f_2(a,b,c)=\la a,b,c\ra$, such that each $n<\omega$ is in the range of exactly one of
$f_1$, $f_2$.

\begin{enumerate}
\item[Case 1.] $i=\la x,y,z,n\ra$ for some $x,y,z,n<\omega$.
If
\[
	(\forall x_0,y_0\in L\union\{x,y,z\})(\exists n_0,z_0)(\la x_0,y_0,z_0,n_0\ra<i\and\forall w(U(x_0,y_0,z_0,w,n_0))
\] 
then $f(L,i)=L\union\{x,y,z\}$, otherwise $f(L,i)=L$.

(Such a pair $\la n_0,z_0\ra$ will always eventually appear, since $(\exists z_0)(x_0\vee^*
y_0=z_0)$. So the above ensures that we always work with an $L$ for which we know $\vee^*\restrict
L$.)

\item[Case 2.] $i=\la x,y,n\ra$ for some $x,y,n<\omega$.

If $\forall zS(x,y,z,n)$ then $f(L,i)$ is the homomorphic image of $L$ obtained from $L$ and the
relation $x\le y$, otherwise $f(L,i)=L$.

\end{enumerate}

Now define $L^0=$the 2-element lattice, and by recursion $L^{i+1}=f(L^i,i)$.

It is clear how to obtain $R$ from $U$ and $V$.

The converse is easy using Definition \ref{d:bo} and the fact that the sequence $i\mapsto\la
L^i,\varphi_i\ra$ is $\mb a'$-recursive. The preorder $\lesssim$ is automatically $\Sigma^0_2(\mb
a)$ as defined there. The definition of $\vee^*$ should be modified so that $L^i$ is not the first
finite lattice containing representatives of the equivalence classes of $a,b$, but rather the first
finite lattice for which we discover that it contains such representatives using the fact that
$\approx$ is $\Sigma^0_2(\mb a)$. (This does not affect the isomorphism type of the limit
$L/\negthickspace\approx$.)
\end{proof}

\subsection{Representations of direct limits}

We now show how to combine embeddings to form a table for an usl given as a direct limit of finite
lattices.

\begin{df} Let $m_i(n)=\mu m(\Theta(\varphi_i)\Theta^{i+1}_n\subseteq\Theta^i_m)$ and let $h(i)=m_i(0)$, for $i<\omega$.
Let $\Theta^0=\Theta(L^0)$ and for $i\ge 1$,
\[
\Theta^i=\Theta(\varphi_0)\cdots\Theta(\varphi_{i-1})\Theta(L^i).
\]
Let a ramification of $\Theta^i$ be given by
\[
\Theta^i_k=
\begin{cases}
\Theta(\varphi_0)\cdots\Theta(\varphi_{i-1})\Theta_j(L^i) & \text{ if } k=m_0m_1\cdots m_{i-1}j\,(\exists j), \\
\Theta^i_{k-1} & \text{otherwise.} \\
\end{cases}
\]
\end{df}

\begin{cor}\label{c:c}
Let $\mb a\in\mc D$, and let $L$ be a $\Sigma^0_1(\mb a)$-presentable usl.
There is
\begin{enumerate}
\item an $\mb a$-recursive function $h$,
\item an $\mb a$-recursive array $\{\Theta^i_j\mid j\ge h(i)\}$, and
\item for each $i<\omega$ an increasing recursive function $m_i$,
\end{enumerate}
such that
\begin{enumerate}
\item $m_i(0)=h(i)$ and for each $k$, $m_i(j)\le k<m_i(j+1)\rightarrow
\Theta^i_k=\Theta^i_{m_i(j)}$,
\item for each $i<\omega$ $\la\Theta^i_{m_i(j)}\mid j\in\omega\ra$ is a recursive sequential Malcev
homogeneous lattice table for $L^i$, and
\item $\Theta^{i+1}$ embeds in $\Theta^i$ with respect to $\varphi_i$.
\end{enumerate}

Each $\Theta^i$ is recursive, $\Theta^0\supseteq\Theta^1\supseteq\cdots$ and $|\Theta^i_j|<\omega$.
There is a recursive function taking $L^0,\ldots,L^i$ to $\Theta^i$.
\end{cor}

\begin{proof}
These facts about $m_i$ and $\Theta^i_j$ as defined above are all either easy to see or proved
above.
\end{proof}

%% file: PRESEN-21.tex
\makeatletter
\renewcommand{\theenumi}{(\roman{enumi})}
\renewcommand{\labelenumi}{\theenumi}
\renewcommand{\p@enumi}{}
\renewcommand{\theenumii}{(\alph{enumii})}
\renewcommand{\labelenumii}{\theenumii}
\renewcommand{\p@enumii}{}
\makeatother

\subsection{Definition of Lerman constructions}

\begin{df}[Tree]
Suppose $\Theta=\Union_{i<\omega}\Theta_i$ where each $\Theta_i$ is a finite subset of $\omega$ and $\Theta_i\subseteq\Theta_{i+1}$ and similarly for $\Theta'$ in place of $\Theta$. Let
\[
	S(\Theta)=\{\sigma\in\omega^{<\omega}\mid (\forall j) \sigma(j)\in\Theta_j\}.
\]
A \emph{$(\Theta,\Theta')$-tree} is a partial function $T$ from $S(\Theta)$ to $S(\Theta')$ such that
\begin{enumerate}
\item the domain of $T$ is an initial segment of $\omega^{<\omega}$, 
\item $\la\dom T,\subseteq\ra\cong\la\ran T,\subseteq\ra$, and 
\item $|\sigma|=|\tau|\and \sigma,\tau\in\dom T \Rightarrow |T(\sigma)|=|T(\tau)|$. 
\end{enumerate}

In the rest of this definition, assume $T$ is a $(\Theta,\Theta')$-tree.

\emph{Level $i$} of $T$ is the interval $[|T(\xi)|,|T(\eta)|)$ where $|\xi|=i=|\eta|-1$. If $T(\emptyset)\ne\emptyset$, we also define \emph{level -1} to be the interval $[|\emptyset|,|T(\emptyset)|)$.

If $\sigma\in\dom T$ then $T(\sigma)$ is called a \emph{potential focal point} of $T$ if $(\forall \tau\in\dom T)(|\tau|=|\sigma|\rightarrow \tau=\sigma)$. If in addition $(\exists\tau\in\dom T)(\tau\supset\sigma)$ then $\sigma$ is called a \emph{focal point}.

A \emph{focal length} is a number $r=|T(\xi)|$ where $T(\xi)$ is a potential focal point. The smallest focal length is consequently $|T(\emptyset)|$.

If the focal lengths of $T$ are $r_0<r_1<\cdots$, then \emph{plateau -1} of $T$ is defined to be level -1 of $T$ (if defined), and \emph{plateau $i$} of $T$ is defined to be $[r_i,r_{i+1})$. The
\emph{height} of plateau $i$ is $r_{i+1}$ and the height of $T$, $\Ht(T)$, is the height of the last plateau of $T$, if it exists.

\emph{Plateau $i$ is full} if whenever $T(\sigma$, $T(\tau)$ are potential focal points with $r_{i}=|T(\sigma)|$, $r_{i+1}=|T(\tau)|$, then $\dom T\supseteq\{\sigma*\eta \in S(\Theta):
|\sigma*\eta|\le|\tau|\}$. Naturally, if plateau $i$ is full for each $i$ for which $r_{i+1}$ is defined, then we say that \emph{each plateau of $T$ is full}.

$T$ is called a \emph{tree} if $T$ is a $(\Theta,\Theta')$-tree (for some $\Theta,\Theta'$) and each plateau of $T$ is full.
\end{df}

\begin{df}\label{V12}
Let $T$ be a tree, and let $\sigma\in\omega^{<\omega}$ and $f:\omega\to\omega$ be given. We say that \emph{$\sigma$ is on $T$} (write $\sigma\subset T$) if $\sigma=T(\tau)$ for some
$\tau\in\dom(T)$. We say that \emph{$\sigma$ is compatible with $T$} if there is a $\tau\supseteq\sigma$ such that $\tau\subset T$. We say that \emph{$f$ is on $T$} (write $f\subset T$) if for all $\sigma\subset f$, $\sigma$ is compatible with $T$. If $f\subset T$, we refer to $f$ as a \emph{branch} of $T$.

Let $g:\omega\to\omega$ with $g\subset T$. We define $T^{-1}(g)=\Union\{\sigma\mid T(\sigma)\subset g\}$. The \emph{height function} $H_T$ for $T$ is defined by $H_T(n)=|T(\sigma)|$ for any
$\sigma\in\dom (T)$ with $|\sigma|=n$. 
\end{df}

\begin{df}\label{V16}
Let $T$ and $T^*$ be trees. We call $T^*$ a \emph{subtree} of $T$ (write $T^*\subseteq T$) if ran$(T^*)\subseteq\ran(T)$.
\end{df}

\begin{df}\label{V211}
Let $\sigma,\tau,\rho\in\omega^{<\omega}$ be given such that $\sigma\subset\rho$ and $|\sigma|=|\tau|$. We define the string tr$(\sigma\to\tau;\rho)$, the transfer of $\sigma$ into $\tau$ below $\rho$, by tr$(\sigma\to\tau;\rho)(x)=\tau(x)$ if $x<|\tau|$, $=\rho(x)$ if $|\tau|\le x<|\rho|$, and $\uparrow$ otherwise.
\end{df}

\subsubsection{Subtree constructions}

\begin{df}\label{21}
Suppose $T$ is a tree, that $T_s$ is a finite tree for each $s<\omega$, and $T=\Union_s T_s$. Let $\Theta,\Theta'$ be given such that $T$ and each $T_s$ are $(\Theta,\Theta')$-trees.

If $\alpha=T(\eta)\subset T_s$, let $\alpha^*=T(\eta^*)$ be the longest focal point of $T_s$ with $\alpha^*\subseteq\alpha$ (call $\alpha^*$ the \emph{shortening} of $\alpha$) and fix the greatest
$m$ with $T_s(\delta)\downarrow$ for some $\delta$ with $|\delta|=m$.

$T_{s+1}$ is a \emph{type 0} (grow plateau) (deep and tall) extension of $T_s$ for $\alpha$ if $T_{s+1}(\lambda)\downarrow$ is added for each $\lambda\supseteq\eta^*$ with $|\lambda|\le m+1$ and
$\lambda\in S(\Theta)$.

$T_{s+1}$ is a \emph{type 1} (create plateau) (tall) extension of $T_s$ if the same is true for $\alpha$ instead of the focal point.

$T_{s+1}$ is a \emph{type 2} (combine plateaus) (deep) extension of $T_s$ if the same is true for the focal point but for $m$ instead of $m+1$.

In suggestive terminology, we may also say that type 0,1,2 extensions are responses to \emph{grow, create, combine} (plateaus) requests.
\end{df}

\begin{df}[Dynamical tree]
A \emph{dynamical tree} is a tree $T$ together with
\begin{enumerate}
\item an integer $k=k(T)\ge -1$ and (inductively) a sequence of dynamical trees $T_0,\ldots,T_{k+1}$ with $T=T_{k+1}$,
\item a decomposition of $T_{k+1}$ into finite trees, $T_{k+1}=\Union_{s<\omega} T_{k+1,s}$ (also sometimes written $T=\Union_{s<\omega}T_s$ if no confusion is likely),
\item at each \emph{stage} $s<\omega$ a finite set $S_{k+1,s}$ of \emph{received} pairs $\la \alpha,i\ra\in\omega^{<\omega}\times\omega$, and
\item at each stage $s<\omega$ a finite set $S'_{k+1,s}$ of \emph{transmitted} pairs such that either $S'_{k+1,s}=\emptyset$ or $S'_{k+1,s}$ has exactly one element $\la\alpha,i\ra\in\omega^{<\omega}\times\omega$,
\end{enumerate}
such that there is a total recursive function $\{e_T\}$ such that
\[
	\{e_T\}(s, \la T_{i,j} \mid i\le k+1, j\le s\ra, \la S_{k+1,j}\mid j\le s+1\ra)=\la T_{k+1,s+1},S'_{k+1,s+1}\ra.
\]

To \emph{designate} $T$ means to define $e_T$. Hence designating $T$ suffices to define $T$, modulo the sequence $T_0,\ldots,T_k$ and the map $s\mapsto S_{k+1,s}$.

In the following, all trees considered (except Ext trees) will be dynamical trees, which will hence also be referred to simply as \emph{trees}. More generally we allow that $T=\Union_{s\ge s^*}T_s$ for
some $s^*<\omega$ and $T_s$ is undefined for $s<s^*$.
\end{df}

\begin{df}[Appropriate instructions]\label{23}
Suppose the tree $T$ receives pairs $\la\alpha,i\ra$ and $\la\beta,j\ra$ at stage $s+1$. We say that reception of pairs by $T$ is \emph{appropriate} if the following conditions obtain:
\begin{enumerate}
\item \label{23i} $\alpha\subset T_s$.
\item \label{23ii} $i=0\to|\alpha|<$ht$(T_s)$.
\item \label{23iii} $i=1\to\alpha$ is a potential focal point of $T_s$ which is not a focal point of
$T_s$.
\item \label{23iv} $(i\ge 2\and T$ does not receive $\la\alpha,i\ra$ at $s$) $\,\to\,\alpha$ is a
potential focal point of $T_s$.
\item \label{23v} $\alpha\subseteq\beta$ or $\beta\subseteq\alpha$.
\end{enumerate}
\end{df}

\begin{rem}
The reception of $\la\alpha,i\ra$ by the tree $T$ at stage $s+1$ will convey the instruction to carry out Objective $i$ whenever possible. The objectives are listed below.

\noindent\emph{Objective 0. Combine plateaus and create a new level.} This objective will be met when $T_{s+1}$ is a type 0 extension of $T_s$ for $\alpha$ of height $t$, where $t>$ht$(T_s)$ is specified. There is a type 2 extension $T^*$ of $T_s$ for $\alpha$ within $T_{s+1}$. $T_{s+1}$ adds one level to $T^*$.

\noindent\emph{Objective 1. Designate a new focal point.} This objective will be met when $T_{s+1}$ is a type 1 extension of $T_s$ for $\alpha$ of height $t$, where $t>$ht$(T_s)$ is specified. It is used to force $T=\Union_s T_s$ to be infinite, with infinitely many focal points.
\end{rem}

The remaining objectives ($i\ge 2$) depend on the requirements we want to satisfy.

\begin{lem}\label{24}
Suppose $T$, $T^*$ are trees with $T\subseteq T^*$ and $T$ finite. Let $\alpha\subset T$.

\begin{enumerate}
\item \label{24i} 
	If $|\alpha|<\Ht(T)<\Ht(T^*)$ and the shortening of $\alpha$ on $T$, $\alpha^*$, is in the last plateau of $T^*$ then there is a deep and tall extension of $T$ within $T^*$ for $\alpha$ of height $\Ht(T^*)$.
\item \label{24ii} 
	If $\alpha$ is not terminal on $T^*$ but in the last plateau of $T^*$ and $\alpha$ is a merely potential focal point of $T$ then there is a tall extension of $T$ within $T^*$ for $\alpha$ of height $\Ht(T^*)$.
\item \label{24iii} 
	If $T^*$ contains a deep and tall extension of $T$ for $\alpha$ then there is a deep extension of $T$ for $\alpha$ within $T^*$.
\end{enumerate}

\end{lem}

\begin{proof}
\begin{enumerate}
\item

The deep and tall extension of $T$ for $\alpha$ within $T^*$ is obtained in two steps.

First we \emph{copy sideways} above $\alpha^*$ if $\alpha$ is not in the last plateau of $T$. \emph{Copying sideways} means that if $T(\sigma)=T^*(\tau)$ and $T(\sigma_0)=T^*(\tau_0)$, $|\sigma_0|=|\sigma|$ and $T(\sigma_0), T(\sigma)\supseteq\alpha^*$ then if $T(\sigma_0*n)\uparrow$
and $T(\sigma*n)\downarrow$, then we copy by defining $T(\sigma_0*n)=T^*(\tau_0*\gamma)$ (which $\downarrow$ by fullness of the last plateau of $T^*$) where $T(\sigma*n)=T^*(\tau*\gamma)$.
Continue this process until $\alpha$ is in the last plateau of $T$.

Then we stretch terminal strings of $T$ to $\Ht(T^*)$.

This means that if $T(\eta)=T^*(\xi)$ is terminal on $T$ then we define $T(\eta*a)=T^*(\xi*a^n)$ (which $\downarrow$ by fullness of the last plateau of $T^*$) [where $a^{n+1}=a^n*a$] where $n$
satisfies $|T^*(\xi*a^n)|=\Ht(T^*)$. (This could be done in various ways, we specify one way for definiteness.)

\item

We have $|\alpha|=\Ht(T)$ so this is true since the last plateau of $T^*$ is full, and the branching rate of $T$ is less than that of $T^*$ there.

\item
Since deep and tall implies deep.
\end{enumerate}
\end{proof}

\begin{df}\label{25}
Let $T$ be a tree which receives the set of ordered pairs $S_{s+1}$ at stage $s+1$. We say that $T_s$ \emph{prefers} $\la\alpha,i\ra$ if $i\le 1$ and $\la\alpha,i\ra$ is the first element of $S_{s+1}$
under a recursive ordering of all such possible pairs, fixed throughout the construction of $T$.
\end{df}

It will turn out that $T$ receives at most one pair $\la\alpha,i\ra$ with $i\le 1$, so the above definition is of interest mostly in so far as it helps with the verification.

\begin{nota}\label{26}
Let $\xi\ne\emptyset$ be given. $\xi^-$ is the unique $\lambda$ such that $\xi=\lambda*i$ for some $i<\omega$. Then $s(\xi)=\xi^-*(i+1)$ and $p(\xi)=\xi^-*(i-1)$, if $i>0$.
\end{nota}

\subsubsection{Init and Ext trees}

\begin{df}\label{27}
Let Id be the full $(\Theta,\Theta)$-identity tree for some $\Theta$ (on which the following definition depends), i.e., the full identity tree on $S(\Theta)$, and let $\Id_t=\{\sigma\subset\Id\mid |\sigma|\le t\}$. 

The \emph{initial tree} Init is defined by: $\Init_{t+1}$ is an extension of $\Init_t$ within $\Id_{t+1}$ of height $t+1$ specified by the highest priority request received at stage $t+1$; if no request with
priority is received, then $\Init_{t+1}$=$\Init_t$.

Let $\{S_t\mid t>0\}$ be a recursive sequence of finite sets of pairs $\la\alpha,i\ra$. We construct a recursive sequence of finite trees, $\{$$\Init_t(\{S_u\mid 0<u\le t\})\mid t<\omega\}$, whose union is the
{initial tree} specified by $\{S_t\mid t>0\}$, $\Init(\{S_t\mid t>0\})$. For convenience, we use $T_t$ to denote $\Init_t(\{S_u\mid 0<u\le t\})$. The set of strings $S_t$ is received at stage $t$.

We begin by setting $T_0=\Id_0$. Given $T_t$, fix $\la\alpha,i\ra\in S_{t+1}$ such that $T_t$ prefers $\la\alpha,i\ra$, if such a pair $\la\alpha,i\ra$ exists. If no such pair $\la\alpha,i\ra$
exists or if $\la\alpha,i\ra$ does not satisfy \ref{23} for $T_t$, let $T_{t+1}=T_t$. Otherwise, let $T_{t+1}$ be a type $i$ extension of $T_t$ for $\alpha$ of height $t+1$ such that
$T_{t+1}\subseteq\Id_{t+1}$. Note that by Lemma \ref{24}, such an extension exists. No information is transmitted by $T_t$ for any $t$.
\end{df}

\begin{lem}\label{28}
Fix $\{S_t\mid t>0\}$, let $\Init_t$ denote $\Init_t(\{S_u\mid 0<u\le t\})$ and let Init denote $\Init(\{S_t\mid t>0\})$. Suppose that \ref{23} is satisfied at stage $t+1$ by $\Init_t$ for all $\la\alpha,i\ra\in S_{t+1}$. Then:
\begin{enumerate}
\item \label{28i} $\Init_{t+1}\ne$$\Init_t$ iff $\Init_t$ prefers some $\la\alpha,i\ra$, in which case $\Init_{t+1}$ is a type $i$ extension of $\Init_t$ for $\alpha$ of height $t+1$.
\item \label{28ii} If $\alpha$ is a potential focal point of $\Init_t$ and for all $\la\beta,j\ra$ received by $\Init_t$ with $j\le 1$, either $j=1$ and $\alpha\subseteq\beta$, or $j=0$ and
$\alpha\subseteq\beta$ and $\alpha$ is a focal point of $\Init_t$, then $\alpha$ is a potential focal point of $\Init_{t+1}$.
\end{enumerate}
\end{lem}

\begin{proof}
\ref{23}\ref{23ii},\ref{23iii} correspond to Lemma \ref{24}\ref{24i},\ref{24ii}. So since requests only involve extensions from Lemma \ref{24}\ref{24i},\ref{24ii}, Init is well-defined.

Indeed if the highest priority is a deep and tall extension for $\alpha$ then by \ref{23}\ref{23ii}, $\alpha$ is not at the top of $\Init_t$. Since Id has only one plateau, $\alpha$ is in the last plateau of $\Id_{t+1}$. So by Lemma \ref{24}\ref{24i}, the extension exists within $\Id_{t+1}$.

And if the highest priority is a tall extension for $\alpha$ then by \ref{23}\ref{23iii} and Lemma \ref{24}\ref{24ii}, the extension exists within $\Id_{t+1}$.

If $\alpha$ satisfies the hypothesis of the lemma then every extension will be above $\alpha$; in particular a deep (and tall) extension will not go below $\alpha$ since if such is received then
$\alpha$ is its own shortening and hence below the shortening of any string above $\alpha$. Hence if no extension was made, then $\alpha$ remains a potential focal point, and if an extension is
made then $\alpha$ is a focal point after the extension.
\end{proof}

\begin{df}[Appropriate execution]\label{29}
Fix $s^*,t,k<\omega$. For each $i\le k$, let $\{T_{i,s}\mid s^*\le s\le t\}$ be a sequence of finite trees. The array $\{T_{i,s}\mid i\le k\and s^*\le s\le t\}$ is \emph{special} if it satisfies the
following conditions:
\begin{enumerate}
\item \label{29i} $\forall i<k\,\forall s\in[s^*,t](T_{i+1,s}\subseteq T_{i,s})\and\forall i\le k\,\forall s\in[s^*,t)(T_{i,s+1}$ extends $T_{i,s})$.
\item \label{29ii} $\forall i\le k\,\forall s\in[s^*,t]\forall\alpha($If $\alpha$ is a potential focal point of $T_{i,s}$ and for all $m\le i$ and $\la\beta,j\ra$ received by $T_{m,s}$, either 
\begin{enumerate}
\item[(1)] $j\ge 2$, or
\item[(2)] $j=1$ and $\alpha\subseteq\beta$, or 
\item[(3)] $j=0$ and $\alpha\subseteq\beta$ and $\alpha$ is a focal point of $T_{m,s}$, 
\end{enumerate}
then $\alpha$ is a potential focal point of $T_{i,s+1}$.$)$
\item \label{29iii} $\forall i\le k\,\forall s\in[s^*,t)(T_{i,s+1}\ne T_{i,s}\rightarrow\forall j<i($ht$(T_{j,s+1})=$ht$(T_{i,s+1}))\and(T_{i,s}\ne\emptyset\rightarrow(\exists m\le 1)(\exists\alpha)(T_{i,s}$ prefers $\la\alpha,m\ra\and T_{i,s+1}$ is a type $m$ extension of $T_{i,s}$ for $\alpha)))$.
\end{enumerate}
\end{df}

\begin{rem}\label{210}
For all $t<\omega$, $\{$$\Init_s\mid s\le t\}$ is special.
\end{rem}

\begin{df}\label{211}
Let $s,s^*<\omega$ be given such that $s^*\ge s$. Let $T$ be a tree defined by $T=\Union\{T_t\mid t\ge s\}$, where $\{T_t\mid t\ge s\}$ is a recursive sequence of increasing finite trees. Let $\xi$ be
given such that $T_{s^*}(\xi)\downarrow$. Define the tree $T^*=\Ext(T,\xi,s^*)=\Union\{T^*_t\mid t\ge s^*\}$, where $T^*_t=\Ext(T_t,\xi) = \{T_t(\eta)\mid \eta\supseteq\xi\and T_t(\eta)\downarrow\and
\eta\in S(\Theta)\}$ where $T_t$ is a $(\Theta,\Theta')$-tree. To remind ourselves that $T^*_t$ exists only for $t\ge s^*$ we may write $\Ext(T_t,\xi,s^*)$ for $\Ext(T_t,\xi)$. $T^*_t$ transmits
exactly those pairs which it receives; these pairs are received by $T_t$.
\end{df}

\begin{rem}\label{212}
Let $T,\xi,s$ and $s^*$ be given as in Definition \ref{211}. Then the following conditions hold:
\begin{enumerate}
\item \label{212i} For all $t\ge s^*$, $\Ext(T_t,\xi,s^*)\subseteq T_t$; and for all $t>s^*$, $\Ext(T_t,\xi,s^*)$ extends $\Ext(T_{t-1},\xi,s^*)$.
\item \label{212ii} For all $t\ge s^*$, $\Ext(T_t,\xi,s^*)$ transmits exactly those pairs which it receives.
\end{enumerate}
\end{rem}

Note that $\Ext(T)$ eliminates some branches and if $T$ is a $(\Theta,\Theta')$-tree then so is $\Ext(T)$. Also note that the Ext tree is computed only after its parent, \emph{at the same stage},
is computed.

\label{thatsubsubsection}

\subsubsection{Phase-1 and Phase-2 trees}

\begin{df}[Tree with states]
A dynamical tree $T$ together with two maps State=State$_T$, States=States$_T:\omega\to\omega$ is called a \emph{tree with states} if
\begin{enumerate}
\item State$(s)<$States$(s)$
\item if $T_{s+1}=T_s$ then States$(s+1)$=States$(s)$. 
\end{enumerate}

Intuitively, the state measures the progress made towards growing the tree $T$, i.e., making $T_{s+1}\ne T_s$.

$T$ is said to be in the \emph{state} State$_T(t)$ at stage $t$. We may say that \emph{the construction of $T$ ends in Step $n$ at the end of stage $t$}, or \emph{$T$ is put in state $n$ at
the end of stage $t$}, meaning that State$_T(t+1)=n$. We may also refer to numbered substeps of a step, subsubsteps and so on. These are associated with states by identifying them with elements of $\omega^{<\omega}$ and ordering lexicographically. For example, substep 2 of step 1 is labelled $\la 1,2\ra$ and precedes step 2. In general, $\alpha,\beta\in\omega^{<\omega}$ are ordered by $\alpha<\beta$ if either $\alpha\ne\beta$ and $\alpha\subseteq\beta$ as strings, or $\alpha(n)<\beta(n)$ for some $n$ with $\alpha\restrict n=\beta\restrict n$. States$(s)$ is
referred to as the number of states of $T$ at stage $s$. Any division into substeps and so on is assumed unchanged as long as $T_{s+1}=T_s$ as well.
\end{df}

\begin{df}\label{21435}
Let $k, s'<\omega$ and let $\{T_{m,t}\mid m\le k\and t\ge s'\}$ be an array of trees. Let $s^*\ge s'$ and $\beta\subset T_k$. Let $T_m=\Union\{T_{m,t}\mid t\ge s'\}$ for all $m\le k$, and let $T_{m}$
receive (see \ref{thatsubsubsection}) the set $S_{m,t+1}$ at $t+1$. Suppose $T_{k+1}=\Union \{T_{k+1,t}\mid t\ge s^*\}$ with $T_{k+1,s^*-1}=\emptyset$.

The strategy for $T_{k+1}$ defines $T_{k+1,t}$ at the end of stage $t$, and transmits pairs that will be received by $T_k$ at stage $t$.

A tree with states $T_{k+1}$ is called a \emph{Pre Phase-tree} if the following conditions hold.

\begin{enumerate}
\item \label{214i35i} For all $t\ge s^*$, $T_{k+1,t}\subseteq T_{k,t}$; and for all $t>s^*$, $T_{k+1,t}\supseteq T_{k+1,t-1}$.
\item \label{214ii35ii} $T_{k+1}$ is recursive and weakly uniform.\footnote{Lerman Def. XII.1.2 calls a tree \emph{weakly uniform} if it is length-invariant and congruence-preserving.}

\item \label{214x35xi} For all $t\ge s^*$, if $T_{k+1}$ prefers the same pair at $t$, $t+1$ (or no pair at either stage), and $t$, $t+1$ are in the same state on $T_{k+1}$, then either $T_{k+1}$ transmits the same pair at $t$, $t+1$, or transmits no pair at either stage.

\item \label{214xi35xii} For all $t\ge s^*$, if $T_{k+1}$ prefers the same pair at $t$, $t+1$ (or no pair at either stage) and $t$, $t+1$ are in different states on $T_{k+1}$, then either the state of $t$ on $T_{k+1}$ lexicographically precedes the state of $t+1$ on $T_{k+1}$ or $T_{k+1,t+1}\ne T_{k+1,t}\ne\emptyset$.

\item \label{214v35iv} For all $\alpha$ and $t\ge s^*$, if $T_{k+1}$ transmits $\la\alpha,0\ra$ at $t+1$ then $|\alpha|<$ht$(T_{k,t})$.

\item \label{214vi35v} For all $\alpha$, $i>0$ and $t\ge s^*$, if $T_{k+1}$ transmits $\la\alpha,i\ra$ at $t+1$ and $t$ and $t+1$ are in different states on $T_{k+1}$, then for all $m<k+1$, $\alpha$ is a potential focal point of $T_{m,t}$ which is not a focal point of $T_{m,t}$. 

\item \label{214ix35viii} For all $t\ge s^*$ and $\alpha$, if $T_{k+1}$ prefers the same pair at $t$ and $t+1$ (or no pair at either stage) and transmits $\la\alpha,i\ra$ with $i\le 1$ at $t$,
then $t$ and $t+1$ will be in different states on $T_{k+1}$ exactly when 
 \begin{enumerate}
 \item \label{214ixa35viiia}
 $T_{k,t}$ is a type $i$ extension of $T_{k,t-1}$ for $\alpha$ such that $\Ht(T_{k,t})=$ht$(T_{0,t})$.
 \end{enumerate}

\item \label{214iv35iii} For all $t\ge s^*$ there is at most one pair $\la\alpha,i\ra$ such that
$T_{k+1}$ transmits $\la\alpha,i\ra$ at $t+1$. For this pair, $\alpha\subset T_{k,t}$.

\item \label{214xii35ix}
For all $t\ge s^*$, if $T_{k+1,t+1}\ne T_{k+1,t}$ then
ht$(T_{k+1,t+1})$=ht$(T_{k,t})$=ht$(T_{0,t})$ and $T_{k+1}$ does not transmit any pairs at $t+1$.

\end{enumerate}
\end{df}

\begin{df}\label{214}
With notation as in Definition \ref{21435}, $T_{k+1}$ is called a \emph{Phase-1 tree} at stage $t+1$ if $T_{k+1,t}=\emptyset$, $T_{k+1}$ is a Pre Phase-tree and receives no strings, and moreover the following conditions hold.
\begin{enumerate}
\item \label{214xii} For all $t\ge s^*$, if $T_{k+1,t+1}\ne T_{k+1,t}=\emptyset$ then $|T_{k+1,t+1}(\emptyset)|=\Ht(T_{k+1,t+1})$.

\item \label{214iii} For all $t\ge s^*$, if $T_{k+1,t}(\emptyset)\downarrow=T_{k,t}(\xi)$, then $T_{k+1}$ has a fixed Ext or Phase-2 tree designation (see below) at each stage $u\ge t+1$ (unless $T_{k+1}$ is cancelled).

\item \label{214viii} For all $t\ge s^*$, if $T_{k+1,t}\ne\emptyset$ then $|T_{k+1,t}(\emptyset)|>|T_{k,t}(\emptyset)|$.

\item \label{214vii} For all $t\ge s^*$, if $T_{k+1,t}=\emptyset$ and $T_{k+1}$ transmits $\la\alpha,i\ra$ at $t$ with $i>1$ then $|\alpha|>|T_{k,t-1}(\emptyset)|$. 

\item \label{214xiii} For all $t\ge s^*$, if $T_{k+1,t}=\emptyset$ then $T_{k+1}$ transmits a pair at $t$ $\Iff$ $T_{k,s^*-1}(\emptyset)\downarrow$ and $\Ht(T_{k,s^*-1})=$ht$(T_{0,s^*-1})$. 

\end{enumerate}
\end{df}

\begin{df}\label{35}
With notation as in Definition \ref{21435}, for all $t+1\ge s^*$, define $\alpha^*(t+1)$ (if it exists) as follows. Suppose $T_{k+1}$ prefers a pair $\la\alpha,i^*\ra$. Let $\alpha^*=\alpha$ if
$i^*=1$, and let $\alpha^*$ be the longest focal point of $T_k$ contained in $\alpha$, if $i^*=0$ and such a focal point exists. Suppose also given a parameter string $\beta$, which $T_{k+1}$ is
said to be \emph{working above}.

$T_{k+1}$ is called a \emph{Phase-2 tree} if $T_{k+1}$ is a Pre Phase-tree and moreover the following conditions hold.
\begin{enumerate}
\item \label{35vii}

For all $t\ge s^*$, $T_{k+1,t}(\emptyset)\downarrow$ $\Iff$ $|\beta|=$ht$(T_{0,s^*-1})$ and $T_{k+1,t}(\emptyset)=T_{k+1,s^*}(\emptyset)\downarrow=\beta\subset T_{k,s^*-1}\subseteq
T_{k,t-1}$.

\item \label{35xiii} For all $t>s^*$, if $T_{k+1,t}=T_{k+1,t-1}$, reception of pairs by $T_{k+1}$ at $t$ satisfies \ref{23}, and $T_{k+1}$ prefers a pair at $t$ (so by \ref{23}\ref{23i} $T_{k+1,t}\ne\emptyset$) then $T_{k+1}$ transmits a pair at $t$.

\item \label{35x} $T_{k+1}$ transmits a pair at $t+1$ only if $T_{k+1}$ prefers a pair at $t+1$. 

\item \label{35ix} For all $t\ge s^*$, if $T_{k+1,t+1}\ne T_{k+1,t}\ne\emptyset$, then $T_{k+1}$ prefers some $\la\alpha,i\ra$ at $t+1$, $T_{k+1,t+1}$ is a type $i$ extension of $T_{k+1,t}$ for $\alpha$, and for all $\delta\subset T_{k+1,t+1}-T_{k+1,t}$, $\delta\supseteq\alpha^*(t+1)\downarrow$.

\item \label{35iii} For all $t\ge s^*$, if $T_{k+1}$ transmits $\la\alpha,i\ra$ at $t+1$ then $\alpha^*(t+1)\subseteq\alpha$, and if $i=0$ then $\alpha^*(t+1)=\alpha$.

\item \label{35iv} For all $\alpha$ and $t\ge s^*$, if $T_{k+1}$ transmits $\la\alpha,0\ra$ at $t+1$ and $T_{k+1,t}\ne\emptyset$ then $|\alpha|<$ht$(T_{k+1,t})\le$ht$(T_{k,t})$.
\end{enumerate}

A \emph{Phase tree} is a tree which is either Phase 1 or Phase 2.
\end{df}

\subsubsection{Construction template}

\begin{df}[Priority tree]\label{41} $\gamma<_L\delta\iff(\exists x)(\gamma\restrict x=\delta\restrict x\and\gamma(x)<\delta(x))$.

$\gamma<\delta\iff\gamma<_L\delta\text{ or }\gamma\subset\delta$ (where $\gamma\subset\delta\iff\gamma\subseteq\delta\and\gamma\ne\delta$). If $\gamma<\delta$ then we say that $\gamma$ has higher priority than $\delta$.

$\gamma<^*\delta\iff\gamma<_L\delta\text{ or }\gamma\supset\delta$.
\end{df}

\begin{df}\label{42}
Let $\beta,\delta$ be given such that $\beta\subseteq\delta$. Let $s<\omega$ be given such that $T_{\delta,s}=\emptyset$ and $T_{\delta,s}$ is designated as a Phase-1 tree. The sequence
$\{\la\sigma_\gamma,i_\gamma\ra\mid \beta\subseteq\gamma\subset\delta\}$ (which $=\emptyset$ if $\beta=\delta$, a case which is not ruled out in the following) is a \emph{transmission sequence at stage $s+1$} if the following conditions hold:

\begin{enumerate}
\item \label{42i} If $\beta\ne\delta$ then $T_{\delta,s}$ transmits $\la\sigma_{\delta^-},i_{\delta^-}\ra$ and $i_{\delta^-}\le 1$.
\item \label{42ii} For all $\gamma$ such that $\beta\subset\gamma\subset\delta$, if $T_{\gamma,s}$ is designated as a Phase-2 tree, then:
 \begin{enumerate}
 \item \label{42iia} $T_{\gamma,s}$ transmits $\la\sigma_{\gamma^-},i_{\gamma^-}\ra$ and
 $i_{\gamma^-}\le 1$.
 \item \label{42iib} $T_{\gamma,s}$ prefers $\la\sigma_\gamma,i_\gamma\ra$ and $\sigma_\gamma\subset
 T_{\gamma,s}$.
 \end{enumerate}
\item \label{42iii} For all $\gamma$ such that $\beta\subset\gamma\subset\delta$, if $T_{\gamma,s}$ is designated as an Ext tree, then $T_{\gamma,s}$ satisfies (iia) and (iib) and
$\la\sigma_{\gamma^-},i_{\gamma^-}\ra=\la\sigma_\gamma,i_\gamma\ra$.

The transmission sequence $\{\la\sigma_\gamma,i_\gamma\ra\mid \beta\subseteq\gamma\subset\delta\}$ at stage $s+1$ is a \emph{triggering sequence at stage $s+1$} if it satisfies:
\item \label{42iv} Either $\beta=\emptyset$ or for some $\xi$ such that $\beta\subseteq\xi\subseteq\delta$, $s$ and $s+1$ are in different states on $T_\xi$ and $T_{\beta,s}$ does not transmit any pair $\la\alpha,i\ra$ with $i\le 1$.
\end{enumerate}
We say that \emph{$T_\delta$ triggers $T_\beta$ at stage $s+1$} if there is a triggering sequence $\{\la\sigma_\gamma,i_\gamma\ra\mid \beta\subseteq\gamma\subset\delta\}$ at stage $s+1$. We call $T_\delta$ a \emph{trigger at stage $s+1$} if there is a $\beta\subseteq\delta$ such that $T_\delta$ triggers $T_\beta$ at stage $s+1$.
\end{df}

\begin{df}\label{43-}
\noindent {\bf The Construction.} \emph{Stage 0:} Designate $T_{\emptyset,0}$ as $\Init(\{S_{\emptyset,t}\mid t>0\})$. Let $\alpha_0=\gamma_0=\emptyset$.

To say that a node $\gamma$ on the priority tree acts means that before doing anything else we define $T_{\gamma,s+1}$ and define the transmission (if any) of the dynamical tree $T_\gamma$ at stage $s+1$, given the trees $T_{\delta,s}$ for $\delta\subseteq\gamma$ and given $S_{\gamma,t+1}$.

\noindent \emph{Stage $s+1$:}

Let the $<^*$-least designated non-cancelled node on the priority tree act. Inductively, after the $n$th-$<^*$-least node has acted, let the $n+1$th-$<^*$-least node act, if one exists.

This is done subject to the constraint that if a node $\gamma*a$ transmits a pair $\la\alpha,i\ra$ with $i\le 1$ at stage $s+1$ then all strategies $\gamma*b$ with $a<b$ and their descendants are
cancelled before another node acts.

For all $\beta$ such that $T_{\beta,s}$ is designated, the information received by $T_{\beta,s}$ at stage $s+1$ is the information transmitted by those $T_{\delta,s}$ such that $\beta=\delta^-$ at
stage $s+1$. $S_{\lambda,t+1}$ is defined to be what $\lambda$ receives from trees (equivalently, from non-cancelled trees) at stage $t+1$.

After all designated, noncancelled trees have acted, or a trigger is discovered, designate new trees as follows. 

\noindent Case 1. There is a trigger at stage $s+1$.

Fix the tree $T_\delta$ of highest priority such that $T_\delta$ is a trigger at stage $s+1$, fix $\beta\subseteq\delta$ such that $T_\delta$ triggers $T_\beta$ at stage $s+1$, and let
$\{\la\sigma_\gamma,i_\gamma\ra\mid \beta\subseteq\gamma\subset\delta\}$ be the corresponding triggering sequence. Note that there is only one possible choice for $T_\beta$.

If $T_{\beta,s}$ is designated as a Phase-2 tree or if $\beta=\delta$, let $\la\alpha^*,i^*\ra$ be the information transmitted by $T_{\beta,s}$ if any information is transmitted. Note that if $i^*$
is defined, then $i^*\ge 2$.

\begin{itemize}
\item \emph{Subcase 1. New Ext tree. $i^*\ne 3$ (so $i^*\in\{2,4,5,\ldots\}$).} Let $\alpha_{s+1}=\alpha^*=T_{\beta^-,s}(\xi^*)$. Let $\gamma_{s+1}=s(\beta)$ and designate $T_{\gamma_{s+1}}$ as the following extension tree:
\[ 
	T_{\gamma_{s+1}}=\Ext(T_{\beta^-},\xi^*,s+1). 
\]
Begin building $T_{\gamma_{s+1}}$.

\item \emph{Subcase 2. New Phase-2 tree. $i^*=3$.}
Let $\alpha_{s+1}=\alpha^*=T_{\beta^-,s}(\xi^*)$. Let $\gamma_{s+1}=s(\beta)$. Designate and begin building $T_{\gamma_{s+1}}$ as a Phase-2 tree working above $\alpha^*$.

\item \emph{Subcase 3. Phase-1 tree becomes nonempty.} $\beta=\delta$ and $T_{\delta,s}$ has no transmission.

Let $\alpha_{s+1}=T_{\delta,s+1}(\emptyset)$ and $\gamma_{s+1}=\delta$.

\item \emph{Subcase 4. Phase-2 or Init tree grows.} Otherwise. Then $T_{\beta,s}$ will be designated either as the initial tree or as a Phase-2 tree, and has no transmission. Let $\alpha_{s+1}=\sigma_\beta$ and $\gamma_{s+1}=\delta$.

\end{itemize}

\noindent \emph{Case 2.} There is no trigger at stage $s+1$.

Let $\gamma_{s+1}=\gamma_s*0$, $\alpha_{s+1}=T_{\gamma_s}(\emptyset)$. Designate a new tree $T_{\gamma_{s+1}}$ as a Phase tree.

\begin{itemize}
\item \emph{Subcase 1. $|\gamma_s|$ is even. New Phase-2 tree.} Let $\alpha^*=\alpha_{s+1}=T_{\gamma_s,s}(\emptyset)$ and $\gamma_{s+1}=\gamma_s*0$. Designate and begin building $T_{\gamma_{s+1}}$, working above $\alpha_{s+1}$.

\item \emph{Subcase 2. $|\gamma_s|$ is odd. New Phase-1 tree.}
Let $\alpha_{s+1}=T_{\gamma_s}(\emptyset)$, $\gamma_{s+1}=\gamma_s*0$ and designate and begin building $T_{\gamma_{s+1}}$.
\end{itemize}

In all cases, we say that $T_{\beta,s+1}$ is \emph{newly designated} if $T_{\beta,s+1}$ is designated and either $T_{\beta,s}$ is not designated or $T_{\beta,s}$ is cancelled at stage $s+1$.

\noindent {\bf End of construction.}
\end{df}

\begin{df}\label{lermanconstruction}
A recursive construction of the form of \ref{43-} with a finitely branching priority tree (i.e., at each level there is a number $n$ such that at each stage, the branching at that level is less than
$n$) is called a \emph{Lerman construction}. The requirement that the construction be recursive implies that there is an algorithm that decides at any point in the construction what should be the
index of the new Phase tree to be designated.
\end{df}

%% file: PRESEN-22.tex
In this Subsection we establish properties of an arbitrary Lerman construction as defined in Subsection 3.1. The results of this Section are due to Lerman \cite{Lerman:83} and are included here for the sake of completeness.

\subsubsection{Tree growth and state changes}

\begin{rem}[Timing]
Note that if $T_{i,s+1}\ne T_{i,s}$ then for each $j<i$, $T_{j,s+1}= T_{j,s}$, since $T_i$ has no transmission at stage $s+1$ and hence none of the $T_j$ prefer a pair at stage $s+1$. So in particular $\Ht(T_{j,s+1})=\Ht(T_{0,s})$, although each strategy $T_i$ can only ensure that it gets $\Ht(T_{i-1,s})$. Also note that $T_{\lambda,t}$ is defined at the end of stage $t$, and then the triggering takes place at stage $t+1$. Note that ``$T$ is designated at (or during) stage $t+1$'' means, in the beginning of the stage, and ``newly designated at stage $t+1$'' means, was first
given a designation at the end of stage $t$. ``Cancelled at $t+1$'' means designated during $t+1$, but cancelled at the end of (i.e., during) $t+1$. So cancelled at $t+1$ implies not designated at
$t+2$ unless newly designated at $t+2$.
\end{rem}

\begin{lem}\label{44}
\begin{enumerate}
\item \label{44i} 
	If $T_\lambda$ receives $\la\alpha,i\ra$ with $i\le 1$ at $t+1$, then there is an $\eta\supset\lambda$ and a transmission sequence $\{\la\sigma_\beta,i_\beta\ra\mid\lambda\subset\beta\subseteq\eta\}$ at stage $t+1$.
\item \label{44ii} 
	If $\xi^-=\lambda^-$, $T_\lambda$ is not cancelled at stage $t+1$, $\xi$ has higher priority than $\lambda$, and $T_\xi$ transmits $\la\alpha,i\ra$ at $t+1$, then $i\ge 2$.
\item \label{44iii} 
	If $T_\lambda$ is not cancelled at stage $t+1$, $\la\alpha,i\ra\in S_{\lambda,t+1}$ and $i\le 1$, then $T_\lambda$ prefers $\la\alpha,i\ra$ at $t+1$.
\item \label{44iv} 
	Suppose that $T_\delta$ is designated but not cancelled at stage $t$, that
	\[
		S=\{\la\sigma_\beta,i_\beta\ra\mid\lambda\subseteq\beta\subset\delta\}
	\]
	is a transmission sequence at stage $t$, and that for all $\xi$ such that $\lambda\subset\xi\subseteq\delta$, $t$ and $t+1$ are in the same state on $T_\xi$. Then $S$ is a transmission sequence at stage $t+1$.
\item \label{44v} 
	Suppose that $T_\delta$ is designated but not cancelled at stage $t$, that
	\[
		S=\{\sigma_\beta,i_\beta\ra\mid\lambda\subseteq\beta\subset\delta\}
	\]
	is a transmission sequence at stage $t$, and that $T_\lambda$ does not transmit any $\la\alpha,i\ra$ with $i\le 1$ at $t+1$. Let $\xi$ be given such that $\lambda\subset\xi\subseteq\delta$, $T_{\xi}$ is designated as a Phase tree at $t+1$, and suppose that either $t$ and $t+1$ are in different states on $T_\xi$ or $T_{\xi,t+1}\ne T_{\xi,t}$. Then $T_{\lambda,t}$ is a type $i_\lambda$ extension of $T_{\lambda,t-1}$ for $\sigma_\lambda$ and $\Ht(T_{\lambda,t})=\Ht(T_{\emptyset,t-1})$.
\item \label{44vi} 
	If $T_\delta$ is not cancelled at stage $t+1$, $S=\{\la\sigma_\beta,i_\beta\ra\mid\lambda\subseteq\beta\subset\delta\}$ is a transmission sequence at stage $t$, and $T_{\lambda,t}$ is a type $i_\lambda$ extension of $T_{\lambda,t-1}$ for $\sigma_\lambda$ with $\Ht(T_{\lambda,t})=\Ht(T_{\emptyset,t-1})$, then either $\delta=\lambda$ and $T_\lambda$ has no transmission at $t+1$, or there is a $\xi$ such that $T_\delta$ triggers $T_\xi$ at stage $t+1$.
\item \label{44vii} 
	If $T_\lambda$ is not cancelled at stage $t+1$ and $T_{\lambda,t+1}\ne T_{\lambda,t}$, then there are $\eta\supseteq\lambda$ and $\xi\subseteq\lambda$ such that $T_\eta$ triggers $T_\xi$ at stage $t+1$ and $T_\eta$ is not cancelled at stage $t+1$; and for all $\nu$ of lower priority than $\eta$, $T_\nu$ is not designated at $t+2$.
\end{enumerate}
\end{lem}

\begin{proof}
We proceed by induction on $t$. The lemma follows easily for $t=-1$.

\ref{44i} We proceed by induction on those $\lambda$ such that $T_{\lambda,t}$ is designated, with lower priority strings coming first. Since only finitely many trees are designated at a given
stage, the ordering on which the induction is carried out is well-founded.

Everything in the following discussion takes place within a single stage $t+1$. Assume that $T_{\lambda}$ receives $\la\alpha,i\ra$ with $i\le 1$. Then $T_{\lambda}$ receives and prefers some
pair $\la\sigma_\lambda,i_\lambda\ra$, and there is a $\xi$ such that $\xi^-=\lambda$ and $T_{\xi}$ transmits $\la\sigma_\lambda,i_\lambda\ra$. If $T_{\xi}$ is designated (\emph{i.e., at least at the start of the stage}) as a Phase-1 tree, then $\{\la\sigma_\lambda\ra\}$ is a transmission sequence. If $T_{\xi}$ is designated as an Ext tree at stage $t+1$, then by \ref{212}\ref{212ii}, $T_{\xi}$ receives $\la\sigma_\lambda,i_\lambda\ra$. And if $T_{\xi}$ is designated as a Phase-2 tree at stage $t+1$, then by \ref{35}\ref{35x}, $T_{\xi}$ must receive and prefer some $\la\beta,j\ra$ with
$j\le 1$. Thus in all cases, it follows by induction that there is an $\eta\supseteq\xi$ and a transmission sequence $\{\la\sigma_\beta,i_\beta\ra\mid\xi\subseteq\beta\subseteq\eta\}$. Hence $\{\la\sigma_\beta,i_\beta\ra\mid\lambda\subseteq\beta\subset\eta\}$ will be a transmission sequence.

\ref{44ii} Immediate from the construction.

\ref{44iii} We proceed by induction on those $\lambda$ such that $T_{\lambda}$ is designated during $t+1$, with lower priority strings coming first. Everything in the following discussion takes place within a single stage $t+1$. Let $\la\alpha,i\ra\in S_{\lambda,t+1}$ be given, with $i\le 1$. Fix the lowest priority $\eta$ such that $T_{\eta}$ is designated and not cancelled, with
$\eta^-=\lambda$. By \ref{44ii}, $T_{\eta}$ transmits $\la\alpha,i\ra$ . If $T_{\eta}$ is designated as a Phase tree, then by \ref{21435}\ref{214iv35iii}, $\la\alpha,i\ra$ is the unique transmission of $T_{\eta}$, so $T_{\lambda}$ must prefer $\la\alpha,i\ra$ at $t+1$. And if $T_{\eta}$ is designated as an Ext tree, then by \ref{212}\ref{212ii}, $\la\alpha,i\ra\in S_{\eta,t+1}$, so by induction, $T_{\eta}$ prefers $\la\alpha,i\ra$. Hence $T_{\lambda}$ must prefer $\la\alpha,i\ra$. (Since $\la\alpha,i\ra$ was arbitrary, we have shown that there is at most one $\la\alpha,i\ra\in S_{\lambda,t+1}$ such that $i\le 1$.)

\ref{44iv} An induction argument on $\{\xi\mid\lambda\subset\xi\subseteq\delta\}$, longer strings first, using \ref{44iii}, \ref{212}\ref{212ii}, \ref{214}\ref{214iii} and \ref{21435}\ref{214x35xi}
shows that for all $\xi$ such that $\lambda\subset\xi\subseteq\delta$, $T_{\xi}$ transmits the same pair at stages $t$ and $t+1$. Hence $S$ is a transmission sequence at stage $t+1$.

\emph{Remark:} (ii), (iii), (iv) are quite clear once we observe that if two siblings $\gamma<_L\delta$ both try to transmit type $\le 1$ messages at $t+1$ and did not both do so at
$t$, then $\gamma$ just changed state and so there will be a triggering sequence whose triggered node $\eta$ satisfies $\eta\le\gamma$.

\ref{44v} Fix the longest $\eta$ such that $\lambda\subseteq\eta\subset\xi$ and $T_{\eta}$ is designated as an Init or Phase-2 tree at $t$. Note that by choice of $\xi$, $T_{\lambda}$ is
designated as an Init or Phase-2 tree at $t$, so $\eta$ exists. By \ref{21435}\ref{214xii35ix}, $T_{\xi,t}=T_{\xi,t-1}$. By \ref{21435}\ref{214ixa35viiia} and \ref{212}\ref{212ii}, $T_{\eta,t}$
is a type $i_\xi$ extension of $T_{\eta,t-1}$ for $\sigma_\xi$ and $\Ht(T_{\eta,t})=\Ht(T_{\emptyset,t-1})$. Furthermore, $\eta=\lambda$, else $\la\sigma_\eta,i_\eta\ra$ would be defined, and so by \ref{21435}\ref{214xii35ix}, we would have $T_{\eta,t}=T_{\eta,t-1}$.

Informally: Since $\lambda\ne\xi$, there is an immediate (except for Ext trees) ancestor of $\xi$ whose growth led to the change in $\xi$, and so this ancestor must have responded to a request and so does not transmit a pair and hence must in fact be $\lambda$.

(Note that for an Ext tree, growth does not imply state change.)

\ref{44vi} If $\lambda=\delta$, then by \ref{21435}\ref{214xii35ix}, $T_{\lambda}$ has no transmission at $t+1$. Assume that $\lambda\ne\delta$, and fix the shortest $\eta\supset\lambda$
for which $T_\eta$ is designated as a Phase tree. By \ref{42}\ref{42iii}, $T_{\eta}$ transmits $\la\sigma_\lambda,i_\lambda\ra$ at $t$, so by \ref{21435}\ref{214ixa35viiia}, \ref{212}\ref{212ii}
and since $T_{\eta}$ is not cancelled at stage $t+1$, $t$ and $t+1$ are in different states on $T_\eta$ as long as $T_{\eta}$ receives $\la\sigma_\eta,i_\eta\ra$ at $t+1$. Hence there is a
longest $\nu$ such that $\eta\subseteq\nu\subseteq\delta$ and $t$ and $t+1$ are in different states on $T_\nu$. Hence for some $\xi\subseteq\nu$, $T_\delta$ triggers $T_\xi$ at stage $t+1$.

Informally: If a tree grows in response to a request then the tree originating that request changes state, and hence is part of a triggering sequence.

Note that \ref{44vi} says that growth leads to a trigger at the next stage because a child of the growing tree will be able to change state (unless the growing tree has no child).

\ref{44vii} Fix the longest $\xi\subseteq\lambda$ such that $T_{\xi}$ is designated as a non-Ext tree at $t+1$. By \ref{42}\ref{42iii}, $T_{\xi,t+1}\ne T_{\xi,t}$.

We note that by \ref{35}\ref{35vii}, if $T_{\xi,t}\ne\emptyset$ and $r$ is the greatest stage $\le t$ such that $T_{\xi}$ is newly designated as a Phase-2 tree at $r+1$, then $T_{\xi,r}\ne\emptyset$ (else $T_{\xi,t}=T_{\xi,t+1}=\emptyset$). If $T_{\xi}$ is not designated as a Phase-1 tree at $t+1$, then by \ref{35}\ref{35ix} and \ref{28}\ref{28i}, $T_{\xi}$ must receive and prefer some $\la\alpha,i\ra\in S_{\xi,t+1}$ with $i\le 1$ at $t+1$. By \ref{44i}, there is a highest priority $\eta\supseteq\xi$ and a transmission sequence $\{\la\sigma_\beta,i_\beta\ra\mid\xi\subseteq\beta\subset\eta\}$ at $t+1$. Since $T_{\xi,t+1}\ne T_{\xi,t}$, it follows from \ref{21435}\ref{214xii35ix} that $T_{\xi}$ has no transmission at $t+1$, and $t$ and $t+1$ are in different states on $T_\xi$. Hence $T_\eta$ triggers $T_\xi$ at stage $t+1$. $T_{\eta}$ cannot have been cancelled at stage $t+1$, else $\la\alpha,i\ra\not\in S_{\xi,t+1}$. Hence $T_\eta$ is the unique trigger at stage $t+1$, and Case 1, Subcase 3 or Subcase 4 of \ref{43} is followed at stage $t+1$. But then if $\nu$ has lower priority than $\eta$, then $T_{\nu}$ is cancelled at stage $t+1$, and $T_{\nu}$ cannot be newly designated at stage $t+2$. Hence $T_{\nu}$ cannot be designated at $t+2$ if $\nu$ has lower priority than $\eta$.

Informally: This says that not only is a tree that grows in response to a request part of a triggering sequence, but in fact any tree that grows at $t+1$ if it was designated at $t$ (though newly designated Phase-2 trees may have nonempty domains without being triggered) grows in response to a request and hence was triggered. It is not hard to see that (non-Ext) trees only grow in
response to requests in these cases; for Phase-2 trees and Init trees it is immediate, and for Phase-1 trees the growth is in response to a request from itself.
\end{proof}

\subsubsection{Transmissions}

This lemma essentially follows from \ref{44}.

\begin{lem}\label{45}
Fix $\lambda$, $\sigma$, $t\in\omega$ and $i\le 4$ such that $\lambda\ne\emptyset$ and $T_\lambda$
is designated and not cancelled during $t+1$. Then:
\begin{enumerate}
\item \label{45i}
	If $T_\lambda$ transmits $\la\sigma,i\ra$ at $t+1$, then $\sigma\subset T_{\lambda^-,t}$. If, in addition, $T_{\lambda,t}(\emptyset)\downarrow$, then $\sigma\supseteq T_{\lambda,t}(\emptyset)$.
\item \label{45ii}
	If $\lambda\ne\lambda^-*0$ then $T_{p(\lambda)}$ transmits a unique pair $\la\sigma,i\ra$ at $t+1$. For this pair, $i\ge 2$ and $T_{\lambda,t}(\emptyset)\downarrow=\sigma$
\item \label{45iii}
	If $\lambda\ne\lambda^-*0$ and $T_{p(\lambda),t}(\emptyset)\downarrow$, then $T_{\lambda,t}(\emptyset)\supset T_{p(\lambda),t}(\emptyset)$.
\item \label{45iv}
	For all $\xi$, if $\xi^-=\lambda^-$, $\xi$ has higher priority than $\lambda$, and $T_\xi$ transmits $\la\sigma,i\ra$ at $t+1$, then $T_{\lambda,t}(\emptyset)\supseteq\sigma$ and $i\ge 2$.
\item \label{45v}
	If $\la\beta,j\ra$, $\la\gamma,k\ra\in S_{\lambda,t+1}$ and $\la\beta,j\ra\ne\la\gamma,k\ra$, then $\beta\subseteq\gamma$ or $\gamma\subseteq\beta$. If, in addition, $k\le 1$, then
$\beta\subseteq\gamma$ and $j\ge 2$.
\item \label{45vi}
	If $T_\lambda$ receives $\la\sigma,0\ra$ at $t+1$ then $|\sigma|<\Ht(T_{\lambda,t})$.
\item \label{45vii}
	If $T_\lambda$ receives $\la\sigma,0\ra$ at $t+1$, $\beta\subseteq\sigma$ and $\beta$ is a potential focal point of $T_{\lambda,t}$, then $\beta$ is a focal point of $T_{\lambda,t}$.
\item \label{45viii}
	If $T_\lambda$ is newly designated at $t+1$, then dom$(T_{\lambda,t})\subseteq\{\emptyset\}$; and if $T_{\lambda,t}=\emptyset$, then either $T_\lambda$ is designated as Phase-1 tree at $t+1$ or $T_{\lambda^-,t}=\emptyset$.
\end{enumerate}
\end{lem}

\begin{proof}
We proceed by induction on $t$ and then by induction on those $\lambda$ such that $T_{\lambda,t}$
is designated, with lower priority $\lambda$ considered first.

\ref{45i} If $T_{\lambda}$ is designated as a Phase-1 tree at $t+1$, then \ref{45i} follows from Lemma \ref{21435}\ref{214iv35iii}. If $T_{\lambda}$ is designated as an Ext tree during $t+1$, then
by \ref{212}\ref{212ii} there is an $\eta$ such that $\eta^-=\lambda$ and $T_{\eta}$ transmits $\la\sigma,i\ra$ at $t+1$. Applying \ref{45i} by induction to $\eta$, we see that $\sigma\subset
T_{\lambda,t}\subseteq T_{\lambda^-,t}$ so $\sigma\supseteq T_{\lambda,t}(\emptyset)$. Finally, if $T_{\lambda}$ is designated as a Phase-2 tree during $t+1$, then \ref{45i} follows from
\ref{21435}\ref{214iv35iii} and \ref{35}\ref{35iii} since $T_{\lambda,t}(\emptyset)\downarrow\subseteq\alpha^*(t)$.

\ref{45ii} First consider the case where $T_{\lambda}$ is newly designated at $t+1$. Then Case 1, Subcase 1 or Subcase 2 of the construction of \ref{43} is followed at stage $t$, and by \ref{45i},

$T_{p(\lambda)}$ transmits a unique pair $\la\sigma,i\ra$, $i\ge 2$ at $t$, $T_{\lambda,t}(\emptyset)\downarrow=\sigma$, and there is a $\delta\supseteq p(\lambda)$ such that $T_\delta$ triggers $T_{p(\lambda)}$ at $t$ and $T_\delta$ is not cancelled at $t$. Let $S$ be the transmission sequence from $T_\delta$ to $T_{p(\lambda)}$ at $t$. By \ref{44}\ref{44v}, \ref{44}\ref{44vii} and \ref{21435}\ref{214ix35viii}, $t$ and $t+1$ are in the same state on $T_\eta$ and $T_{\eta,t}=T_{\eta,t-1}$ for all $\eta$ such that $p(\lambda)\subseteq\eta\subseteq\delta$. Hence by \ref{21435}\ref{214x35xi}, $T_{p(\lambda)}$ also transmits $\la\sigma,i\ra$ at $t+1$. Hence \ref{45ii} holds.

Now assume that $T_{\lambda}$ is not newly designated at $t+1$. Since $T_{\lambda}$ is designated at $t+1$, $T_{\lambda}$ is designated and not cancelled at $t$. We assume by induction that $S$ is still a transmission sequence at $t-1$. By \ref{44}\ref{44v}, \ref{44}\ref{44vii} and \ref{21435}\ref{214ix35viii}, $t$ and $t-1$ are in the same state on $T_{p(\lambda)}$; and by
\ref{21435}\ref{214x35xi} and \ref{44}\ref{44iv}, $S$ is a transmission sequence at stage $t$, and $T_{p(\lambda),t}$ transmits $\la\sigma,i\ra$ at $t+1$.

\ref{45iii} By \ref{45ii}, $T_{p(\lambda}$ has a unique transmission $\la\sigma,2\ra$ or $\la\sigma,3\ra$ at $t+1$ and $T_{\lambda,t}(\emptyset)=\sigma$. By \ref{45i}, $\sigma\supset T_{p(\lambda),t}(\emptyset)$.

\ref{45iv} Define $s^1(\xi)=s(\xi)$ and $s^{k+1}(\xi)=s(s^k(\xi))$. There must be a $k$ such that $\lambda=s^k(\xi)$, and for all $j\le k$, $T_{s^j(\xi)}$ must be designated at $t+1$. By \ref{45ii}, $T_{s(\xi,t}(\emptyset)=\sigma$ and $i\ge 2$. Iterating \ref{45iii}, we see that $T_{\lambda,t}(\emptyset)=T_{s^k(\xi),t}(\emptyset)\supseteq\cdots\supseteq T_{s(\xi),t}(\emptyset)=\sigma$.

\ref{45v} Choose $\xi,\eta$ such that $\xi^-=\eta^-=\lambda$, $T_{\xi}$ transmits $\la\beta,j\ra$ at $t+1$, and $T_{\eta}$ transmits $\la\gamma,k\ra$ at $t+1$. If $k\le 1$, then by \ref{45iv},
either $\xi$ has higher priority than $\eta$ and $j\ge 2$, or $\xi=\eta$. Assume first that $\xi$ has higher priority than $\eta$. (A similar proof will work if $\eta$ has higher priority than $\xi$.) Then $\eta\ne\eta^-*0$. By \ref{45i} and \ref{45iv}, $\beta\subseteq T_{\eta,t}(\emptyset)\subseteq\gamma$. Finally, assume that $\xi=\eta$. Then by \ref{21435}\ref{214iv35iii}, $T_{\xi}$ cannot be designated as a Phase-2 or Phase-1 tree at $t+1$. Hence $T_{\xi}$ is designated as an Ext tree at $t+1$. By \ref{212}\ref{212ii}, $T_{\xi}$ receives both $\la\beta,j\ra$ and $\la\gamma,k\ra$ at $t+1$. Hence \ref{45v} follows by induction.

\ref{45vi} There is a $\xi$ such that $\xi^-=\lambda$ and $T_{\xi}$ transmits $\la\sigma,0\ra$ at $t+1$. If $T_{\xi}$ is designated as an Ext tree at $t+1$, then by \ref{212}\ref{212ii} and
induction on $\lambda$, $|\sigma|<\Ht(T_{\xi,t})\le\Ht(T_{\lambda,t})$. If $T_{\xi}$ is designated as a Phase tree at $t+1$, then by \ref{21435}\ref{214v35iv}, $|\sigma|<\Ht(T_{\lambda,t})$.

\ref{45vii} A potential focal point $\beta$ of will be a focal point of $T_{\lambda,t}$ iff $|\beta|<\Ht(T_{\lambda,t})$. Let $T_{\lambda}$ receive $\la\sigma,0\ra$ at $t+1$, and let $\beta\subseteq\sigma$ be a potential focal point of $T_{\lambda,t}$. By \ref{45vi}, $|\sigma|<\Ht(T_{\lambda,t})$. Since $\beta\subseteq\sigma$, $|\beta|<\Ht(T_{\lambda,t})$. Hence
$\beta$ is a focal point of $T_{\lambda,t}$.

\ref{45viii} Let $T_{\lambda}$ be newly designated at $t+1$.

First assume that Case 1, Subcase 1 or Subcase 2 of the construction of \ref{43} is followed at stage $t$. Then $T_{p(\lambda)}$ is designated either as a Phase-1 tree or as a Phase-2 tree at $t$ and is not cancelled at $t$; and there is a $\delta\supseteq p(\lambda)$ such that $T_\delta$ triggers $T_{p(\lambda)}$ at $t$ and $T_\delta$ is not cancelled at $t$. By \ref{42}\ref{42iv},
\ref{44}\ref{44v} and \ref{21435}\ref{214xii35ix}, $t$ and $t-1$ must be in different states on $T_{p(\lambda)}$. Hence by \ref{45i}, there is a pair $\la\sigma,i\ra$ transmitted by $T_{p(\lambda)}$ at $t$ such that $\sigma\subset T_{\lambda^-,t-1}$, so $T_{\lambda,t}(\emptyset)\downarrow=\sigma$ (indeed, if $\sigma$ was not $\subset T_{\lambda^-,t}$ then by definition of Ext tree we would have $T_{\lambda,t}(\emptyset\uparrow$). By \ref{21435}\ref{214xii35ix} (first equality) and \ref{44}\ref{44vii} (second equality), $|\sigma|=\Ht(T_{\lambda^-,t-1})=\Ht T_{\lambda^-,t})$. Hence (again by definition of Ext tree) dom$(T_{\lambda,t})=\{\emptyset\}$, and so in this case $T_{\lambda,t}\ne\emptyset$.

The other possibility is that Case 2 of \ref{43} is followed at $t$.

\begin{itemize}
\item Case 2, Subcase 1: $T_{\lambda}$ is a Phase-2 tree at $t+1$, $\gamma_{t-1}=\lambda^-$ is originally a Phase-1 tree, and $\gamma_t=\lambda$.

For this subcase 2.1 there are two cases:

\begin{itemize}
\item Case A: $T_{\lambda^-}$ is newly designated at $t$. Then by induction on $t$,
\[
	\dom(T_{\lambda^-,t-1})\subseteq\{\emptyset\},
\]
and so dom$(T_{\lambda,t})\subseteq\{\emptyset\}$, since $T_{\lambda,t}\subseteq T_{\lambda^-,t-1}$ (since $T_\lambda$ is not an Ext tree). And if $T_{\lambda^-}\ne\emptyset$ then $T_{\lambda}\ne\emptyset$ since $|T_{\lambda^-}(\emptyset)|=\Ht(T_{\emptyset})$ by \ref{21435}\ref{214xii35ix}.

\item Case B: $T_{\lambda^-}$ was triggered at $t-1$. Then $T_{\lambda^-}$ had no children ($T_\lambda$ is newly designated) so it must have been a Phase-1 tree triggered by itself, so by
\ref{21435}\ref{214xii35ix}, dom$(T_{\lambda^-,t-1}\subseteq\{\emptyset\}$, and so dom$(T_{\lambda,t})\subseteq\{\emptyset\}$, since $T_{\lambda,t}\subseteq T_{\lambda^-,t-1}$ (since
$T_\lambda$ is not an Ext tree). We now need to show that if $T_{\lambda^-,t}\ne\emptyset$ then $T_{\lambda,t}\ne\emptyset$. By \ref{44}\ref{44vii}, $|\beta|=\Ht(T_{\emptyset,s^*-2})=\Ht(T_{\emptyset,s^*-1})=\Ht(T_{\emptyset,s^*})$. Now apply \ref{35}\ref{35vii}.
\end{itemize}

\item Case 2, Subcase 2: $T_{\lambda}$ is a newly designated Phase-1 tree at $t+1$. Then we only need to show the first part of \ref{45viii}.

\begin{itemize}
\item Case A: $T_{\lambda^-}$ newly designated at $t$: By induction on $t$ as above.

\item Case B: $T_{\lambda^-}$ was triggered at $t-1$: This is impossible since $T_{\lambda^-}$ is not a Phase-1 tree and has no Phase-1 descendants at $t$ (since $T_{\lambda}$ is newly designated at $t+1$).
\end{itemize}

\end{itemize}

\end{proof}

\subsubsection{Focal points and monotonicity}

\begin{lem}\label{46} Let $\lambda$ and $t\in\omega$ be given such that $T_\lambda$ is designated at $t+1$. Then:
\begin{enumerate}
\item \label{46i} If $\delta$ has lower priority than $\lambda$, $T_\delta$ is not cancelled at stage $t+1$, $\{\la\sigma_\beta,i_\beta\ra\mid\xi\subseteq\beta\subset\delta\}$ is a transmission sequence
at stage $t+1$, $T_\xi$ transmits $\la\sigma,i\ra$ at $t+1$, and $T_{\lambda,t+1}(\emptyset)\downarrow$, then $T_{\lambda,t+1}(\emptyset)\subseteq\sigma$. 
\item \label{46ii} If $T_{\xi,t+1}(\emptyset)\downarrow$, $\xi$ has lower priority than $\lambda$, and $T_{\lambda,t+1}(\emptyset)\downarrow$, then $T_{\lambda,t+1}(\emptyset)\subseteq
T_{\xi,t+1}(\emptyset)$.
\item \label{46iii} Reception of pairs by $T_\lambda$ at $t+1$ is appropriate, i.e., satisfies \ref{23}.
\item \label{46iv} If for all $\xi\subseteq\lambda$, $\beta$ is a potential focal point of $T_{\xi,t}$, and for all $\la\alpha,i\ra\in S_{\lambda,t+1}$, either $i\ge 2$ or $\beta\subseteq\alpha$, then $\beta$ is a potential focal point of $T_{\lambda,t+1}$.
\item \label{46v} $\{T_{\xi,t+1}\mid\xi\subseteq\lambda\and t\le s\le t+1\}$ is special.
\item \label{46vi} If $T_{\lambda,t+1}\ne\emptyset$, then for all $\xi\subseteq\lambda$, $T_{\lambda,t+1}(\emptyset)$ is a potential focal point of $T_{\xi,t+1}$.
\end{enumerate}
\end{lem}

\begin{proof}
We proceed by induction on $t$. The lemma follows easily for $t=-1$.

To unravel the logical structure of the proof the reader should read first \ref{46iv}, then \ref{46v}, then \ref{46iii}. Then \ref{46i}, \ref{46ii}, \ref{46vi} are proved in the following way: Assuming \ref{46ii}, \ref{46vi} at stage $t$, we prove \ref{46i} at stage $t+1$, and assuming \ref{46i} at stage $t+1$ we prove \ref{46ii}, \ref{46vi} at stage $t+1$.

\ref{46i} Since $S$ is a transmission sequence, $T_{\delta}$ is designated as a Phase-1 tree at $t+1$, hence $\delta=\delta^-*0$. Thus is $\lambda$ has higher priority than $\delta$, then either
$\lambda=\delta^-$ or $\lambda$ has higher priority than $\delta$. By \ref{21435}\ref{214iv35iii}, $T_{\delta^-,t}(\emptyset)\downarrow$, else we have nothing to show. We note that
$T_{\lambda,t}(\emptyset)\downarrow$ and $T_{\lambda}$ is not cancelled at stage $t+1$, else $T_{\delta}$ would be cancelled at stage $t+1$. It thus follows inductively from \ref{46ii} that
$T_{\lambda,t+1}(\emptyset)=T_{\lambda,t}(\emptyset)\subseteq T_{\delta^-,t}(\emptyset)$. Let $\la\sigma_{\xi^-},i_{\xi^-}\ra$ be the transmission of $T_{\xi}$ at $t+1$ if such a transmission
exists. It suffices to show that for all $\beta$ such that $\xi^-\subset\beta\subseteq\delta$, $T_{\delta^-}(\emptyset)\subseteq\sigma_{\beta^-}$. We proceed by induction on
$\{\beta\mid\xi^-\subset\beta\subseteq\delta\}$, longer strings first. If $T_\beta$ is designated as a Phase-1 tree, then $\beta=\delta$, so by \ref{21435}\ref{214iv35iii},
$T_{\delta^-}(\emptyset)\subseteq\sigma_{\delta^-}$. If $T_\beta$ is designated as an Ext tree, then by induction, $T_{\delta^-,t}(\emptyset)\subseteq\sigma_\beta=\sigma_{\beta^-}$. And if
$T_\beta$ is designated as a Phase-2 tree, then by \ref{35}\ref{35iii}, $\alpha^*(t+1)\subseteq\sigma_{\beta^-}$. By \ref{46vi} and \ref{45}\ref{45vii}, $T_{\delta^-,t}(\emptyset)$ is a potential focal point of $T_{\beta,t}$ and if $i_\beta=0$, then $T_{\delta^-,t}(\emptyset)$ is a focal point of $T_{\beta,t}$. Hence $T_{\delta^-,t}(\emptyset)\subseteq\alpha^*(t+1)\subseteq\sigma_{\beta^-}$.

\ref{46ii} If $\xi\ne\gamma_{t+1}$, then $T_{\xi}$ is designated at $t+2$ iff $T_{\xi}$ is designated at $t+1$ and not cancelled at stage $t+1$. Hence if $\xi$ has higher priority than $\lambda$, then applying \ref{46ii} by induction, we see that $T_{\lambda,t+1}(\emptyset)=T_{\lambda,t}(\emptyset)\subseteq T_{\xi,t}(\emptyset)=T_{\xi,t+1}(\emptyset)$. Assume that $\xi=\gamma_{t+1}$. If Case 2 of the construction of \ref{43} is followed at stage $t+1$, then $\xi=\xi^-*0$, so either $\xi^-=\lambda$ or $\lambda$ has higher priority than $\xi^-$. By induction and since $T_{\xi,t+1}\subseteq T_{\xi^-,t}$, $T_{\lambda,t+1}(\emptyset)=T_{\lambda,t}(\emptyset)\subseteq T_{\xi^-,t}(\emptyset)\subseteq T_{\xi,t+1}(\emptyset)$. Otherwise, Case 1 of the construction of
\ref{43} is followed at stage $t+1$, and there is a $\delta\supset\xi^-$ and a transmission sequence $\{\la\sigma_\beta,i_\beta\ra\mid p(\xi)\subseteq\beta\subset\delta\}$ such that $T_{\xi,t+1}(\emptyset)=\sigma_\xi$. By the cancellation procedure and since $T_{\delta,t}=\emptyset$, $\lambda$ has higher priority than $\delta$. Hence by \ref{46i}, $T_{\lambda,t+1}(\emptyset)\subseteq\sigma_\xi=T_{\xi,t+1}(\emptyset)$.

\ref{46iii} \ref{23}\ref{23i}, \ref{23ii}, and \ref{23v} for $T_{\lambda}$ at $t+1$ follow from \ref{45}\ref{45i}, \ref{45}\ref{45vi} and \ref{45}\ref{45v} respectively. Let $T_{\lambda}$ receive $\la\alpha,i\ra$ with $i\ge 1$ at $t+1$. Fix the shortest $\xi\supset\lambda$ such that $T_{\xi}$ transmits $\la\alpha,i\ra$ at $t+1$ and $T_{\xi,t}$ is designated as a Phase-2 or Phase-1 tree at
$t+1$. If either $T_{\xi}$ is newly designated at $t+1$ or if $t$ and $t+1$ are in different states on $T_\xi$, then by \ref{21435}\ref{214vi35v}, for all $\eta\subseteq\lambda$, $\alpha$ is a
potential focal point of $T_{\eta,t}$ which is not a focal point of $T_{\eta,t}$, so \ref{23}\ref{23iii} and \ref{23}\ref{23iv} hold for $T_{\lambda,t}$. We must now verify \ref{23}\ref{23iii} for subsequent stages. \ref{23}\ref{23iii} will follow for $T_{\lambda,t}$ by induction unless $T_{\lambda,t}\ne T_{\lambda,t-1}$. Thus assume that $T_{\lambda,t}\ne T_{\lambda,t-1}$. Since $T_{\xi}$ transmits $\la\alpha,i\ra$ at $t$ and is not cancelled at $t$, it follows from \ref{44}\ref{44iii} that $T_{\lambda}$ prefers $\la\alpha,i\ra$ at $t$. Applying \ref{46v} inductively, we see that $T_{\lambda,t}$ is a type $i$ extension of $T_{\lambda,t-1}$ for $\alpha$ such that $\Ht(T_{\lambda,t})=\Ht(T_{\emptyset,t})$. Since $T_{\lambda,t}\subseteq T_{\emptyset,t-1}$, $\Ht(T_{\lambda,t})=\Ht(T_{\emptyset,t-1})$. Hence by \ref{21435}\ref{214ixa35viiia}, $t$ and $t+1$ are in different states on $T_\xi$, a case which has already been considered.

[\emph{Note:} Id grows at every stage, but Init does not. Hence it is perfectly possible that all the trees from $T_\emptyset$ to the tree that has just grown have the same height.]

\ref{46iv} We proceed by induction on $|\lambda|$, shorter strings first. Assume that for all $\xi\subseteq\lambda$, $\beta$ is a potential focal point of $T_{\xi,t}$, and that for all
$\la\alpha,i\ra\in S_{\lambda,t+1}$, either $i\ge 2$ or $\beta\subseteq\alpha$. If $T_{\lambda,t+1}=T_{\lambda,t}$ then $\beta$ is a potential focal point of $T_{\lambda,t+1}$. Assume that $T_{\lambda,t+1}\ne T_{\lambda,t}$. If $T_{\lambda}$ is designated as an Init or Phase-2 tree at $t+1$, then by \ref{28}\ref{28i} and \ref{35}\ref{35ix}, $T_{\lambda}$ receives and prefers some $\la\alpha,i\ra$ at $t+1$ and $T_{\lambda,t+1}$ is a type $i$ extension of $T_{\lambda,t}$ for $\alpha$. By \ref{45}\ref{45vii}, for all $\rho\subset T_{\lambda,t+1}-T_{\lambda,t}$, $\beta\subseteq\rho$. Hence $\beta$ is a potential focal point of $T_{\lambda,t+1}$. $T_{\lambda,t}\ne\emptyset$, else $\beta$ would not be defined. Suppose that $T_{\lambda}$ is designated as an Ext tree at $t+1$. Then $T_{\lambda^-,t}$ is designated at $t+1$ and receives at $t+1$ all pairs which $T_{\lambda}$ receives at $t+1$. Let $T_{\lambda^-}$ receive $\la\alpha,i\ra$ with $i\le 1$ at $t+1$. Then there is a $\xi$ such that $\xi^-=\lambda^-$ and $T_{\xi}$ transmits $\la\alpha,i\ra$ at $t+1$. By the construction of \ref{43-}, $\lambda$ cannot have higher priority than $\xi$. It now follows from \ref{45}\ref{45iv} that for all $\la\alpha,i\ra$ received by $T_{\lambda^-}$ at $t+1$, either $\beta\subseteq\alpha$ or $i\ge 2$. Hence \ref{46iv} follows by induction.

\ref{46v} If $\lambda=\emptyset$, then it follows from Remark \ref{210} that $\mc T=\{T_{\eta,s}\mid\eta\subseteq\lambda\and t\le s\le t+1\}$ is special. Assume that $\lambda\ne\emptyset$. \ref{29}\ref{29i} for $\mc T$ follows from \ref{212}\ref{212i} and \ref{21435}\ref{214i35i}. \ref{29}\ref{29ii} for $\mc T$ follows from \ref{46iv}. We now verify \ref{29}\ref{29iii} for $\mc T$.

We proceed by induction on $|\lambda|$, shorter strings first. Thus we may assume that $\{T_{\eta,s}\mid\eta\subset\lambda\and t\le s\le t+1\}$ is special. We assume that $T_{\lambda,t+1}\ne T_{\lambda,t}$ or that $T_{\lambda}$ is newly designated at $t+1$. If $T_{\lambda}$ is designated as an Init or Phase tree at $t+2$, then it follows from \ref{28}\ref{28ii} and \ref{21435}\ref{214xii35ix} that $\Ht(T_{\lambda,t+1})=\Ht(T_{\emptyset,t})$, $T_{\lambda,t}$ has no transmission, and if $T_{\lambda,t}\ne\emptyset$ or if $T_{\lambda}$ is newly designated at $t+2$, then it follows from \ref{28}\ref{28i} and \ref{35}\ref{35ix} again that $T_{\lambda}$ receives and prefers some $\la\alpha,i\ra$ at $t+1$ and $T_{\lambda,t+1}$ is a type $i$ extension of $T_{\lambda,t}$ for $\alpha$. By \ref{44}\ref{44vii}, only one tree grows per stage, so we cannot have $T_{\eta,t+1}\ne T_{\eta,t}$ for any $\eta\subset\lambda$. Hence $\Ht(T_{\lambda,t+1})=\Ht(T_{\lambda,t})$ for all $\eta\subseteq\lambda$, and \ref{29}\ref{29iii} holds in this case.

Suppose that $T_{\lambda}$ is designated as an Ext tree at $t+2$.

First suppose that $T_{\lambda}$ is newly designated at $t+2$. Then Case 1, Subcase 1 or Subcase 2 of the construction of \ref{43} is followed at stage $t+1$.

By \ref{45}\ref{45ii}, $T_{p(\lambda)}$ must newly transmit some $\la\tau,j\ra$ with $j\ge 2$ at $t+1$, hence by \ref{21435}\ref{214vi35v}, $\tau$ is a potential focal point of $T_{\eta,t}$ which
is not a focal point of $T_{\eta,t}$ for all $\eta\subseteq\lambda^-$, and $T_{\lambda}$ has no transmission at $t+1$. Hence $\Ht(T_{\lambda,t+1})=\Ht(T_{\emptyset,t})$. Again it follows from
\ref{44}\ref{44vii} that $\Ht(T_{\lambda,t+1})=\Ht(T_{\eta,t+1})$ for all $\eta\subseteq\lambda$.

Next suppose that $T_{\lambda}$ is not newly designated at $t+2$. Fix the longest $\xi\subset\lambda$ such that $T_{\xi}$ is designated as an Init or Phase-2 tree at $t+1$. Then $T_{\xi,t+1}\ne T_{\xi,t}$ so by induction, $\Ht(T_{\xi,t+1})=\Ht(T_{\emptyset,t+1})$, $T_{\xi}$ receives and prefers some $\la\alpha,i\ra$ at $t+1$ and $T_{\xi,t+1}$ is a type $i$ extension of $T_{\xi,t}$ for $\alpha$. By \ref{212}\ref{212ii}, it suffices to show that $T_{\lambda}$ receives $\la\alpha,i\ra$ at $t+1$. But this follows from \ref{45}\ref{45iv} and induction, as if $\lambda^-=\xi$, then by the construction of \ref{43}, there can be no $\eta$ such that $T_{\eta}$ is designated at $t+1$, $\eta^-=\xi$, and $\lambda$ has higher priority than $\eta$.

\ref{46vi} Fix $\lambda$ such that $T_{\lambda,t+1}(\emptyset)\downarrow$. First suppose that $T_{\lambda}$ is newly designated at $t+2$ or that $T_{\lambda,t}=\emptyset$. By
\ref{45}\ref{45viii}, dom$(T_{\lambda,t+1})=\{\emptyset\}$. By \ref{46v} and \ref{29}\ref{29iii}, $T_{\lambda,t+1}(\emptyset)$ is a potential focal point of $T_{\xi,t+1}$ for all $\xi\subseteq\lambda$. Otherwise, $T_{\lambda,t+1}(\emptyset)=T_{\lambda,t}(\emptyset)$ which, by induction, is a potential focal point of $T_{\xi,t}$ for all $\xi\subseteq\lambda$. By \ref{46iv}, it suffices to show that for all $\xi\subseteq\lambda$ and $\la\alpha,i\ra\in S_{\xi,t+1}$, if $i\le 1$ then $T_{\lambda,t+1}(\emptyset)\subseteq\alpha$. Let $\la\alpha,i\ra\in S_{\xi,t+1}$ be given with $i\le 1$. By \ref{44}\ref{44i} and \ref{44}\ref{44iii}, there is a transmission sequence $\{\la\sigma_\beta,i_\beta\ra\mid\xi\subseteq\beta\subset\delta\}$ for some $\delta\supset\xi$ at stage $t+1$, with $\la\sigma_\xi,i_\xi\ra=\la\alpha,i\ra$ and $T_{\delta,t}=\emptyset$. If $\lambda$ has higher priority than $\delta$, then by \ref{46i}, $T_{\lambda,t+1}(\emptyset)\subseteq\sigma_\xi=\alpha$. We cannot have $\lambda=\delta$, as $T_{\delta,t}=\emptyset$, so by \ref{45}\ref{45i}, $S_{\delta,t+1}=\emptyset$. We complete the proof by assuming that $\delta$ has higher priority than $\lambda$ and obtaining a contradiction. Fix the longest $\eta$ such that $\eta\subseteq\delta$ and $\eta\subseteq\lambda$. Then $\xi\subseteq\eta$. Hence $i_\eta\le 1$, contradicting \ref{44}\ref{44ii}.
\end{proof}

\subsubsection{Non-hanging}

\begin{lem}\label{claim47i}
If $T_\delta$ is an empty Phase-1 tree and $\delta=\gamma_s$ and $T_\delta$ transmits a pair at $s+1$ then $T_\delta$ is a trigger at $s+1$.
\end{lem}
\begin{proof}
Consider the maximal transmission sequence $\{\la\sigma_\gamma,i_\gamma\ra\mid\beta\subseteq\gamma\subset\delta\}$ at $s+1$. There are 3 possible cases.
\begin{itemize}
\item[(1)] 
	$\beta=\emptyset$ (so $T_{\gamma_s}$ is a trigger),
\item[(2)] 
	$T_\beta$ has no transmission (which means some Phase-2 tree just grew, by \ref{35}\ref{35xiii} and \ref{46}\ref{46iii} and the assumption that $T_\delta$ does have a transmission, so $T_{\gamma_s}$ is a trigger)
\item[(3)] 
	$T_\beta$ transmits a type $i\ge 2$ message (in which case we redirect the true path, and since we hadn't already done so (which we hadn't, since the Phase-1 tree was $=\gamma_s$), this must be the first stage that $T_\beta$ transmitted that type $i\ge 2$ message, so $T_{\gamma_s}$ is a trigger, even if $T_\beta$ is $T_{\gamma_s}$ itself, by \ref{42}\ref{42iv} (in this case of course the transmission sequence is $=\emptyset$)).
\end{itemize}
\end{proof}

\begin{lem}\label{47}
Let $\delta$ and $t\in\omega$ be given such that $T_\delta$ is designated at (the end of) stage
$t$. Then
\begin{enumerate}
\item \label{47i} 
	If $T_{\delta,t}=\emptyset$ then $T_\delta$ is designated as a Phase-1 tree at $t$. If, in addition, $\eta$ is given such that $\eta^-=\delta$, then $T_\eta$ is not designated at $t$.
\item \label{47ii} 
	If $T_\delta$ is not cancelled at stage $t+1$ or later, then $\{s\mid\gamma_s=\delta\}$ is finite.
\end{enumerate}
\end{lem}
\begin{proof}
\ref{47i} The second statement follows from the first, since children of Phase-1 trees are not Phase-1 trees, by Definition \ref{43}, and since the child of an empty tree is empty.

Now observe that a tree is first newly designated at some stage and then keeps its designation until it is newly designated again. Hence to determine what a tree is designated as, it suffices to
determine what it is designated as when it is newly designated.

Hence it is more than sufficient to prove the following: If $T(t):=T_{\gamma_t,t}=\emptyset$ is newly designated then $T(t)$ is a Phase-1 tree.

By \ref{45}\ref{45viii}, it suffices to show that if an empty Phase-1 tree is $\gamma_s$ then there is a trigger at stage $s+1$, and hence $\gamma_s*0$ will not be newly designated at $s+2$.

Now such a Phase-1 tree will have a transmission (by \ref{45}\ref{45viii}, \ref{214}\ref{214xiii} and induction along the ctp, it will not be in state $\la 0,0\ra$ and will have a transmission).
Hence we are done by Lemma \ref{claim47i}.

\ref{47ii} Suppose that $T_{\delta}$ is designated at $t+1$ and $T_{\delta}$ is not cancelled at any stage $s+1>t$. If $s>t$ and $\gamma_s=\delta$, then $T_\delta=\emptyset$ must be a trigger at stage $s$. We proceed by induction on $\{\beta\mid\beta\subseteq\delta\}$, longer strings first. For each such $\beta$, we show that $\{s\mid T_\delta$ triggers $T_\beta$ at stage $s\}$ is finite, thus
proving \ref{47ii}. First suppose that $\beta=\delta$. If $r>s>t$ and $T_\delta$ triggers $T_\delta$ at stages $r$ and $s$, then by \ref{21435}\ref{214xi35xii} and \ref{214}\ref{214xiii},
the state of $s$ on $T_\delta$ lexicographically precedes the stage of $r$ on $T_\delta$. Since there are only finitely many possible states, by \ref{21435}\ref{214x35xi}, there must be a stage
$t(\delta)>t$ and a pair $\la\sigma_{\delta^-},i_{\delta^-}\ra$ such that for all $s\ge t(\delta)$, $s$ and $t(\delta)$ are in the same state on $T_\delta$ and $T_{\delta}$ transmits
$\la\sigma_{\delta^-},i_{\delta^-}\ra$ at $s+1$. Thus $T_\delta$ cannot trigger $T_\delta$ at any stage $s>t(\delta)$. If $i_{\delta^-}\ge 2$, then $T_\delta$ cannot be a trigger after stage
$t(\delta)$, so the proof of \ref{47ii} is complete.

Assume by induction that $\{s\mid T_\delta$ triggers $T_\beta$ at stage $s\}$ is finite. We will have
the following induction hypothesis:

\begin{equation} \label{4(5)} \text{There is a fixed transmission sequence
$\{\la\sigma_\lambda,i_\lambda\ra\mid\beta^-\subseteq\lambda\subset\delta\}$}
\end{equation}
\begin{equation*} \text{ at all $s\ge$ some
$t(\beta)>t$.} 
\end{equation*}

\ref{4(5)} has been verified in the previous paragraph for $\beta=\delta$. Note that $T_{\beta^-}$ is not cancelled at any stage $s+1>t$. Hence we may define $T_{\beta^-}=\Union\{T_{\beta^-,s}\mid s>t\}$ and designate it in the same way in which $T_{\beta^-}$ is designated at $t+1$. If $T_{\beta^-}$ is designated as an Ext tree, then by \ref{212}\ref{212ii}
and induction, $T_{\beta^-}$ transmits $\la\sigma_{\beta^-},i_{\beta^-}\ra$ at $s+1$ for all $s\ge t(\beta)$. Furthermore, $T_\delta$ cannot trigger $T_{\beta^-}$. Hence $t(\beta^-)=t(\beta)$ and
\ref{4(5)} holds. If $T_{\beta^-}$ is designated as an initial tree, then $\beta^-=\emptyset$ and $T_\delta$ triggers $T_\emptyset$ at stage $t(\beta)$. By \ref{4(5)} and \ref{28}\ref{28i},
$T_\delta$ triggers some $T_\eta$ at stage $t(\beta)+1$ with $\emptyset\subset\eta\subseteq\delta$, contrary to the choice of $t(\beta)$. Hence if $\beta^-=\emptyset$, we have completed the
verification of \ref{47ii}.

The remaining case is when $T_{\beta^-}$ is designated as a Phase-2 tree. By \ref{44}\ref{44iii}, $T_{\beta^-}$ prefers $\la\sigma_{\beta^-},i_{\beta^-}\ra$ at all stages $s\ge t(\beta)$. If
$r>s\ge t(\beta)$ and $T_{\beta^-}$ transmits different pairs at $r+1$ and $s+1$, then by \ref{21435}\ref{214x35xi}, $r$ and $s$ must be in different states on $T_{\beta^-}$; by
\ref{21435}\ref{214xi35xii}, the state of $s$ on $T_{\beta^-}$ lexicographically precedes the state of $r$ on $T_{\beta^-}$, unless there is a stage $u$ such that $s<u\le r$ and $T_{\beta^-,u}\ne
T_{\beta^-,u-1}$. If no such stage $u$ exists, then all sufficiently large states must be in the same state on $T_{\beta^-}$, so $t(\beta^-)$ must exist as specified in \ref{4(5)}. But if $u$
exists, then by \ref{47i}, $T_\delta$ would trigger some $T_\eta$ at stage $u+1$ with $\beta^-\subset\eta\subseteq\delta$ (by \ref{47i}, at each stage there is at most one empty Phase-1
tree along any path in the priority tree), contrary to the choice of $t(\beta)$. Hence \ref{47ii} must hold.

Less formal proof of \ref{47ii}: If a Phase-2 tree eventually always prefers the same pair then it will eventually always transmit the same pair and remain in a fixed state, and so is eventually
never triggered.
\end{proof}

\subsubsection{True path}

\begin{lem}\label{48}
$|\Gamma|=\infty$. Furthermore, there are $\gamma(m)$ and $t(m)<\omega$ such that:
\begin{enumerate}
\item \label{48i} 
	$|\gamma(m)|=m$.
\item \label{48ii} 
	For all $t\ge t(m)$, either $\gamma_t\supseteq\gamma(m)$ or $\gamma_t$ has lower priority than $\gamma(m)$.
\item \label{48iii} 
	$\{r\mid\gamma_r\supset\gamma(m)\}$ is infinite.
\item \label{48moreover} 
	For each $\beta$ of higher priority than $\Gamma$ such that $T_{\beta}$ is designated at infinitely many stages $t$, there is a least stage $t(\beta)$ such that $T_{\beta}$ is not
cancelled at any stage $t>t(\beta)$, $T_{\beta}$ is designated at $t(\beta)+1$, and if there is a $t\ge t(\beta)$ such that $T_{\beta,t}(\emptyset)\downarrow$ then $T_{\beta,t(\beta)}(\emptyset)\downarrow$.
\end{enumerate}
\end{lem}

\begin{proof}
We proceed by induction on $m$. \ref{48i}-\ref{48iii} are easily verified for $m=0$. Assume that \ref{48i}-(iii) hold for $m=k-1$. Since there are only finitely many $\gamma$ on the priority tree
with $|\gamma|=k$, it follows from (iii) that there is a $\gamma(k)\supset\gamma(k-1)$ of highest priority such that $|\gamma(k)|=k$ and $\gamma_s\supseteq\gamma(k)$ for infinitely many $s$. Fix a stage $t(k)$ such that $T_{\gamma(k)}$ is not cancelled at any stage $s\ge t(k)$. By (ii) for $m=k-1$, it follows that $T_{\gamma(k-1)}$ is not cancelled at any stage $s\ge t(k-1)$, hence by
\ref{47}\ref{47ii}, $t(k)$ must exist. (i) and (ii) are now easily verified for $m=k$. And (iii) is immediate from the choice of $\gamma(k)$ and \ref{47}\ref{47ii}. We now note that
$\Gamma=\Union\{\gamma(m)\mid m<\omega\}$, so by (i), $|\Gamma|=\infty$. For \ref{48moreover}, since $\beta$ has higher priority than $\Gamma$, $\beta$ is cancelled finitely often. Let $t(\beta)$ be the least stage at which $T_\beta$ is newly nonempty and never again cancelled.
\end{proof}

\begin{df}
For $\beta$ as in \ref{48}\ref{48moreover}, let $T_\beta=\Union\{T_{\beta,t}\mid t>t(\beta)\}$, and let $T_\beta$ have the same designation as it had at $t(\beta)+1$. Let $g(x)=\lim_{s\to\infty}\alpha_s(x)$ for all $x$ for which this limit is defined.
\end{df}

\begin{lem}\label{49}
$|g|=\infty$ and $\mb g\le\mb 0'$.
\end{lem}
\begin{proof}
Since $s\mapsto\alpha_s$ is recursive, if $|g|=\infty$ then $\mb g\le\mb 0'$. We show that $|g|=\infty$ in a two-part proof. Let $g^*=\Union\{T_\gamma(\emptyset)\mid\gamma\subset\Gamma\}$. We
first show that $|g^*|=\infty$ and then show that $g=g^*$.

Let $\gamma\subset\Gamma$ be given such that $\gamma\ne\emptyset$ and $\gamma$ is (initially, i.e., at the stage where it is newly designated and never again cancelled) designated as a Phase-1 tree.
If $\gamma=\gamma^-*0$, then by \ref{214}\ref{214viii}, $|T_\gamma(\emptyset)|>|T_{\gamma^-}(\emptyset)|$. And if $\gamma=\gamma^-*1$ then there is a pair $\la\alpha,i\ra$ such that $T_{\gamma^-*0}$ transmits $\la\alpha,i\ra$ at all sufficiently large stages, $i\ge 2$, and $T_{\gamma,t}(\emptyset)=\alpha$. By \ref{214}\ref{214vii}, $|T_\gamma(\emptyset)|>|T_{\gamma^-}(\emptyset)|$. Hence the unique accumulation point of the set of $|T_\gamma(\emptyset)|$ for $\gamma\subset\Gamma$ is $\infty$. Since $T_\gamma\subseteq T_{\gamma^-}$ for all $\gamma\subset\Gamma$ such that $\gamma\ne\emptyset$, $|g^*|=\infty$.

We will show that $g=g^*$ by proving that for all $\gamma\subset\Gamma$ initially designated as a Phase-1 tree and all $t\ge t(\gamma)$, $T_\gamma(\emptyset)\subseteq\alpha_t$. Fix such a $\gamma$.
By \ref{48}\ref{48ii}, for all $t\ge t(\gamma)$, either $\alpha_t=T_{\delta,t}(\emptyset)$ for some $\delta$ of lower priority than $\gamma$ such that $T_{\delta}$ is designated but not cancelled at
stage $t$, $\{\la\sigma_\beta,i_\beta\ra\mid\xi\subseteq\beta\subset\delta\}$ is a transmission sequence at stage $t$, $T_{\xi}$ transmits $\la\sigma,i\ra$ at $t$ and $\alpha_t=\sigma$. Hence by
\ref{46}\ref{46i} and \ref{46}\ref{46ii}, $T_\gamma(\emptyset)\subseteq\alpha_t$.
\end{proof}

%% file: PRESEN-3.tex
\subsection{Trees for initial segments}

\subsubsection{Definition of trees and projections}
\begin{df}\label{V210}
Given a finite lattice $L$, with $\Theta=\Theta(L)$, and $\sigma,\tau\in S(\Theta)$, and $a\in L$, let $\sigma^{\la i\ra}$ be defined as follows: $|\sigma^{\la i\ra}|=|\sigma|$, and for each
$x<|\sigma|$, let $\sigma^{\la i\ra}(x)=(\sigma(x))^{[i]}$. The notation $f^{\la i\ra}$ for functions is defined similarly; $f^{\la i\ra}(x)=(f(x))^{[i]}$.
\end{df}

\begin{df}
Let $i\mapsto\la\varphi_i,L^i\ra$ have direct limit $L$. If $i\le j$, $b\in L^i$, $b_0\in L^j$ and $(b\approx b_0)[j]$ then we write $b_0=b[j]$. Note that $b\mapsto b[j]$ defines a $\la 0,1,\vee\ra$-homomorphism from $L^i$ to $L^j$ and is equal to $\varphi_{j-1}\cdots\varphi_i$. So $1[i]$ denotes the greatest element of $L^i$.
\end{df}

\begin{df}[Trees for lattices]
Let $i\mapsto\la\varphi_i,L^i\ra$ have direct limit $L$. Let $\Theta=\Theta(L)$.

Suppose $T$ is a $(\Theta^i,\Theta^0)$-tree. Suppose we are given a sequence of trees $\la T_n\mid n<\omega\ra$ with $T_{n+1}\subseteq T_n$ such that for some $f:\omega\rightarrow\omega$ and each $i$, $n$ with $f(i)\le n<f(i+1)$,
$T_n:S(\Theta^i)\to S(\Theta^0)$, $f(i)<f(i+1)$, $f(0)=0$.

If $g:\omega\to\omega$, $g\subset\inter_n T_n$, $i<\omega$ and $a\in L^i$, then let $g^{\la a\ra}$ be $T_{f(i)}^{-1}(g)$.
\end{df}

\subsubsection{Computation Lemma}

\begin{df}\label{V24}
Let $\sigma,\tau,\rho\in\omega^{<\omega}$ be given. We call $\la\sigma,\tau\ra$ an \emph{$e$-splitting} of $\rho$ if $\rho\subset\sigma$, $\rho\subset\tau$, and for some $x<\omega$, $\{e\}^\sigma(x)\downarrow\ne\{e\}^\tau(x)\downarrow$. For such an $x$, we call $\la\sigma,\tau\ra$ an \emph{$e$-splitting of $\rho$ for $x$}. If $T$ is a tree and $\sigma,\tau\subset T$, then we call $\la\sigma,\tau\ra$ an $e$-splitting of $\rho$ on $T$.
\end{df}

\begin{df}\label{11}
Given a finite lattice $L$, $\sigma,\tau\in S(\Theta(L))$ and $a\in L$, define $\sigma\sim_a\tau$ if $\sigma^{[a]}(x)=\tau^{[a]}(x)$ for all $x<\min(\{|\sigma|,|\tau|\})$.

If $a\in L$, we say that $\la\sigma,\tau\ra$ is an \emph{$e$-splitting mod $a$} if $\la\sigma,\tau\ra$ is an $e$-splitting and $\sigma\sim_a\tau$.
\end{df}

\begin{df}\label{112}
Level $i$ is an \emph{$e$-splitting level} of $T$ for $k$ if for all $\xi,\eta$, if $|\xi|=|\eta|=i+1$, $T(\xi)\downarrow$, $T(\eta)\downarrow$ and $\xi\not\sim_k\eta$ then $\la T(\xi),T(\eta)\ra$ is an $e$-splitting.
\end{df}

\begin{df}\label{113}
If $\{T_s\mid s\ge s^*\}$ is an increasing recursive sequence of finite trees for $L$ with $T=\union\{T_s\mid s\ge s^*\}$ then $T$ is a \emph{weak $e$-splitting tree for $k$} generated by $\{T_s\mid s\ge s^*\}$ if
\begin{enumerate}
\item there are no $e$-splittings mod $k$ on $T$, and
\item the last level of every plateau of each $T_s$ is an $e$-splitting level of $T_s$ for $k$.
\end{enumerate}
\end{df}

In the following we assume that all trees considered are weakly uniform.

\begin{lem}[Computation Lemma]\label{114}
Let $e<\omega$, a finite lattice $L$ and $k\in L$ be given and let $T$ be a partial recursive $(\Theta(L),\Theta')$-tree (for some $\Theta'$) which is weak $e$-splitting for $k$. If $g$ is an infinite branch of $T$ such that $\{e\}^g$ is total then $\{e\}^g\sim_T g^{\la k\ra}$.
\end{lem}
\begin{proof}
We first show how to compute $\{e\}^g(x)$ recursively from $g^{\la k\ra}$. Search for $\sigma\subset T$ such that $\sigma\sim_k g$ and $\{e\}^\sigma(x)\downarrow$. Let $\tau\subset g$ be given such that $\tau\subset T$ and $\{e\}^\tau(x)\downarrow$. Such $\sigma$ and $\tau$ must exist since $\{e\}^g$ is total. Since $\sigma\sim_k\tau$ and there are no $e$-splittings mod $k$ on $T$, $\{e\}^\sigma(x)=\{e\}^\tau(x)=\{e\}^g(x)$. Since $\{e\}^\sigma(x)$ was computed following a procedure which is uniformly recursive in $g^{\la k\ra}$, $\{e\}^g\le_T g^{\la k\ra}$.

We now show how to recover $g^{\la k\ra}$ recursively from $\{e\}^g$. We proceed by induction on $j$, finding, at step $j$, $\sigma_j\subset T$ such that $\sigma_j=T(\xi_j)$, $|\xi_j|=j$, and
$\sigma_j\sim_k g$. When $j=0$, we choose $\sigma_0=T(\emptyset)$.

At step $j+1$, expressing $T$ as $\union\{T_s\mid s\ge s^*\}$, find the least $s\ge s^*$ and the smallest level $r$ of $T_s$ such that:
\begin{equation}\label{1(1)}
	\text{Level $r$, $[u,v)$, is an $e$-splitting level of $T_s$ for $k$ and $v>|\sigma_j|$.}
\end{equation}
\begin{equation}\label{1(2)}
	\exists\tau\subset T_s(|\tau|=v\and\tau\sim_k\sigma_j).
\end{equation}
Note that since each $T_s$ is weakly uniform, each plateau of each $T_s$ is full, so the interval $[|\sigma_j|,v)$ is full on $T_s$. Fix $T(\eta)=\rho\subset\tau$ such that $|\rho|=|\sigma_j|$. If
$\mu,\nu\subset T_s$ are such that $\rho\subseteq\mu$, $\rho\subseteq\nu$, $|\mu|=|\nu|=v$ and $\mu\not\sim_k\nu$, then by \ref{1(1)}, $\la\mu,\nu\ra$ $e$-splits on some $x$. Hence $\{e\}^g$ can be used to eliminate at least one of $\mu$ and $\nu$ as a potential candidate for a string $\sigma$ such that $|\sigma|=v$ and $\sigma\sim_k g$. Complete this elimination process, ending with $\mu$.
(If no string remains at the end, choose $\mu$ arbitrarily.) Let $\mu=T(\eta*\delta)$. Choose $\sigma_{j+1}=T(\xi_{j+1})$ such that $\sigma_{j+1}\subseteq\mu$ and $|\xi_{j+1}|=j+1$. Since
$\sigma_{j+1}\sim_k\mu$, it suffices to show that $\mu\sim_k g$. Fix $\alpha,\beta$ such that $|\alpha|=|\eta|$, $|\beta|=|\eta*\delta|$ and $T(\alpha)\subset T(\beta)\subset g$. Let $\beta=\alpha*\gamma$. Since $[|\sigma_j|,v)$ is full on $T_s$, $T_s(\eta*\gamma)\downarrow$. By the choice of $s$, $T(\eta*\gamma)$ cannot be eliminated during the above process since
$\eta\sim_k\sigma_j$ and there are no $e$-splittings mod $k$ on $T$, hence on $T_s$. Thus $\mu\sim_k T(\eta*\gamma)\sim_k T(\beta)\sim_k g$.
\end{proof}

\subsubsection{Homomorphism Lemma}

\begin{df}[$L$-tree]\label{116}
If $T_k:S(\Theta^i)\to S(\Theta^0)$ and $T_{k+1}\subseteq T_k$ then $T_{k+1}$ is an $(L,i,i+1)$-\emph{subtree of} $T_k$, or for short $T_{k+1}$ is an $L$-\emph{tree}, if the following holds:

$T_{k+1}(\emptyset)=T_k(\xi)$ for some $\xi$ with $|\xi|\ge m_i(0)$, and $T_{k+1}(\eta*a)=T_k(\xi*a^n)$ if $T_{k+1}(\eta)=T_k(\xi)$ where $n\ge m_i(j)-m_i(j-1)$, $j=|\eta|$, $a\in\Theta^{i+1}_{j-1}$ and $T_{k+1}:S(\Theta^{i+1})\to S(\Theta^0)$. Note that $n$ can be found from the height functions of $T_k$, $T_{k+1}$.
\end{df}

\begin{lem}
In Definition \ref{116} with $i-1$ in place of $i$, let $T=T_{k+2}$ and $U=T_{k+1}$ ($T_k$ is the $k$th tree in the construction, $T_{k+2}\subseteq T_{k+1}\subseteq T_k$). I.e., $U$ is an $L$-tree used to go from $L^{i-1}$ to $L^i$. Let $\varphi=\varphi_{i-1}$. So $g^{\la a\ra}$ for $g\subset T$ and $a\in L^i$ is defined in terms of the signature of $g$ on $U$. Then for each $g\subset T$, $g^{\la b\ra}\equiv_T g^{\la\varphi b\ra}$ whenever $b\in L^{i-1}$, and there is a recursive function $\la e,b,\varphi\ra\mapsto \la c,d\ra$ such that for all $g\subset T$, $g^{\la b\ra}=\{c\}(g^{\la \varphi b\ra})$ and $g^{\la \varphi b\ra}=\{d\}(g^{\la b\ra})$.
\end{lem}
\begin{proof}
Since $T$, $U$ are partial recursive, $H_T$, $H_U$ are recursive. Now $g^{\la a\ra}=U^{-1}(g)^{\la a\ra}$ and $g^{\la\varphi_i a\ra}=T^{-1}(g)^{\la\varphi a\ra}=T^{-1}(g)^{\la a\ra}$ (the last
equation by Chapter 1), and $H_T^{-1} H_U U^{-1}(g)=T^{-1}(g)$ and $U^{-1}(g)$ is obtained from $T^{-1}(g)$ by repeating the $n$th element not once (as in $T^{-1}(g)$ itself) but $H_T^{-1}
H_U(n)-H_T^{-1} H_U(n-1)$ times.
\end{proof}

\begin{lem}[Homomorphism Lemma]\label{l:homomorphismlemma}
Suppose that 
\[
g^{\la a\ra}\equiv_T g^{\la\varphi_i a\ra} \tag{*}
\]
whenever $a\in L^i$.

Define $\psi:L\to 2^\omega$ by $\psi(a)=g^{\la a\ra}$. Define $\pi:L\to L/\negthickspace\approx$ by $\pi(a)=\pi(b)\iff a\approx b$. Define deg$:2^\omega\to\mc D$ by deg$(g)=\mb g$, the Turing degree of $g$. Then there exists $\Psi:L/\negthickspace\approx\to\mc D$ such that deg$\psi=\Psi\pi$ and $\Psi$ is a $\la 0,\vee\ra$-homomorphism.
\end{lem}
\begin{proof}
We need to show that $a\lesssim b\Implies g^{\la a\ra}\le_T g^{\la b\ra}$. We have $a\lesssim b\Implies (\exists a_0,b_0)(a\approx a_0\le b_0\approx b) \Implies g^{\la a\ra} \equiv_T g^{\la
a_0\ra}\le_T g^{\la b_0\ra}\equiv_T g^{\la b\ra}$, the $\equiv_T$ by several applications of (*) and the $\le_T$ by a basic property of lattice tables (since our notation $a_0\le b_0$ implies that
$a_0,b_0$ are in the same $L^i$).

Moreover if $a\vee^* b\approx c$ then $(\exists a_0,b_0,c_0)(a\approx a_0\and b\approx b_0\and c\approx c_0\and a_0\vee b_0=c_0)$, and by a basic property of lattice tables, $g^{\la
a_0\ra}\oplus g^{\la b_0\ra}\equiv_T g^{\la a_0\vee b_0\ra}$ (where $\vee$ denotes the join in the $L^i$ containing $a_0,b_0,c_0$). Hence by repeated application of (*), $g^{\la a\ra}\oplus g^{\la
b\ra}\equiv_T g^{\la c\ra}$.
\end{proof}

We can now define numbers $(a\to i), (a\from i)<\omega$ by $\{a\to i\}(g^{\la a\ra})=g^{\la a[i]\ra}$ and $\{a\from i\}(g^{\la a[i]\ra})=g^{\la a\ra}$. These numbers exist by the Homomorphism
Lemma \ref{l:homomorphismlemma}.

\begin{lem}
Suppose $T$ is a tree such that if $T:S(\Theta^i)\rightarrow S(\Theta^0)$ then $T(\eta)^{\la a\ra}$ is defined for $a\in L^i$ and not for $a\in L^{i-1}$ (i.e., $T$ is not an $L$-tree).
Then $\sigma\sim_a\tau\iff T(\sigma)\sim_a T(\tau)$ for each $\sigma,\tau\in\dom T$.
\end{lem}
\begin{proof}
By homogeneity, and, by the non-twisting property ($T_{k+1}(\xi)=T_k(\eta)\Implies T_{k+1}(\xi*a)\supseteq T_k(\eta*a)$) as in \cite{Lerman:83}*{VI.2.7}.
\end{proof}

\subsubsection{The plan}

\begin{df}\label{18}
A tree $T$ is called $\la e,a,b\ra$-\emph{differentiating} if $T$ is an $(\Theta^i,\Theta^0)$-tree for $L^i\ni a,b$ and there is an $x<|T(\emptyset)|$ such that $\{e\}(T(\emptyset)^{\la b\ra})(x)\downarrow\ne T(\emptyset)^{\la j\ra}(x)\downarrow$. (There is an implicit parent tree on which $\la a\ra$, $\la b\ra$ are computed.)
\end{df}

\begin{rem}\label{19}
If $T$ is $\la e,a,b\ra$-differentiating and $g$ is an infinite branch of $T$, then $\{e\}(g^{\la b\ra})\ne g^{\la a\ra}$.
\end{rem}

\begin{df}\label{110}
A tree $T$ is called $\la e,b\ra$-\emph{divergent} if $T$ is an $L^i$-tree for $L^i\ni b$ and there is an $x<\omega$ such that for all $\sigma\subset T$, $\{e\}(\sigma^{\la b\ra})(x)\uparrow$. (There is an
implicit parent tree on which $\la b\ra$ is computed.)
\end{df}

\begin{rem}\label{111}
If $T$ is $\la e,b\ra$-divergent and $g$ is an infinite branch of $T$, then $\{e\}(g^{\la b\ra})$ is not total.
\end{rem}

\begin{pro}\label{115}
Let $g:\omega\rightarrow\omega$ be given. Assume \ref{115i} and \ref{115ii} below hold.
\begin{enumerate}
\item \label{115i} Suppose $e,i<\omega$ and $a,b\in L^i$, $a\not\lesssim b$. Then there is $j\ge i$ and a partial recursive tree $T$ such that $g\subset T$ and either $T$ is $\la e',a[j],b[j]\ra$-differentiating or $\la e'',b[j]\ra$-divergent, where $\{e'\}=\{a\to j\} \{e\} \{b\from j\}$ and $\{e''\}=\{e\}\{b\from j\}$.

\item \label{115ii} For all $e<\omega$, there is $i<\omega$ and a partial recursive tree $T$ such that $g\subset T$ and either $T$ is $\la e',1[i]\ra$-divergent or there is $a\in L^i$ such that $T$ is weak $e'$-splitting for $a$, where $\{e'\}=\{1\from i\}$.
\end{enumerate}

Then $L\cong[\mb 0,\mb g]$.
\end{pro}
\begin{proof}
By the Homomorphism Lemma, if $a\lesssim b$ then $g^{\la a\ra}\le_T g^{\la b\ra}$.

Suppose $a\not\lesssim b$ and let $e<\omega$. If $g^{\la a\ra}=\{e\} g^{\la b\ra}$ then $g^{\la a[i]\ra}=\{a\to i\} g^{\la a\ra} =\{a\to i\} \{e\} g^{\la b\ra} =\{a\to i\} \{e\} \{b\from i\} g^{\la b[i]\ra}$, which contradicts \ref{115i}. Hence $g^{\la a\ra}\not\le_T g^{\la b\ra}$, and so $a\mapsto g^{\la a\ra}$ is a partial order embedding.

If $e<\omega$ then for the $i$ of \ref{115ii}, $\{e\}(g)=\{e'\}(g^{\la 1[i]\ra})$ and $\{e'\}(g^{\la 1[i]\ra})$ is either not total, or $\equiv_T g^{\la a\ra}$ for some $a\in L^i$, by \ref{115ii}. Hence
$a\mapsto g^{\la a\ra}$ is onto, and so it is a partial order isomorphism between $L$ and $[\mb 0,\mb g]$, and consequently also a $\la \vee,0,1\ra$-isomorphism.
\end{proof}

\subsection{The strategies}

\begin{rem}\label{22}
The reception of $\la\alpha,i\ra$ by the tree $T$ at stage $s+1$ will convey the instruction to carry out Objective $i$ whenever possible. The objectives for $i\ge 2$ are listed below.

\noindent\emph{Objective 2. Specify an $\la e,b\ra$-divergent Ext tree.} When $T_s$ receives $\la\alpha,2\ra$, $\alpha$ is a potential focal point of $T_s$, and a tree $T^*_s$ is specified
such that $T^*_s\subseteq T_s$ with $\alpha=T^*_s(\xi)$. $T_{s+1}$ is instructed to preserve $\alpha$ as a potential focal point, while a search for suitable strings in $\Ext(T^*,\xi)$
proceeds. If this search is unsuccessful, then $\Ext(T^*,\xi)$ will be $\la e,b\ra$-divergent for some $b\in L$.

\noindent\emph{Objective 3. Specify a tree with no $e$-splittings mod $k$.} The process is the same as in Objective 2, except that if the search is unsuccessful, then $\Ext(T^*,\xi)$ will have no
$e$-splittings mod $k$, for some $k\in L$.

\noindent\emph{Objective 4. Start a Right $L$-tree.} The process is the same as in Objective 2, except that if the Left $L$-tree is incorrect, then $\Ext(T^*,\xi)$ will be designated as a Right
$L$-tree cofinitely often. (See below.)
\end{rem}

\subsubsection{Diff and $L$ trees}

In the following Definitions \ref{213} and \ref{34}, ``transmit $\la\alpha,i\ra$'' means: put $T_{k+1}$ in a state indicating the current step, substep or subsubstep, let $T_{k+1}$ transmit
$\la\alpha,i\ra$ to $T_{k}$ at the current stage $t+1$, and go to the next stage. In a few important cases we explicitly name the state. If $T_{k+1}$ ``does not transmit any strings'', this
also should be taken to mean that we put $T_{k+1}$ in a state and go to the next stage. Also, ``least'' is used freely to mean the least element of some set under a fixed recursive
$\omega$-ordering of that set, if that set has no obvious ordering. Unless otherwise specified, $T_{k+1,t+1}=T_{k+1,t}$.

\begin{df}[Diff and $L$ tree constructions]\label{213}
\noindent{\bf Diff tree construction.}

Let $k,s<\omega$ be given, and let $\{T_{m,t}\mid m\le k\and t\ge s\}$ be an array of trees such that for all $m\le k$ and $t\ge s$, $T_{m,t+1}$ extends $T_{m,t}$ and if $m\ne 0$ then
$T_{m-1,t}\supseteq T_{m,t}$. For each $m\le k$, let $T_m=\Union\{T_{m,t}\mid t\ge s\}$ and let $T_{m,t-1}$ receive $S_{m,t}$ at stage $t$. Fix $e,s^*<\omega$ such that $s^*>s$ and $a,b\in L^i$
such that $a\not\le b$. We construct an $\la e,a,b\ra$-differentiating tree
\[ 
	T_{k+1}=\Diff(\{T_{m,t}\mid m\le k\and t\ge s\}, e,a,b,s^*) 
\]
as the union of the increasing sequence of trees
\[ 
	\{T_{k+1,t}=\Diff_t(\{T_{m,r}\mid m\le k\and t\ge r\ge s\}, e,a,b,s^*) \mid t\ge s^*\}. 
\]
Let $T_{k+1,0}=\emptyset$. We proceed by induction on $\{t\mid t\ge s^*\}$. Fix $p,q\in\Theta^i_0$ such that $p\sim_j q$ but $p\sim_i q$, and the least $x$ such that $T(p)(x)\not\sim_i T(q)(x)$.

Stage $t+1$ of the construction proceeds through the following sequence of steps. At the end of stage $t+1$, $t+1$ is placed in some \emph{state}. If $t+1>s^*$, proceed directly to the beginning
of the step or substep of the construction in which a state was assigned at stage $t$.

\noindent\emph{Step 0.} If $T_{k,t}(\emptyset)\uparrow$ or if $\Ht(T_{k,t})\ne$ht$(T_{0,t})$, transmit nothing. Otherwise, proceed to Step 1 if $T_{k,t}(p)\uparrow$, and to Step 2 if
$T_{k,t}(p)\downarrow=\delta$.

\noindent\emph{Step 1.} If $T_{k,t}$ is a type 1 extension of $T_{k,t-1}$ and $\Ht(T_{k,t})=$ht$(T_{0,t})$, proceed to Step 3 letting $\delta=T_{k,t}(p)$. Otherwise, transmit
$\la T_{k,t}(\emptyset),1\ra$.

\noindent\emph{Step 2.} Run the special module (see \ref{34}) with $\alpha^*=T_{k,t}(\emptyset)$,
then go to Step 3.

\noindent\emph{Step 3. Force $\{e\}(g^{\la b\ra};x)$.} Let $\tau=T_{k,t}(\beta)$ be the least string (under some $\omega$-order) such that $\delta\subseteq\tau$ and $\tau$ is a potential focal
point of $T_{k,v}$, where $v$ is the first stage at which Step 1 or Step 2 is completed. Search for $\sigma\subset\Ext(T_{k,t},\beta)$ such that $\tau\subseteq\sigma$ and $\{e\}(\sigma^{\la
b\ra};x)\downarrow$. If no such $\sigma$ is found, transmit $\la\tau,2\ra$. Otherwise, fix the least such $\sigma=T_{k,t}(\eta)$. Proceed to Step 4.

\noindent\emph{Step 4.}

\emph{Substep 0.} Run the special module with $\alpha^*=T_{k,t}(\emptyset)$, then go to Substep 1.

\emph{Substep 1.} Let $z=\{e\}(\sigma^{\la b\ra};x)$. Let $r$ be the first of $\{p,q\}$ such that
\[
	T_{k,t}(r)^{\la a\ra}(x)\ne z.
\]
If $r=p$ let $\lambda=\eta$ and if $r=q$ let $\lambda=\text{tr}(p\to q;\eta)$. Let $\rho=T_{k,t}(\xi)$ be the first string such that $\lambda\subseteq\xi$ and $|\rho|=$ht$(T_{k,t})$. Set $T_{k+1,t+1}(\emptyset)=\rho$ and transmit nothing. For all stages $u\ge t+1$, let $T_{k+1,u}=$$\Ext(T_{k,u},\xi,t+1)$ and let $T_{k+1}$ be designated as an Ext tree at stage $u+1$.

\noindent {\bf $L$-tree construction.}

We construct a tree to take us from $L^i$ to $L^{i+1}$, given a $\la \vee,0,1\ra$-homomorphism $\varphi_i:L^i\to L^{i+1}$. We may denote the tree by $T_{k+1}=\text{Lat}(\{T_{m,t}\mid m\le k\and
t\ge s\},\varphi_i,L^i,L^{i+1},s^*,\{S_{m,r}\mid m\le k+1,s^*<r\})$.

\noindent\emph{Stage $t+1\ge s^*$.}

If $t+1=s^*$, place $T_{k+1}$ considered as a Phase-2 tree in state $\la 0\ra$.

If $T_{k+1}$ considered as a Phase-2 tree is in state $\la 0\ra$, proceed to Phase 1, otherwise proceed to Phase 2.

\noindent {\bf Phase 1.}

\noindent\emph{Step 0. Begin $T_{k+1}$.} If $T_{k,t}(\emptyset)\uparrow$ or if $\Ht(T_{k,t})\ne$ht$(T_{0,t})$, set $T_{k+1,t+1}=\emptyset$. $T_{k+1}$ has no transmission at $t+1$.
Place $t+1$ in state $\la 0\ra$ and proceed to the next stage. Otherwise, proceed to Step 1.

\noindent\emph{Step 1. Designate a new focal point.} If $T_{k,t-1}$ is a type 1 extension of $T_{k,t-2}$ and $\Ht(T_{k,t-1})=$ht$(T_{0,t-1})$, proceed to Step 2, Substep 0. Otherwise, set
$T_{k+1,t}=\emptyset$ and transmit $\la T_{k,t-1}(\emptyset),1\ra$. Place $t$ in \emph{state} $\la 1\ra$ and proceed to the next stage.

\noindent\emph{Step 2, Substep $j\le m_i(0)$.} Let $d=\max\{|\xi|\,\mid \,T_{k,t}(\xi)\downarrow\}$.

If it is not the case that $T_k(\emptyset)$ is in the last plateau of $T_{k,t}$ with $\Ht(T_{k,t})=$ht$(T_{0,t-1})$ then transmit $\la T_k(\emptyset),0\ra$ and go to the next stage.

(If it is the case:) If $d<m_i(0)$ then go to Substep $j+1$. If $d\ge m_i(0)$ then go to Step 0, Substep 0 of Phase 2.

\noindent {\bf Phase 2.}

\noindent\emph{Step 0.}
\noindent\emph{Substep 0.}

Let $T_{k+1}(\emptyset)=T_k(\xi)$ for some $\xi$ of length $d$, set $j=1$, and go to Substep 1.

\noindent\emph{Substep 1. Express preference.}

If $S_{k+1,t+1}$ does not satisfy \ref{23}\ref{23i}-\ref{23v} or if $T_{k+1,t}$ does not prefer any element of $S_{k+1,t+1}$, place $t$ in state $\la 0,1\ra$. $T_{k+1,t}$ does not transmit any strings to $T_{k,t}$. Proceed to the next stage. Otherwise, go to Step 1.

\noindent\emph{Step 1.} Fix $\la\alpha,i^*\ra\in S_{k+1,t+1}$ such that $T_{k+1}$ prefers $\la\alpha,i^*\ra$ at $t+1$.

If $T_{k+1}$ also preferred $\la\alpha,i^*\ra\in S_{k+1,t}$ at $t$, proceed directly to the point in the construction at which stage $t$ ended. (Thus if stage $t$ ended within a last step, substep,
or subsubstep, we proceed directly to the beginning of that step, substep or subsubstep, with everything in the construction which has been defined at stage $t$ unchanged at stage $t+1$.)

Otherwise $T_{k+1}$ prefers a new pair $\la\alpha,i^*\ra$.

Let $\alpha^*=\alpha$ if $i^*=1$, and let $\alpha^*$ be the longest focal point of $T_k$ contained in $\alpha$, if $i^*=0$.(In the latter case, \ref{23}\ref{23ii} implies the existence of such a
focal point.) $\alpha^*$ is tentatively designated as the next focal point of $T_{k+1}$. [If we did not replace $\alpha$ by $\alpha^*$ then it might happen that in the end $T_k(\emptyset)$ is not a
focal point of $T_{k-1}(\emptyset)$.]

We now begin another level of $T_{k+1}$. We thus require that $T_{k,t}$ be a type $i^*$ extension of $T_{k,t-1}$ for $\alpha^*$ with $\Ht(T_{k,t})=$ht$(T_{0,t-1})$. If this is not the case, transmit
$\la\alpha^*,i^*\ra$.

Otherwise let $T'$ be a type 2 extension of $T_{k+1,t}$ for $\alpha^*$ contained in $T_{k,t}$, choose $\alpha'\supseteq\alpha$ (could use $\alpha'\supseteq\alpha^*$ but for later permitting
arguments it is better to use $\alpha$), $\alpha'\subset T_{k,t}$ such that $\alpha'$ is a potential focal point of $T_{k,t}$ which is not a focal point of $T_{k,t}$, and go to Step 2.

\noindent\emph{Step 2. $\Sigma^0_2$-outcome transmission.}
We say that \emph{the $\Pi^0_2$-outcome looks correct at a stage $s$} if $\forall x\le s\exists y\le s S(x,y,i)$, where $S$ is a recursive relation such that $\forall x R(x,i)\iff\forall x\exists y S(x,y,i)$ where $R$ is obtained from Definition \ref{permissible} with $\mb a=\mb 0'$. [Here $i$ is the defining parameter to the $L$-tree as in $L^i$, not any other use of the letter $i$.]
 
If the $\Pi^0_2$-outcome looks correct at stage $t+1$, go to Step 3, substep 0. Otherwise let $T_{k+1,t+1}=T_{k+1,t}$ (this will aid in the verification), transmit $\la\alpha',4\ra$ (the $\alpha'$ of Step 1, if you're newly in Step 2, otherwise same $\alpha'$ as $T_{k+1}$ transmitted at last stage), and go to the next stage.

\noindent\emph{Step 3, Substep $i\ge 0$. Make progress toward the next plateau. $\Pi^0_2$-outcome.}
 
Run the special module. Then let $d=\max\{|\xi|\,\mid \,T_k(\xi)\downarrow\}-|T_k^{-1}(\alpha^*)|$, the domain height difference between $\alpha^*$ and the top of the plateau of $\alpha^*$.

If $d<m_i(j)-m_i(j-1)$, go to Substep $i+1$. If $d\ge m_i(j)-m_i(j-1)$, extend $T_{k+1}$ to $T'$ and then define $T_{k+1}(\eta*a)=T_k(\xi*a^n)$ if $T_{k+1}(\eta)=T_k(\xi)\supseteq\alpha^*$ is at the top of $T_{k+1}$ and $T_k(\xi*a^n)$ is at the top of $T_k$ and $a\in\Theta^{i+1}_{j-1}$, and let $T_{k+1}$ transmit nothing, increase $j$, and place $t+1$ in state $\la 0,1\ra$.
\end{df}

Note that we check the $\Pi^0_2$ approximation once for each level we build of the tree. This suffices since then if $\Pi^0_2$ holds then the tree gets infinitely many levels and so satisfies
its requirement, and if $\Sigma^0_2$ holds then we will eventually be stuck in Step 2 forever.

\subsubsection{Extendibility Interpolation Lemma}
\begin{df}[Extendibility]\label{32}
Given a sequential lattice table $\Theta$ and $m,n<\omega$ we define $\Theta^{m+n}_m=\,\Theta_m\times\cdots\times\Theta_{m+n}$ and note that a finite cartesian product
of usl tables is an usl table, using the definition $\sigma\sim_k\tau\iff(\forall x<|\sigma|)(\sigma(x)\sim_k\tau(x))$. This gives a notion of homomorphism between such tables. A
partial homomorphism is a partial function that has the properties of a homomorphism on its domain. A partial homomorphism is called extendible if it is the restriction of some homomorphism. We write $\la \alpha,\beta\ra\mapsto\la \gamma,\delta\ra$ for the partial function $\{\la\alpha,\gamma\ra,\la\beta,\delta\ra\}$.
\end{df}

\begin{lem}[Extendibility Interpolation Lemma]\label{33}
Let $L$ be a finite lattice, $\Theta=\Theta(L)$, $m<\omega$, $u,v\in\Theta_m$, and $\lambda,\lambda'\in S(\Theta)$ be given such that $\la u,v\ra\mapsto \la\lambda,\lambda'\ra$ is a
partial homomorphism. Then there exists $t<\omega$ and $\lambda=\lambda_0,\lambda_1,\ldots,\lambda_t=\lambda'\in \Theta_{m+1}^{m+|\mu_0|}$ such that for each $0\le r<t$, $\la u,v\ra\mapsto \la \lambda_r,\lambda_{r+1}\ra$ or $\la v,u\ra\mapsto \la \lambda_r,\lambda_{r+1}\ra$ extends to a homomorphism $f_r:\Theta_m\rightarrow\Theta_{m+1}^{m+|\lambda|}$. Moreover $f_0(u)=\lambda$ and $f_{t-1}(u)=\lambda'$. Hence if $\eta,\eta'\in S(\Theta)$ are such that $\la\eta,\eta'\ra\mapsto\la\lambda,\lambda'\ra$ is a partial homomorphism and $\la u,v\ra=\la\eta(m),\eta'(m)\ra$ where $m=\mu m(\eta(m)\ne\eta'(m))$, then by using $f_r\circ\pi_m$ where $\pi_m=(\eta \mapsto \eta(m))$ we get the same result with $\eta,\eta'$ in place of $u,v$.
\end{lem}

\begin{proof}
For each $x<|\lambda|$, $u,v\in\Theta_m$ and $\lambda(x),\lambda'(x)\in\Theta_{x}$, so $\{u,v,\lambda(x),\lambda'(x)\}\subseteq\Theta_m\cup\Theta_{x}=\Theta_{\max\{m,x\}}$.

Let $x<|\lambda|$. By Malcev homogeneity there exists $w(x)=w<\omega$ and $\mu_s(x)$ for $0\le s\le w+1$ with $\lambda(x)=\mu_0(x)$, $\lambda'(x)=\mu_{w+1}(x)$, all in
$\Theta_{\max\{m,x\}(+1)}\subseteq\Theta_{m+x+1}$, and homomorphisms $f^x_{4s+1}:\Theta_m\rightarrow\Theta_{m+x+1}$ extending $\la
u,v\ra\mapsto\la\mu_s(x),\mu_{s+1}(x)\ra$ or $\la v,u\ra\mapsto\la\mu_s(x),\mu_{s+1}(x)\ra$ for $0\le s\le w$.

By taking $w$ as the max of all such $w(x)$ for $x<|\lambda|$ we may assume that the same $w$ works for each $x<|\lambda|$.

Let $t=4(w+1)$ and define $\lambda_0,\ldots,\lambda_t$ as follows. For $0\le s\le w$, let $\lambda_{4s}=\mu_s$, $\lambda_{4s+1} (x) = f^x_{4s+1}(v)$, $\lambda_{4s+2}(x)=f^x_{4s+1}(u)$ and
for $-1\le s\le w$ let $\lambda_{4s+3}=\lambda_{4(s+1)}=\mu_{s+1}$ and let $f^x_{4s+3}$ be the constant $\mu_{s+1}(x)$ map.

Define for $0\le s\le w$,
\begin{equation*} 
	f^x_{4s}=
	\begin{cases}
	 f^x_{4s-1} & \text{if } \lambda_{4s}(x)=\lambda_{4s+1}(x) 
	\\ f^x_{4s+1} & \text{otherwise}
	\end{cases}
\end{equation*}
\begin{equation*}
	f^x_{4s+2}=
	\begin{cases}
	 f^x_{4s+3} & \text{if } \lambda_{4s}(x)=\lambda_{4s+1}(x)
	\\ f^x_{4s+1} & \text{otherwise}
	\end{cases}
\end{equation*}
Let $f_r=\la f^0_r,\ldots,f^{|\lambda_0|-1}_r\ra$, i.e., $(\forall i\in\Theta_m)(f_r(i)=\la f^0_r(i),\ldots,f^{|\lambda_0|-1}_r(i)\ra)$, for $0\le r\le t$.

Then $\la u,v\ra\mapsto\la\lambda_{2k},\lambda_{2k+1}\ra$ and $\la v,u\ra\mapsto\la\lambda_{2k+1},\lambda_{2k+2}\ra$ are extendible for $0\le k<2(w+1)$, via $f_{2k}$ and $f_{2k+1}$, respectively. See Table \ref{t:1}.
\end{proof}

\begin{table}\label{t:1}
\begin{center}
{ $
\begin{array}{|c|c|}
\hline & u \rightarrow \mu_0=\lambda_0
\\ f_0\,\,\,\,\la &
\\ & v \rightarrow\lambda_1
\\ \hline f_1\,\,\,\,\la &
\\ \hline & u \rightarrow \lambda_2
\\ f_2\,\,\,\,\la &
\\ & v \rightarrow\mu_1=\lambda_3
\\ f_3\,\,\,\,\la &
\\ & u \rightarrow \mu_1=\lambda_4
\\ f_4\,\,\,\,\la &
\\ & v \rightarrow\lambda_5
\\ \hline f_5\,\,\,\,\la & \ldots
\\ \hline\ldots & u \rightarrow\lambda_{4w+2}
\\ f_{4w+2}\,\,\,\,\la &
\\ & v \rightarrow\mu_{w+1}=\lambda_{4w+3}
\\ f_{4w+3}\,\,\,\,\la &
\\ & u \rightarrow \mu_{w+1}=\lambda_{4(w+1)}
\\ \hline 
\end{array}
$ } \caption{The construction for the Extendibility interpolation Lemma.}
\end{center}
\end{table}

\subsubsection{Splitting trees}

\begin{lem}[GLB Interpolation Lemma]\label{31}
Let $L$ be a finite lattice and let $a,b,c\in L$ and $\sigma,\tau,\rho\in S(\Theta(L))$ be given such that $a\wedge b=c$, $|\sigma|>0$, $|\tau|=|\rho|$, and $\tau\sim_c\rho$. Then there is a
sequence $\tau=\tau_0,\tau_1\ldots,\tau_m=\rho$ such that for all $p\le m$, $|\tau_p|=|\tau_0|$ and $\sigma*\tau\in S(\Theta(L))$ and $\tau_0\sim_a \tau_1\sim_b\tau_2\sim_a\cdots\sim_b\tau_m$.
\end{lem}
\begin{proof}
We proceed by induction on $|\tau|$. The lemma is trivial for $\tau=\emptyset$. Assume that the lemma holds whenever $|\tau|=s$. Let $a,b,c\in L$ and $\sigma,\tau,\rho\in S(\Theta(L))$ be given satisfying the hypothesis of the lemma with $|\tau|=s+1$. Fix $q,r\in\Theta_s(L)$ such that $\tau=\tau^-*q$ and $\rho=\rho^-*r$. By induction, there are interpolants $\tau^-=\rho_0,\rho_1,\ldots\rho_v=\rho^-$, all of length $s$, such that $\rho\sim_a\rho_1\sim_b\rho_2\sim_a\cdots\sim_b\rho_v$ and $\sigma*\rho_p\in S(\Theta(L))$ for all $p\le v$. Since $\Theta(L)$ is a lattice table, there are $q=q_0,\ldots,q_w=r$ such that for all $p\le w$, $q_p\in\Theta_{s+1}(L)$ and $q_0\sim_a q_1\sim_b q_2\sim_a\cdots \sim_b q_w$. Since $|\sigma|>0$, the sequence $\rho_0*q_0,\rho_1*q_0,\ldots,\rho_v*q_0,\rho_v*q_1,\rho_v*q_2\ldots,\rho_v*q_w$ has the required properties.
\end{proof}

\begin{df}[Sp tree construction]\label{34}

Let $n, s'<\omega$ be given, and let $\{T_{n,t}\mid m<n\and t\ge s'\}$ be an array of trees. Fix $i<\omega$ and $k\in L^i$ and $e,s^*<\omega$ such that $s^*\ge s'$ and fix $\beta\subset T_{n-1}$.
Let $S=S(\Theta^i)$. Let $T_{m}=\Union\{T_{m,t}\mid t\ge s'\}$ for all $m<n$, and let $T_m$ receive the set $S_{m,t+1}$ at stage $t+1$. We construct a tree
\[ 
	T_{n}=\Sp(\{T_{m,t}\mid m<n\and t\ge s'\},e,k,\beta,s^*,\{S_{m,t}\mid m\le n\and t>s^*\})
\]
as the union of an increasing sequence of trees $\{T_{n,t}\mid t\ge s^*\}$ where
\[ 
	T_{n,t}=\Sp_t(\{T_{m,r}\mid m<n\and s'\le r\le t\},e,k,\beta,s^*,\{S_{m,r}\mid m<n\and s^*<r\le t\}).
\]
$T_{n}$ will be a weak $e$-splitting subtree of $T_{n-1}$ for $k$ above $\beta$ whose construction begins at stage $s^*$ and which receives $S_{n,t}$ at stage $t>s^*$.

We will proceed by induction on the set of stages $\{t\mid t\ge s^*-1\}$. Unless either $t$ is in state $\la 0\ra$ or all steps in the construction are completed at the end of stage $t$, we
will set $T_{n,t}=T_{n,t-1}$.

We begin by placing $s^*-1$ in state $\la 0\ra$ and setting $T_{n,s^*-1}=\emptyset$. Fix $t+1\ge s^*$, and assume by induction that $T_{n,t}$ has been defined. We indicate how to define
$T_{n,t+1}$. We begin stage $t+1$ with Step 0, Substep 0 if $t$ is in state $\la 0\ra$. Otherwise, we begin stage $t+1$ with Step 0, Substep 1.

\noindent Special module: If $\alpha^*$ is not in the last plateau of $T_{n-1,t}$ with $\Ht(T_{n-1,t})=$ht$(T_{0,t})$, then transmit $\la\alpha^*,0\ra$ (in the sense of the remarks in the beginning of Definition \ref{213}).

\noindent\emph{Step 0.}

\noindent\emph{Substep 0. Define $T_{n}(\emptyset)$.} Define $T_{n,t+1}(\delta)=\beta$ if $\beta\subset T_{n-1,t}\and\delta=\emptyset\and|\beta|=\Ht T_{0,t}\and t+1=s^*$, and $T_{n,t+1}(\emptyset)\uparrow$ otherwise.

If $T_{n,t+1}=\emptyset$, place $t$ in state $\la 0\ra$ and transmit nothing. Otherwise, proceed to Substep 1.

\noindent\emph{Substep 1. Express preference.} Proceed as in Step 0, Substep 1 of the Phase-2 $L$-tree construction with $k+1=n$.

\noindent\emph{Step 1.} Proceed as in Step 1 of the Phase-2 $L$-tree construction with $k+1=n$.

Then if $t$ is not placed in state $\la 1\ra$, let $\{\xi'_i\mid i\le p\}$ be the set of all strings $\xi$ such that $\alpha^*\subseteq T'(\xi)\downarrow$ and is terminal on $T'$. For each $i\le p$, fix $\xi^+_i$ such that $T_{n-1,t}(\xi^+_i)=T'(\xi'_i)$. Let $\eta^0_i\mid i\le v\}$ be a list of all $\eta\in S$ such that $\eta=\xi^+_i*j$ for some $i\le p$ and $j<f(|\xi'_j|)$. For each $i\le v$, let $\beta^0_i=T_{n-1,t}(\eta^0_i)$. $\{\beta^0_i\mid i\le v\}$ starts the splitting level for $T_{n}$, but must be extended to insure that we have the desired $e$-splittings. We will do this as follows. Following a fixed recursive procedure, form a list of all pairs $\{\la\beta^0_{i_u},\beta^0_{j_u}\ra\mid \beta^0_{i_u}\not\sim_k\beta^0_{j_u}\and u\le q\}$. We will proceed inductively through steps $\{u+2\mid u\le q\}$, building $\{\beta^{u+1}_i=T_{n-1,t}(\eta^{u+1}_i)\mid i\le v\}$ at step $u+2$ such that:
\begin{equation} \label{3(1)} 
	\forall i\le v(\beta^{u+1}_i\supseteq\beta^u_i). 
\end{equation}
\begin{equation} \label{3(2)} 
	\forall i,j\le v(|\beta^{u+1}_i|=|\beta^{u+1}_j|). 
\end{equation}
\begin{equation} \label{3(3)} 
	\la\beta^{u+1}_{i_u},\beta^{u+1}_{j_u}\ra\text{ form an $e$-splitting.} 
\end{equation}
\begin{equation} \label{3(4)} 
	\forall i,j\le v\forall m\le n(\eta^0_i\sim_m\eta^0_j\rightarrow\beta^{u+1}_i\sim_m\beta^{u+1}_j). 
\end{equation} 
Note that conditions \ref{3(2)} and \ref{3(4)} will hold with 0 in place of $u+1$ once we show that $T_{n-1}$ is weakly uniform. We now proceed to the next step.

\noindent\emph{Step $u+2$. Define $\{\beta^{u+1}_i\mid i\le v\}$ satisfying \ref{3(1)}-\ref{3(4)}.}
This step has several substeps. Two interpolations may be needed, so we must always work above level 1 of $T_{n-1}$. So far, we have only guaranteed that level 0 of $T_{n-1}$ has been defined. Thus we begin with Substeps 0 and 1.

\noindent\emph{Substep 0.} Run the special module. Then choose $\hat\beta^u_{i_u}=T_{n-1,t}(\hat\eta^u_{i_u})\supseteq\beta^u_{i_u}$ such that $\hat\beta^u_{i_u}\subset T_{n-1,t}$ and $|\hat\beta^u_{i_u}|=$ht$(T_{0,t})$ and proceed to the next substep.

\noindent\emph{substep 1. Define level 1 of $\Ext(T_{n-1,t},\hat\eta^u_{i_u})$.} If $T_{n-1,t}$ is a type 1 extension of $T_{n-1,t-1}$ for $\hat\beta^u_{i_u}$ and $\Ht(T_{n-1,t})=$ht$(T_{0,t})$, go
to the next substep. In this case, $T_{n-1,t}(\hat\eta^u_{i_u}*0)\downarrow$ and is a potential focal point of $T_{n-1,t}$ which is not a focal point of $T_{n-1,t}$. Otherwise, transmit
$\la\hat\beta^u_{i_u},1\ra$.

\noindent\emph{Substep 2. Find an $e$-splitting of $\beta^u_{i_u}$.} Fix $\hat\eta\in S$ such that $\hat\eta^u_{i_u}=\eta^u_{i_u}*\hat\eta$, and for all $j\le v$, define
$\hat\eta^u_j=\eta^u_j*\hat\eta$. Since $\beta^u_{i_u}\not\sim_k\beta^u_{j_u}$, there is a least $y<|\hat\eta^u_{i_u}|$ such that $\eta^u_{i_u}(y)\not\sim_k\beta^u_{j_u}(y)$. Fix this $y$ and fix
the greatest element $m\in L^i$ for which $\eta^u_{i_u}(y)\sim_k\beta^u_{j_u}(y)$. Let $b=m\wedge k$. Search for an $e$-splitting mod $b$ on $\Ext(T_{n-1,t},\hat\eta^u_{i_u}*0)$.
(We will then interpolate to get an $e$-splitting mod $m$ or an $e$-splitting mod $k$.) If no such $e$-splitting exists, transmit $\la T_{n-1,t}(\hat\eta^u_{i_u}*0),3\ra$.

Suppose that $e$-splittings mod $b$ exist. Let $\la\gamma_0'',\gamma_1''\ra$ be the least $e$-splitting mod $b$ found at stage $t$. Fix $x$ such that
$\{e\}(\gamma_0'';x)\downarrow\ne\{e\}(\gamma_1'';x)\downarrow$ and go to the next substep.

\noindent\emph{Substep 3.} We wish to use the GLB Interpolation Lemma \ref{31} to transform $\la\gamma_0''\gamma_1''\ra$ into an $e$-splitting mod $k$ or an $e$-splitting mod $m$. The
procedure for obtaining such an $e$-splitting involves searching through an Ext tree $T^*$ of $T_{n-1,t}$ for $e$-splittings. In order for the sequence of trees to remain special, we must
define $T^*(\emptyset)$ to be a potential focal point of $T_{n-1,t}$.

Run the special module. Then fix $\gamma_i'=T_{n-1,t}(\zeta_1')\supseteq\gamma_i''$ for $i\le 1$ such that $|\gamma_0'|=|\gamma_1'|=\Ht(T_{0,t})$, $\gamma_0'\sim_b\gamma_1'$, and
$\gamma_1'\subset\Ext(T_{n-1.t},\hat\eta^u_{i_u}*0)$ and proceed to the next substep.

\noindent\emph{Substep 4. Interpolate to get an $e$-splitting mod $m$.} By the GLB Interpolation Lemma \ref{31}, there are arrays $\{\nu^j_i\mid i\le w\and j\le w-1\}$ and $\{\delta^j_i\mid i\le w\and j\le
w-1\}$ satisfying \ref{3(5)}-\ref{3(9)} below for $j=0$.
\begin{equation} \label{3(5)} 
	\forall i,i'<w(|\nu^j_i|=|\nu^j_{i'}|). 
\end{equation}
\begin{equation} \label{3(6)} 
	\zeta_0'=\hat\eta^u_{i_u}*0*\nu^0_0, \zeta_1'=\hat\eta^u_{i_u}*0*\nu^0_w, \forall i\le w(\nu^j_i\supseteq \nu^{j-1}_i), \delta^j_i=T_{n-1,t}(\hat\eta^u_{i_u}*0*\nu^j_i).
\end{equation}
\begin{equation} \label{3(7)} 
	\forall i\le w(0*\nu^j_i\in S).
\end{equation}
\begin{equation} \label{3(8)} 
	\{e\}(\delta^j_j;x)\downarrow. 
\end{equation}
\begin{equation} \label{3(9)} 
	\nu^j_0\sim_m \nu^j_1\sim_k\cdots \nu^j_w. 
\end{equation}
Set $j=1$.
\noindent\emph{Subsubstep $j$.} We construct $\{\nu^j_i\mid i\le w\}$ and $\{\delta^j_i\mid i\le w\}$ satisfying \ref{3(5)}-\ref{3(9)}.

\noindent\emph{Sub$^3$step 0. Define $\delta^j_j$ satisfying \ref{3(8)}.} Institute a search in $\Ext(T_{n-1,t},\hat\eta^u_{i_u}*0*\nu^{j-1}_j)$ for $\nu\in S$ such that $\{e\}(\hat\eta^u_{i_u}*0*\nu^{j-1}_j*\nu;x)\downarrow$.

If no such $\nu$ exists, transmit $\la\delta^{j-1}_j,2\ra$. Otherwise, fix the least such $\nu$. For all $i\le w$, let $\hat \nu^j_i=\nu^{j-1}_i*\nu$. Proceed to Sub$^3$step 1.

\noindent\emph{Sub$^3$step 1.} We will want to define $\delta^{j+1}_{j+1}$ satisfying \ref{3(8)} with $j+1$ in place of $j$. Thus we will search for such a $\delta=\delta^{j+1}_{j+1}$ on an
extension tree $T^*$ of $T_{n-1,t}$, and, failing to find $\delta$, we will use $T^*$ as the next tree in our sequence of trees.

Run the special module. Then fix the least $\nu^*$ such that $\nu^*\supseteq \nu$ and $|T_{n-1,t}(\hat\eta^u_{i_u}*0*\hat \nu^j_j*\nu^*)|=$ht$(T_{0,t})$. For all $i\le w$, let
$\nu^j_i=\hat \nu^j_i*\nu^*$ and $\delta^j_i=T_{n-1,t}(\hat\eta^u_{i_u}*0*\nu^j_i)$. Proceed to the next subsubstep if $j<w-1$.

(Note that it is not necessary to follow this procedure for $j=w$ since $\{e\}(\delta^0_w;x)\downarrow$. At the end of subsubstep $w-1$, we will either have an $e$-splitting mod $k$ or an $e$-splitting mod $m$ on $\Ext(T_{n-1,t},\hat\eta^u_{i_u})$.)

Suppose that $j=w-1$. Then there is a least $i\le w$ such that $\la\delta^{w-1},i,\delta^{w-1}_{i+1}\ra$ $e$-splits on $x$, since $\la\delta^{w-1}_0,\delta^{w-1}_w\ra$ $e$-splits on $x$ and $\{e\}(\delta^{w-1}_j;x)\downarrow$ for all $j\le w$. Fix this $i$. If $i$ is odd, then we have found an $e$-splitting mod $k$ on $\Ext(T_{n-1,t},\hat\eta^u_{i_u})$ and the construction of $T_{n}$ is terminated at this point. If $i$ is even, let $\hat\gamma_0=\delta^{w-1}_i$ and $\hat\gamma_1=\delta^{w-1}_{i+1}$. Note that $\la\hat\gamma_0,\hat\gamma_1\ra$ $e$-splits $T_{n-1,t}(\hat\eta^u_{i_u})$. By the last subsubstep, $|\hat\gamma_0|=|\hat\gamma_1|=\Ht(T_{0,t})$. Proceed to the next substep, defining $\hat\zeta_j\in S$ by $T_{n-1,t}(\hat\zeta_j)=\hat\gamma_j$ for $j\le 1$.

\noindent\emph{Substep 5. Interpolate to get an extendible $e$-splitting mod $m$.} The $e$-splitting mod $m$, $\la\hat\gamma_0\hat\gamma_1\ra$, of $T_{n-1,t}(\hat\eta^u_{i_u})$ will
only be useful if it is appropriately extendible above $\beta^u_{j_u}$. Let $\lambda,\lambda'\in S$ be defined by $\hat\zeta_0=\hat\eta^u_{i_u}*\lambda$ and $\hat\zeta_1=\hat\eta^u_{i_u}*\lambda'$. Since $|\eta^u_{i_u}|>0$, it follows from the Extendibility Interpolation Lemma that there exists $w$ and $\lambda=\lambda_0,\ldots,\lambda_{w+1}=\lambda'$ all of equal length such that $\hat\eta^u_{i_u}*\lambda_n\in S$ and each $\la\lambda_n,\lambda_{n+1}\ra$ is extendible for $\la\hat\eta^u_{i_u},\hat\eta^u_{j_u}\ra$ or for $\la\hat\eta^u_{j_u},\hat\eta^u_{i_u}\ra$. (Note
that $T_{n-1,t}(\hat\eta^u_{i_u}*\lambda_n)\downarrow$.) Define the strings $\sigma_m$ for $0\le m\le w+1$ as follows: alternating $\hat\eta^u_{i_u}$, $\hat\eta^u_{j_u}$ followed by $\lambda_m$. We now force $\{e\}(T_{n,t}(\sigma^{(m)}_m);x)\downarrow$ for each $m$, for some $\sigma^{(m)}_m\supseteq\sigma_m$. Let for $0\le j\le w+1$, $\sigma_j^{(0)}=\sigma_j$.

\noindent\emph{Subsubstep $0<m<w+2$. Get convergent computations.}

\noindent\emph{Sub$^3$step $0$. Extend an interpolant to get a convergent computation on $x$.}
Search for $\sigma\subset\Ext(T_{n-1,t},\sigma_m^{(m-1)})$, $\sigma=T_{n-1,t}(\sigma_m')$, such that $\{e\}(\sigma;x)\downarrow$. If no such $\sigma$ is found, transmit $\la T_{n,t}(\sigma_m^{(m-1)}),2\ra$. Otherwise, fix such a $\sigma$.

\noindent\emph{Sub$^3$step 1.} Run the special module. Then there is $\sigma_m^{(m)}\supseteq\sigma_m'$ at the top of the domain of $T_{n-1,t}$ and there are $\sigma_j^{(m)}$ for $0\le j\le w$ all of the same height, extending $\sigma_j^{(m-1)}$ as $\sigma_m^{(m)}$ extends $\sigma_m^{(m-1)}$. Go to the next subsubstep.

\noindent\emph{Subsubstep $w+2$. Pick the extendible $e$-splitting and define $\{\beta^{u+1}_c\mid c\le v\}$.} By Subsubstep 0, and the Extendibility Interpolation Lemma, we have a sequence of strings, the first and last of which form an $e$-splitting, and each $\sigma$ of which makes $\{e\}^\sigma(x)\downarrow$; hence there now exists $m$ such that $\la T_{n-1,t}(\sigma_m^{(w+1)}),T_{n-1,t}(\sigma_{m+1}^{(w+1)})\ra$ is an extendible $e$-splitting and at the top of $T_{n-1,t}$. Let $f$ be a map witnessing extendibility.

Define $\beta^{u+1}_c=T_{n-1,t}(\eta^{u+1}_c)$, where $\eta^{u+1}_c=\hat\eta^u_c*f(\hat\eta^u_c)$ and go to the next step, i.e., the next pair that needs to be extended to an $e$-splitting.

If $u=q$, place $t$ in state $\la 0,1\ra$, with $T_{n,t}$ having no transmission. Extend $T'$ to $T_{n,t+1}$ as follows: $T_{n,t+1}(\xi'_{a}*b)=\beta^{v+1}_c$ where $\eta^0_c=\xi^+_{a}*b$.
\end{df}

\subsection{Properties of the strategies}

\begin{lem}\label{214*}\label{35*} 
Suppose $T_{k+1}$ is an $L$-tree as constructed in \ref{213}. Then in Phase 1, $T_{k+1}$ is a Phase-1 tree, and in Phase 2, $T_{k+1}$ is a Phase-2 tree as defined in \ref{214}, \ref{35}.
Moreover,
\begin{enumerate}
\item \label{214*ii}\label{214*iii} If $T_{k+1}\ne\emptyset$ then $T_{k+1}$ is a Phase-2 tree and an $L$-tree in the sense of Definition \ref{116}.
\item \label{35*viii} For all $t\ge s^*$, if $T_{k+1,t}=\emptyset$ then $T_{k+1}$ does not transmit $\la\alpha,i\ra$ with $i\ge 2$ at $t$.
For all $t\ge s^*$, $\alpha$ and $i\ge 2$, if $T_{n}$ prefers the same pair at $t$ and $t+1$ and transmits $\la\alpha,i\ra$ at $t$, then $t$ and $t+1$ will be in different states on $T_{n}$
exactly when the following condition holds:
 \begin{enumerate}
 \item \label{35*viiid} $i=4$ and the relevant $\Sigma^0_2$ outcome holds at stage $t$ but not at stage $t+1$. 
 \end{enumerate}
\end{enumerate}
\end{lem}
\begin{proof} 
Immediate from inspection of the $L$-tree construction. 
\end{proof}

\begin{lem}\label{214**}
Suppose $T_{k+1}$ is a Diff tree. Then if $T_{k+1,t}=\emptyset$ then $T_{k+1}$ is a Phase-1 tree at
$t$. Moreover,
\begin{enumerate}
\item \label{214**ii}\label{214**iii} 
	If $T_{k+1}\ne\emptyset$ then $T_{k+1}$ is $\la e,i,j\ra$-differentiating.
\item \label{214**ix} 
	For all $t\ge s^*$, if $T_{k+1,t}=\emptyset$ and $T_{k+1}$ transmits $\la\alpha,i\ra$ with $i\ge 2$ at $t$, then $t$ and $t+1$ are in different states on $T_{k+1}$ exactly when the following condition holds:
	\begin{enumerate}
	\item \label{214**ixb} 
		$i=2$ and $\alpha=T_{k,t-1}(\xi)$ for a specified $\xi$, and for a specified $x\in\omega$, there is a $\sigma\subset\Ext(T_{k,t},\xi)$ such that $\{e\}^{\sigma^{\la j\ra}}(x)\downarrow$.
	\end{enumerate}
\end{enumerate}
\end{lem}
\begin{proof} 
Immediate from inspection of the Diff tree construction. 
\end{proof}

\begin{lem}\label{215}
Fix $e,a,b,k,s,s^*<\omega$ and let
\[ 
	T_{k+1}=\Diff(\{T_{m,t}\mid m\le k\and t\ge s\}, e,a,b,s^*) 
\]
be defined as in 2.13 through the recursive approximation $\{T_{k+1,t}\mid t\ge s^*\}$. For each $t\ge s^*$, assume that $\{T_{m,r}\mid m\le k\and t\ge r\ge s^*\}$ is special. Then if $T_k\ne\emptyset$,
$T_{k+1}=\emptyset$, and no sufficiently large $t$ is in state $\la 0\ra$ on $T_{k+1}$, then there are $\xi$ and $i\le 2$ such that $T_{k+1,r}$ transmits $\la T_k(\xi),i\ra$ for all
sufficiently large $r$ and $\Ext(T_k,\xi)$ is either finite or $\la e,b\ra$-divergent.
\end{lem}
\begin{proof}
We note that there are only finitely many states which $t\ge s^*$ can occupy on $T_{k+1}$.
Furthermore, if $t\ge s^*$ and $T_{k+1,t}=\emptyset$, then $t$ is in some state on $T_{k+1}$, and by \ref{214}\ref{214xiii} and the construction \ref{43}, $T_{k+1,t-1}$ transmits some pair unless
$t$ is in state $\la 0\ra$ on $T_{k+1}$. Hence by \ref{21435}\ref{214x35xi} and \ref{214xi35xii}, there is a stage $r$ and a pair $\la\alpha,i\ra$ with $i\le 2$ such that every stage $t\ge r$ is in
a fixed state and for each such $t$, $T_{k+1,t-1}$ transmits $\la\alpha,i\ra$.

First assume that $i\le 1$. Fix $t\ge r$. If $T_{k,t}\ne T_{k,t-1}$, then by \ref{29}\ref{29iii}, $T_{k,t}$ is a type $i$ extension of $T_{k,t-1}$ for $\alpha$ and $\Ht(T_{k,t})=$ht$(T_{m,t})$ for
all $m\le k$. Hence by \ref{21435}\ref{214ix35viii}, $t$ and $t+1$ are in different states on $T_{k+1}$, contradicting the choice of $r$. Hence $T_{k,t}=T_{k,r}$ for all $t\ge r$. By
\ref{21435}\ref{214iv35iii}, there is a $\xi$ such that $T_k(\xi)=\alpha$. $\Ext(T_k,\xi)$ is now seen to be finite.

Assume that $i\ge 2$. Then $i=2$. By \ref{214**}\ref{214**ixb}, there are $\xi$ and $x<\omega$ such that $\alpha=T_{k,r}(\xi)$, and for all $\sigma\subset\Ext(T_k,\xi,s^*)$, $\{e\}(\sigma^{\la
b\ra};x)\uparrow$. Hence $\Ext(T_k,\xi)$ is $\la e,b\ra$-divergent.
\end{proof}

\begin{lem}\label{35**}
Let $n$, $s'$, $k$, $e$, $s^*\in\omega$, $\beta$, $\{T_{m,t}\mid m<n\and t\ge s'\}$ and $\{S_{m,t}\mid 
m\le n\and t>s^*\}$ be as in the hypothesis of \ref{34}. For all $t\le s^*$, let $\alpha^*(t)$ be
the $\alpha^*$ chosen at stage $t$, Step 1 of \ref{34}. For all $t\ge s^*$, let
\[
	T_{n,t}=\Sp(\{T_{m,r}\mid m<n\and s'\le r\le t\},e,k,\beta,s^*,\{S_{m,r}\mid m\le n\and s^*<r\le t\})
\]
be the tree constructed in \ref{34}. Then $T_{n}$ is a Phase-2 tree. Moreover,

\begin{enumerate}
\item \label{35**ii} If $T_{n}$ has no $e$-splittings mod $k$ then $T_{n}$ is weak $e$-splitting for $k$
\item \label{35**viii} For all $t\ge s^*$, $\alpha$ and $i\ge 2$, if $T_{n}$ prefers the same pair at $t$ and $t+1$ and transmits $\la\alpha,i\ra$ at $t$,
then $t$ and $t+1$ will be in different states on $T_{n}$ exactly when one of the following
conditions holds:
 \begin{enumerate}
 \item \label{35**viiib} $i=2$, $\alpha=T_{n-1,t-1}(\xi)$, and for a specified $x\in\omega$, there is a
 $\delta\subset\Ext(T_{n-1,t},\xi)$ such that $\{e\}^\delta(x)\downarrow$.
 \item \label{35**viiic} $i=3$, $\alpha=T_{n-1,t-1}(\xi)$, $b$ is as in Step $u+2$, Substep 2 where it is decided to transmit
 $\la\alpha,i\ra$, and there is an $e$-splitting mod $b$ on $\Ext(T_{n-1,t},\xi)$.
 \end{enumerate}
\end{enumerate}
\end{lem}
\begin{proof} 
Immediate from inspection of the Sp construction. 
\end{proof}

\begin{lem}\label{36}
Let $n,s',k,e,s^*<\omega$, $\beta$, $\{T_{m,t}\mid m< n\and t\ge s'\}$ and $\{S_{m,t}\mid m\le n\and
t>s^*\}$ be as in the hypothesis of 3.4. For all $t\ge s^*$, let
\[ 
	T_{n,t}=\Sp_t(\{T_{m,r}\mid m<n\and s'\le r\le t\},e,k,\beta,s^*,\{S_{m,r}\mid m\le n\and t\ge r>s^*\})
\] 
and $T_{n}=\Union\{T_{n,t}\mid t\ge s^*\}$. Assume that for all $t\ge s^*$, $\{T_{m,r}\mid m<n\and s'\le r\le t\}$ is special. Also assume that for all sufficiently large $t$, reception of pairs by $T_{n,t}$ satisfies \ref{23}\ref{23i}-\ref{23v}, that $T_{n,t}$ prefers $\la\alpha,i\ra$, and that $T_{n}$ is finite and has no $e$-splittings mod $k$.

Then there are $\lambda$ and $j\le 3$ such that $T_{n,r}$ transmits $\la T_{n-1}(\lambda),j\ra$ for all sufficiently large $r$ and $\Ext(T_{n-1},\lambda)$ is either finite, or $\la e,1[i]\ra$-divergent, or has no $e$-splittings mod $b$ for some $b<k$.
\end{lem}
\begin{proof}
We note that there are only finitely many states which $t$ can occupy on $T_{n}$ if $T_{n,t}$ prefers $\la\alpha,i\ra$ for all sufficiently large $t$. By the hypotheses, every sufficiently
large $t$ occupies some state on $T_{n}$, and transmits some pair. Hence by \ref{21435}\ref{214x35xi} and \ref{214xi35xii}, there is a stage $r$ and a pair $\la\beta,j\ra$
such that every stage $t\ge r$ is in a fixed state and for each such $t$, $T_{n,t-1}$ transmits $\la\beta,j\ra$. The Lemma now follows from \ref{21435}\ref{214ix35viii} and \ref{29}\ref{29iii} in
a similar way to the proof of Lemma 2.15.
\end{proof}

\subsection{The Theorem}
\subsubsection{Priority tree}

The priority tree is a subtree of $\omega^{<\omega}$ under inclusion. The requirements are associated with levels of the priority tree and indexed by $\la e',a[i],b[i]\ra$ (differentiating
tree requirements), $\la e',1[i]\ra$ (splitting tree requirements) and $i$ (as in $L^i$) for $e,i<\omega$ and $a,b\in L$ as in Lemma \ref{115}. The requirements are ordered in a recursive way
which is consistent (a requirement does not follow a requirement whose associated finite lattice it contradicts) and complete (every requirement occurs). The branching rate of the priority tree is
finite and determined by the construction \ref{43}.

With each $L^i$ is associated $\omega$ many requirements $R^i_n$, as in \cite{Lerman:83}*{Ch. XII}. At every 4th level of the priority tree we use an $L$-requirement to go from $L^i$ to $L^{i+1}$.

The pattern is $L$, Sp, Diff, Sp because in \ref{47} it is useful to have every other tree be a Sp tree and hence a tree that immediately defines its root. Hence $T_{4i}$ is the $L$-tree taking us
from $L^{i-1}$ to $L^i$.

The requirement to be satisfied by $T_k$, where $k\in\{4i+1, 4i+2, 4i+3\}$, is chosen from among the requirements $R^j_n$, $0\le j\le i$, $n<\omega$. Namely choose the smallest such $\la j,n\ra$ (in some $\omega$-ordering of $\omega\times\omega$) that is not assigned to any tree among $T_0,\ldots,T_{k-1}$, and that is consistent with $L^i$, i.e., do not choose a Diff requirement for
$a\not\le b$ (for $a,b$ in some $L^j$, $j<i$) if in $L^i$, the representatives of $a,b$ satisfy $a\le b$. Then use the functions Prime to decide how to try to satisfy the chosen requirement, as
in Lemma \ref{115}.

\subsubsection{Construction}
\begin{df}\label{43}
\noindent {\bf The Construction.} \emph{Stage 0:} Designate $T_{\emptyset,0}$ as Init$(\{S_{\emptyset,t}\mid t>0\})$. Let $\alpha_0=\gamma_0=\emptyset$.

\noindent \emph{Stage $s+1$:} For all $\beta$ such that $T_{\beta,s}$ is designated, the information received by $T_{\beta,s}$ at stage $s+1$ is the information transmitted by those
$T_{\delta,s}$ such that $\beta=\delta^-$ at stage $s+1$.

$<^*$ gives the order in which the strategies $\gamma$ on the priority tree act; i.e., the leftmost leaf, the $<^*$-least strategy, acts first.

If a strategy $\gamma*a$ transmits a pair $\la\alpha,i\ra$ with $i\le 1$ then all strategies $\gamma*b$ with $a<b$ and their descendants are immediately cancelled.

$S_{\lambda,t+1}$ is defined to be what $\lambda$ receives from trees (equivalently, from non-cancelled trees) at stage $t+1$.

After all designated, noncancelled trees have acted, designate new trees as follows. (In defining Case $n$ for $1<n\le 3$, we include the assumption that Cases $1,\ldots,n-1$ do not apply.)

\noindent Case 1. There is a trigger at stage $s+1$.

Fix the tree $T_\delta$ of highest priority such that $T_\delta$ is a trigger at stage $s+1$, fix $\beta\subseteq\delta$ such that $T_\delta$ triggers $T_\beta$ at stage $s+1$, and let
$\{\sigma_\gamma,i_\gamma\ra\mid \beta\subseteq\gamma\subset\delta\}$ be the corresponding triggering sequence. Note that there is only one possible choice for $T_\beta$.

If $T_{\beta,s}$ is designated as a Phase-2 tree or if $\beta=\delta$, let $\la\alpha^*,i^*\ra$ be the information transmitted by $T_{\beta,s}$ if any information is transmitted. Note that if $i^*$
is defined, then $i^*\ge 2$.

\begin{itemize}
\item \emph{Subcase 1. New Ext tree. $i^*\in\{2,4\}$.} Let $\alpha_{s+1}=\alpha^*=T_{\beta^-,s}(\xi^*)$.
Let $\gamma_{s+1}=s(\beta)$ and designate $T_{\gamma_{s+1}}$ as the following extension tree:
\[ 
	T_{\gamma_{s+1}}=\Ext(T_{\beta^-},\xi^*,s+1). 
\]
Begin building $T_{\gamma_{s+1}}$.

\item \emph{Subcase 2. New Phase-2 tree. $i^*=3$.} In this case, $T_{\beta,s+1}$ is designated as an $e$-splitting tree for $k$ for some $e<\omega$, $k\in L$, and $s+1$ is in state $\la u+2,0,0\ra$ on $T_\beta$ searching for an $e$-splitting mod $b$ for some $b<k$. Let $\alpha_{s+1}=\alpha^*=T_{\beta^-,s}(\xi^*)$. Let $\gamma_{s+1}=s(\beta)$. If $b=0$, designate
$T_{\gamma_{s+1}}$ as in Subcase 1; and if $b\ne 0$, designate $T_{\gamma_{s+1}}$ as the following $e$-splitting tree for $b$:
\[ 
	T_{\gamma_{s+1}}=\Sp(\{T_{\xi,t}\mid \xi\subset\gamma_{s+1}\and t\ge s+1\}, e,b,\alpha^*,s+1,\{S_{\xi,t}\mid \xi\subseteq\gamma_{s+1}\and t>s+1\}). 
\]
Begin building $T_{\gamma_{s+1}}$.

\item \emph{Subcase 3. Phase-1 tree becomes nonempty.} $\beta=\delta$ and $T_{\delta,s}$ has no transmission.

Let $\alpha_{s+1}=T_{\delta,s+1}(\emptyset)$ and $\gamma_{s+1}=\delta$.

\item \emph{Subcase 4. Phase-2 or Init tree grows.} Otherwise. Then $T_{\beta,s}$ will be designated either as the initial tree or as a splitting or Phase-2 tree, and has no transmission. Let $\alpha_{s+1}=\sigma_\beta$ and $\gamma_{s+1}=\delta$.

\end{itemize}

\noindent \emph{Case 2.} $T_{\gamma_s}$ is nonempty and $T_{\gamma_s*0}$ is not designated.

Let $\gamma_{s+1}=\gamma_s*0$, $\alpha_{s+1}=T_{\gamma_s}(\emptyset)$. Designate a new tree $T_{\gamma_{s+1}}$ as either a Left $L$-tree, a Sp tree or a Diff tree, to satisfy a requirement given by the distribution of strategies on the priority tree (see the subsection on eligible requirements).

\begin{itemize}
\item \emph{Subcase 1. New Phase-2 tree.} Let $\alpha^*=\alpha_{s+1}=T_{\gamma_s,s}(\emptyset)$ and $\gamma_{s+1}=\gamma_s*0$. Set $b=1$ and designate $T_{\gamma_{s+1}}$ as in Case 1, Subcase 2, i.e., as a splitting tree for 1. Begin building $T_{\gamma_{s+1}}$.

\item \emph{Subcase 2. New Phase-1 tree.}

\emph{Subsubcase 1. New Diff tree for Diff requirement $\la e,a,b\ra$ for some $e,a,b$.} Set $\gamma_{s+1}=\gamma_s*0$ and designate $T_{\gamma_{s+1}}$ as the following $\la
e,a,b\ra$-differentiating tree:
\[ 
	T_{\gamma_{s+1}}=\Diff(\{T_{\xi,t}\mid \xi\subseteq\gamma_s\and t\ge s+1\},e,a,b,s+1)
\]
\emph{Subsubcase 2. New Left $L$-tree.} Let $\alpha_{s+1}=T_{\gamma_s}(\emptyset)$, $\gamma_{s+1}=\gamma_s*0$ and designate $T_{\gamma_{s+1}}$ as the appropriate Phase-1 $L$-tree.
\end{itemize}

In all cases, we say that $T_{\beta,s+1}$ is \emph{newly designated} if $T_{\beta,s+1}$ is designated and either $T_{\beta,s}$ is not designated or $T_{\beta,s}$ is cancelled at stage $s+1$.

\noindent {\bf End of construction.}
\end{df}

\subsubsection{Verification}

\begin{lem}\label{l:as}
\begin{enumerate}
\item \label{l:as(1)} For each $s\ge 0$, $\gamma_s$ is the lowest priority tree that is designated at the end of stage $s$.
\item \label{l:as(2)} If $T_{\delta,s+1}\ne T_{\delta,s}$ then $\delta\subseteq\gamma_{s+1}$.
\item \label{l:as(3)} If $\delta<\Gamma$ then $T_\delta$ is finite.
\end{enumerate}

\end{lem}
\begin{proof}
By inspection of Case 1 of the construction \ref{43}, (1) holds. Hence (2) follows from \ref{44}\ref{44vii}. And (3) follows from (2) and the definition of $\Gamma$.
\end{proof}

\begin{thm}\label{410}
Let $L$ be a $\Sigma^0_3$-presentable usl. Then there is a function $g$ of degree $\le\mb 0'$ such that $\mc D(\le\mb g)\cong L$. 
\end{thm}
\begin{proof}
We show that for the function $g=\lim_s\alpha_s$ constructed in \ref{43}, the map $a\mapsto g^{\la a\ra}$ is an isomorphism between $L$ and $[\mb 0,\mb g]$. By \ref{49}, it suffices to verify \ref{115}\ref{115i} and \ref{115}\ref{115ii}. Fix $e<\omega$ and $a,b\in L$ such that $a\not\lesssim b$. Fix $\gamma\subset\Gamma$ such that $T_\gamma$ is designated to handle Diff requirement number $\la e,a,b\ra$, and let $\gamma^*=\gamma^-*0$. Then $T_{\gamma^*}$ is designated as an $\la e,a,b\ra$-Diff tree. If $\gamma=\gamma^*$, then since $|g^*|=\infty$, $T_\gamma(\emptyset)\downarrow$, so by \ref{214**}\ref{214**ii}, $T_\gamma$ is $\la e,a,b\ra$-differentiating. Otherwise, $\gamma=\gamma^-*1$ and $\gamma^-\subset\Gamma$, so since $|g^*|=\infty$, $T_{\gamma^-}(\emptyset)\downarrow$. By the construction, $T_{\gamma^*}(\emptyset)$ will transmit $\la T_\gamma(\emptyset),2\ra$ at all sufficiently large stages (by \ref{45}\ref{45ii} and $T_{\gamma^-*0}=\emptyset$. Since $g\subset T_{\gamma^-}$, $T_{\gamma^-}$ is infinite, and if $T_{\gamma^-}(\eta)=T_\gamma(\emptyset)$, then $\Ext(T_{\gamma^-},\eta)$ is infinite since $T_\gamma(\emptyset)\subset g$. Hence by \ref{214}\ref{214xiii} and \ref{215}, $T_\gamma$ is $\la e,b\ra$-divergent. Thus \ref{115}\ref{115i} holds.

Fix $e<\omega$ and $\gamma\subset\Gamma$ such that $T_\gamma$ is designated to handle Sp requirement number $e$. Let $\gamma=\gamma^-*m$. We verify \ref{115}\ref{115ii} by induction on $k$. We also assume the following induction hypothesis:

\label{4(6)} There is a sequence $1[i]>n_0>n_1>\cdots>n_k$ of elements of $L^i$ such that $T_{\gamma^-*k}$ has no $e$-splittings mod $n_k$ and $T_{\gamma^-*k}$ is designated as an
$e$-splitting tree for $n_k$.

The induction hypothesis is easily verified for $k=0$, as no tree has $e$-splittings mod $1[i]$. By \ref{35**}\ref{35**ii}, $T_{\gamma^-*l}$ is an $e$-splitting tree for $n_k$.

Hence if $T_{\gamma^-*k}$ is infinite, then $k=m$ by Lemma \ref{l:as}\ref{l:as(3)}, and \ref{115}\ref{115ii} will hold. Otherwise, $T_{\gamma^-*k}$ is finite, so $T_{\gamma^-*(k+1)}$ must
be designated (since $g\subset T_\gamma$ implies that $T_\gamma$ is infinite).

When $T_{\gamma^-*(k+1)}$ is first designated after the last time it is cancelled, $T_{\gamma^-*k}$ transmits some $\la\alpha,i\ra$ at with $i\ge 2$ by the construction \ref{43}. Now if $T_{\gamma^-*k}$ ever stopped transmitting $\la\alpha,i\ra$ then $T_{\gamma^-*k}$ would have a state change and so be part of a triggering sequence and so $T_{\gamma^-*(k+1)}$ would be
cancelled, contradiction. So $T_{\gamma^-*k}$ permanently transmit $\la\alpha,i\ra$.

So by \ref{35}\ref{35x}, $T_{\gamma^-*k}$ must prefer some pair. If $T_{\gamma^-}(\eta)=\alpha$, then since $\alpha\subset g\subset T_{\gamma^-}$ (using \ref{45}\ref{45iv}),
$\Ext(T_{\gamma^-},\eta)$ is infinite. Hence by \ref{46}\ref{46iii}, \ref{46}\ref{46v}, \ref{36} and the construction, either $\gamma=\gamma^-*(k+1)$ and $T_\gamma$ is $\la e,1[i]\ra$-divergent, or
there are no $e$-splittings mod $n_{k+1}$ on $T_{\gamma^-*(k+1)}$ for some $n_{k+1}<n_k$. If $n_{k+1}>0$ then \ref{4(6)} holds for $k+1$ since $T_{\gamma^-*(k+1)}$ is designated as an
$e$-splitting tree for $n_{k+1}$. And if $n_{k+1}=0$, then $T_{\gamma^-*(k+1)}$ has no $e$-splittings mod 0, so $T_{\gamma^-*(k+1)}$ is an $e$-splitting tree for 0. Since $L^i$ is
finite, the induction must terminate with $T_\gamma$ satisfying \ref{115}\ref{115ii}.
\end{proof}